\pdfoutput=1

\documentclass[12pt,
headings=small,
]{scrartcl}

 \usepackage[a4paper, left=4cm, right=4cm]{geometry} %
\setkomafont{sectioning}{\rmfamily\bfseries} %

\usepackage{pifont}
\usepackage{times}
\usepackage{amsfonts}       %
\usepackage{amsmath} %
\usepackage{amssymb}        %
\usepackage{amsthm}         %
\usepackage{booktabs}
\usepackage{placeins}
\usepackage{latexsym}
 \usepackage[matrix,curve]{xypic}
\usepackage{microtype}        %
\usepackage{titlefoot}
\usepackage{pdflscape}

\newcommand{\N}{\mathbb{N}}
\newcommand{\Z}{\mathbb{Z}}

\newcommand{\C}{\mathbb{C}}

\newcommand{\CCC}{\mathcal{C}}

\newcommand{\PS}{\mathrm{PS}}

\newtheorem{thm}{Theorem}[section]
\newtheorem{cor}[thm]{Corollary}
\newtheorem{lem}[thm]{Lemma}
\newtheorem{prop}[thm]{Proposition}

\theoremstyle{definition}
\newtheorem{defn}[thm]{Definition}
\newtheorem{example}[thm]{Example}
\newtheorem{rmk}[thm]{Remark}

\newcommand{\diag}{\mathrm{diag}}

\newcommand{\pluss}{+}
\newcommand{\minuss}{-}
\newcommand{\St}{\mathrm{St}}
\newcommand{\Danismanfunctor}{\kappa}
\newcommand{\fin}{\mathrm{fin}}
\newcommand{\norm}{\mathrm{norm}}
\newcommand{\univ}{\mathrm{univ}}
\newcommand{\cell}{\mathrm{cell}}
\newcommand{\crit}{\mathrm{crit}}
\newcommand{\Sp}{\mathrm{Sp}}
\newcommand{\GSp}{\mathrm{GSp}}
\newcommand{\Gl}{\mathrm{Gl}}
\newcommand{\id}{\mathrm{id}}
\newcommand{\Diff}{\;\mathrm{d}}
\newcommand{\im}{\mathrm{im}}
\newcommand\nosf[1]{\begin{footnotesize}\textup{\textsf{#1}}\end{footnotesize}}

\DeclareMathOperator{\Hom}{Hom}
\DeclareMathOperator{\Ext}{Ext}

\DeclareMathOperator{\coker}{coker}
\DeclareMathOperator{\ind}{ind}
\DeclareMathOperator{\Ind}{Ind}

\usepackage{enumitem}
\setlist[enumerate]{itemsep=-0.5ex plus0.1ex minus 0.2ex}
\setlist[description]{itemsep=-0.5ex plus0.1ex minus 0.2ex}
\setlist[itemize]{itemsep=-0.5ex plus0.1ex minus 0.2ex}
\let\svthefootnote\thefootnote
\newcommand\blankfootnote[1]{%
  \let\thefootnote\relax\footnotetext{#1}%
  \let\thefootnote\svthefootnote%
}

\everydisplay{
\thickmuskip=5mu plus 2mu
\medmuskip=4mu plus 1mu
\relax}

\title{\begin{large}Regular poles for spinor $L$-series attached to\\ split Bessel models of $\mathrm{GSp}(4)$\end{large}}
\author{Mirko~R\"osner \and Rainer Weissauer}
\date{}
\begin{document}
\maketitle

\begin{abstract} \noindent
For irreducible smooth representations $\Pi$ of $\mathrm{GSp}(4,k)$ over a non-archimedean local field $k$,
Piatetskii-Shapiro and Soudry have constructed an $L$-factor depending on the choice of a Bessel model.
It factorizes into a regular part and an exceptional part. We determine the regular part for the case of split Bessel models.
\blankfootnote{2010 \textit{Mathematics Subject Classification.} Primary 22E50; Secondary 11F46, 11F70, 20G05.}
\end{abstract}

\bigskip\noindent

\bigskip\noindent
 
\section{Introduction}

For infinite-dimensional irreducible smooth representations $\Pi$ of $G=\mathrm{GSp}(4,k)$ with central character $\omega$, where $k$ is a local non-archimedean field,
and a smooth character $\mu$ of $k^\ast$,
Piatetskii-Shapiro~\cite{PS-L-Factor_GSp4}
constructed local $L$-factors
$$L^{\mathrm{PS}}(s,\Pi,\mu,\Lambda)$$
attached to a choice of a Bessel model $(\Lambda,\psi)$.
To be precise, fix the standard Siegel parabolic subgroup $P=MN$ in $G$ with Levi $M$ and unipotent radical~$N$.
For a non-degenerate linear form $\psi$ of $N$, the connected component $\tilde{T}$ of the stabilizer of $\psi$ in $M$ is isomorphic to the unit group $L^\times$ for a quadratic extension $L/k$.
A Bessel character is a pair $(\Lambda,\psi)$ where $\Lambda$ is a character of $\widetilde{T}$.
The coinvariant space $(\Pi_{\Lambda})_\psi$ with respect to the action of $\widetilde{T}N$ by $(\Lambda,\psi)$ is at most one-dimensional \cite[Thm.\,3.1]{PS-L-Factor_GSp4}, \cite[Thm.\,6.3.2]{Roberts-Schmidt_Bessel}.
If it is non-zero, we say $\Pi$ has a Bessel model. Bessel models have been classified by Roberts and Schmidt~\cite{Roberts-Schmidt_Bessel}. A Bessel model is called anisotropic or split, depending on whether $L$ is a field or not.
The local factor is a product of its regular and its exceptional part
$$L^{\mathrm{PS}}(s,\Pi,\mu, \Lambda)=L_{\mathrm{reg}}^{\mathrm{PS}}(s,\Pi,\mu,\Lambda)L_{\mathrm{ex}}^{\mathrm{PS}}(s,\Pi,\mu,\Lambda)\ .$$

It was expected by Piatetskii-Shapiro and Soudry \cite{Soudry_Piatetski_L_Factors} that the $L$-factor does not depend on the choice of the Bessel model.
They proved this for unitary fully Borel induced $\Pi$ and unitary $\Lambda$ \cite[Thm.\,2.4]{Soudry_Piatetski_L_Factors}.
Novodvorsky made the claim \cite[Thm.\,5]{Novod_L_factors} that for generic $\Pi$ the $L$-factor $L^{\mathrm{PS}}(s,\Pi,\mu,\Lambda)$ coincides with the $L$-factor $L^{\mathrm{Nov}}(s,\Pi,\mu)$, constructed by him in a completely different way using Whittaker models, which of course would imply that $L^{\mathrm{PS}}(s,\Pi,\mu,\Lambda)$ is independent of the choice of the underlying Bessel model.
This expectation holds true in the case of anisotropic Bessel models by results of Dani\c{s}man \cite{Danisman, Danisman_Annals, Danisman2, Danisman3} and the authors \cite{Anisotropic_Exceptional}.

We study the case of split Bessel models. To every split Bessel model~$(\Lambda,\psi)$ we attach the Bessel module $\widetilde{\Pi}=\Pi_{\widetilde{R},\Lambda}$ as a $TS$-module.
Let $M=\nu^{-3/2}\otimes\widetilde{\Pi}/\widetilde{\Pi}^S$ denote the normalized perfect quotient of the Bessel module.
The $L$-factor $L(s,\mu\otimes M)$ divides the regular $L$-factor
$L^{\PS}_{\mathrm{reg}}(s,\Pi,\mu,\Lambda)$ by proposition~\ref{L-reg}. 
Our main result will be the determination of Bessel modules $\widetilde{\Pi}$ for irreducible representations $\Pi$ of $G$ and split Bessel models of $\Pi$, see table~\ref{tab:Bessel_module_split_Bessel_model}.
This in particular determines the $L$-factors $L(s,\mu\otimes M)$.
The \emph{subregular} $L$-factor $$L^\PS_{\mathrm{sreg}}(s,\Pi,\mu,\Lambda)= L_{\mathrm{reg}}^\PS(s,\Pi,\mu,\Lambda)/L(s,\mu\otimes M)$$  is computed by the authors \cite{Subregular};
it was ignored in an earlier version of this paper.
It turns out that the regular factor does not depend on the choice of the split Bessel model, see theorem~\ref{thm:reg_L_factor}.
The list of the regular factors $L^{\mathrm{PS}}_{\mathrm{reg}}(s,\Pi,\mu,\Lambda)$ attached to split Bessel models can be found in table~\ref{tab:regular_poles} in terms of the classification in \cite{Sally_Tadic} and \cite{Roberts-Schmidt}.

The exceptional factor is found in \cite{W_Excep};
it does not depend on the choice of the split Bessel model either.
Summarizing,
$L^{\PS}(s,\Pi,\mu,\Lambda)$ does not depend on the choice of a split Bessel model.
For generic $\Pi$, the $L$-factor $L^{\PS}(s,\Pi,\mu,\Lambda)$
coincides with the $L$-factor of Novodvorsky \cite{Novod_L_factors}, compare \cite{Takloo-Bighash}.
For non-cuspidal $\Pi$, the $L$-factor $L^{\PS}(s,\Pi,1,\Lambda)$ coincides with the $L$-factor of the Langlands parameter attached to $\Pi$ by Roberts and Schmidt~\cite[table~A.8]{Roberts-Schmidt}.
An explicit description of the Bessel model also has applications in the construction of Euler systems for $\GSp(4)$ \cite{LSZ_EulerGSp4}.

Our proof combines elementary methods from homological algebra applied in the context of $P_3$-theory, as already considered in \cite{Bernstein-Zelevinsky76} and
\cite{Roberts-Schmidt}, with results obtained from a detailed study of the action of tori on the Siegel-Jacquet module $J_P(\Pi)$  and
the analysis of extensions that are defined by filtrations on induced representations for which $\Pi$ is an irreducible quotient.

The authors are grateful to David Loeffler for comments and suggestions.

\tableofcontents

\section{Preliminaries} \label{Bessel data}\label{s:preliminaries}
Fix a local non-archimedean field $k$ of characteristic $\mathrm{char}(k)\neq2$ with finite residue field of cardinality $q$.
The group $G\! =\! \GSp(4,k)$ of symplectic similitudes in four variables is
\begin{equation*}
\GSp(4,k) = \{ g\in\Gl(4,k) \vert g' J g = \lambda(g) \cdot  J\} \quad , \qquad
J=\begin{pmatrix} 0 & E_2 \\ - E_2 & 0 \end{pmatrix}
\end{equation*}
with the symplectic similitude factor $\lambda(g)\in k^\times$ and center $Z\cong k^\times$.
Let $P= MN$ denote the standard Siegel parabolic subgroup
with Levi component $M$ and unipotent radical $N$.
We identify elements in $N$ with vectors $(a,b,c) \in k^3$
and elements in $M$ with pairs $(A,\lambda)\in\Gl(2) \times\Gl(1)$ via the embeddings
$$ x_\lambda \cdot m_A \ = \ \begin{pmatrix} \lambda\cdot A  & 0 \cr 0 &  (A')^{-1} \end{pmatrix} \quad , \quad  s_{a,b,c} \ = \ \begin{pmatrix} 1 & 0  & a & b \cr 0 & 1  & b & c \cr
 0 &  0 & 1 & 0 \cr  0 & 0  & 0 & 1  \end{pmatrix} \ .$$ 
We write $\tilde t= \diag(t_1,t_2,t_2,t_1)$, $x_\lambda = \diag(\lambda E, E)$ and $t_\lambda = \diag(E,\lambda E)$ in $M$  for $t_1,t_2,\lambda\in k^\times$
and $s_a= s_{a,0,0}$, $s_b= s_{0,b,0}$ und $s_c= s_{0,0,c}$.  
Notice $x_\lambda s_b x_\lambda^{-1} = s_{\lambda b}$. 
The Weyl group of $G$ has order eight and is generated by \begin{equation*}
\mathbf{s}_1=\begin{pmatrix}0&1&0&0\\1&0&0&0\\0&0&0&1\\0&0&1&0\end{pmatrix}\qquad,\qquad
\mathbf{s}_2=\begin{pmatrix}1&0&0&0\\0&0&0&1\\0&0&1&0\\0&-1&0&0\end{pmatrix}\ .
\end{equation*}

We use the following notation:
\begin{enumerate}
\item ${\omega}$ a fixed smooth character of $Z\cong k^\times$,
\item $\psi:N\to \C$ a nontrivial character with vector group $\tilde N\subseteq\ker\psi\subseteq N $, 
\item $R= \tilde T N\subseteq P$ the connected component of the centralizer of $\psi$ in $P$,  
\item $\tilde T\subseteq M$ the split torus given by the elements $\tilde t$ for $t_1,t_2\in k^\times$,
\item $\tilde R  = \tilde T \tilde N\subseteq R $ the Bessel group,
\item $\delta_P(m_A t_\lambda s_{a,b,c})=\vert\frac{\det A}{\lambda}\vert^3$, the modulus character of $P$, trivial on $R$,
\item $S\subseteq N$ centralizer of $\tilde R$, vector group $S\! =\! \langle s_b, b\in k\rangle$,
\item $N=S\oplus \tilde N$ and $R \! =\! S \times \tilde R$, i.e. both factors commute,
\item $T=\{ x_\lambda\in M\ \vert \  \lambda\in k^\times\}$ with $\delta_P^{1/2}(x_\lambda) = \vert \lambda\vert^{3/2}$,
\item $TS \subseteq MN\! =\! P$ acts on $R$ by conjugation,
\item $\chi, \mu$ smooth characters of $k^\times\cong T$ via $T\ni x_\lambda \mapsto \lambda\in k^\times$,
\item $\rho^\divideontimes = \omega/\rho$ for characters $\rho$ of $k^\times$,
\item $\chi_{\norm}:= \delta_P^{-1/2}\chi = \nu^{-3/2}\chi$ as a character of $T$, 
\item $\Lambda$ smooth character of $\tilde{R}$, trivial on $\tilde N$ with $\Lambda\vert_Z\!=\!{\omega}$, hence 
 $\Lambda(\tilde t) = \rho(t_1)\rho^\divideontimes(t_2)$ for some smooth character $\rho$ of $k^\times$ (Bessel character),
 \item $\rho_+ = \nu^{-1/2}\chi_\Pi$ and $\rho_- = \nu^{1/2}{\omega}\chi_\Pi^{-1} =\rho_+^\divideontimes$. %
\end{enumerate}
The character $(\Lambda,\psi)$ of $R$ defines a \emph{split Bessel datum}
as in \cite{PS-L-Factor_GSp4}.

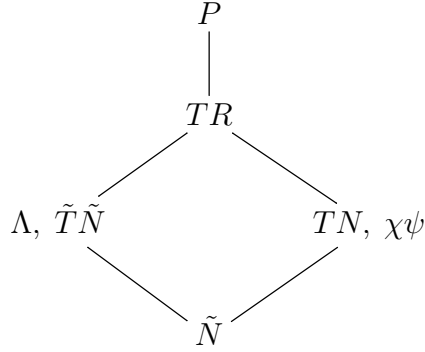
$$ \xymatrix{ & P\ar@{-}[d] & \cr  &     TR\ar@{-}[dl] \ar@{-}[dr]   & \cr
\Lambda,\ \tilde T\tilde N & & TN,\ \chi\psi \cr
& \tilde N \ar@{-}[ul] \ar@{-}[ur] & }$$

Note that there exists an exact sequence $ 0 \to \tilde R \to TR \to  TS  \to 0  \ .$

\subsection{Functors}
For a totally disconnected locally compact group $\Gamma$ let ${\cal C}_{\Gamma}$ be the category of complex vector spaces
with a smooth action of ${\Gamma}$ and let ${\CCC}_{\Gamma}^{\fin}$ be its full subcategory of objects with finite length. Between these categories we consider 
\begin{itemize}
\item[$\Danismanfunctor$] the left exact functor ${\cal C}_{TS} \to {\cal C}_{TS}$ that attaches to $V\in\CCC_{TS}$ its submodule ${\Danismanfunctor}(V)=V^S$ of $S$-invariant vectors.
Notice ${\Danismanfunctor}(V/\Danismanfunctor(V))=0$.

\item[$\eta$] the exact functor ${\cal C}_{G} \to {\cal C}_{\Gl_a(2)}$ that attaches to $V\in\CCC_{G}$ its quotient $\overline{V}=V_{S_A}$ with the action of $\Gl_a(2)$ defined in section~\ref{functor eta}.

\item[$\beta_\rho$] the right exact functor
$\ {\cal C}_{TR} \to {\cal C}_{TS}$ which assigns to $V\in {\cal C}_{TR}$ the maximal quotient space $\widetilde V = V_{\tilde R,\Lambda}$ on which $\tilde R$ acts by $\Lambda$,

\item[$\beta^\rho$] the left exact functor $\ {\cal C}_{TR} \to {\cal C}_{TS}$ which assigns to $V\in {\cal C}_{TR}$ the $\Lambda$-eigenspace $(V_{\tilde N})^{\Lambda}$ of $\tilde T$ in the space of $\tilde N$-coinvariants of $V$,

\item[$k_\rho$] the right exact functor $\CCC_{n}\to\CCC_{n-1}$ defined in section~\ref{the functors beta} with its left derived functor $k^\rho$,

\item[$\mu$] the twist by a smooth character $\mu$ of $k^\times$ is the exact functor
$\CCC_G(\omega)\to\CCC_G(\mu^2\omega)$ given on objects by $\Pi\mapsto \mu\Pi=(\mu\circ\lambda)\otimes\Pi$ with the obvious action on morphisms.
The twist functor $\CCC_{n}\to\CCC_{n}$, $M\mapsto i_*(\mu\circ\det)\otimes M$ is also denoted by $\mu$.
The meaning is always clear from the context.
\end{itemize}
It immediately follows from the definitions that there are natural equivalences
of functors $\CCC_G\to\CCC_{TS}$
$$  \beta_{\mu\rho}\circ \mu  \cong \mu \circ \beta_{\rho} \quad,\quad\beta^{\mu\rho}\circ \mu  \cong \mu\circ \beta^{\rho}\ \ .$$

For $V\in\CCC_{S}$ and a character $\psi$ of $S$
we denote by $V_\psi = V_{S,\psi} = V/V(S,\psi)$ the $\psi$-coinvariants,
i.e.\ the maximal quotient on which $S$ acts by the character $\psi$.
For $V\in {\CCC}_T$ and a smooth character $\chi:T \to \C^\times$, we consider
\begin{enumerate}
\item[] $\chi$-invariants $V^{\chi}=\{v\in V\mid x\cdot v \!=\! \chi(x) v \ \forall t\in T\}$,
\item[] $\chi$-coinvariants $V_{\chi}=V_{T,\chi}=V/V(T,\chi)$,
\item[] generalized $\chi$-invariants
$V^{(\chi)}=\{v\in V\mid \exists n\ \forall x\! \in T:\ (x- \chi(x))^n\cdot v\! =\! 0 \}$.
\end{enumerate}

\section{The category ${\cal C}$ of finitely generated $TS$-modules}

The group $TS$ is generated by the matrices $x_\lambda$ for $\lambda\in k^\times$ and $s_a$ for $a \in k$.
Due to the relation
$x_\lambda s_b x_{\lambda}^{-1} = s_{\lambda b}$ the algebraic group $TS$ is isomorphic 
to the group $\Gl_a(1)$ of affine linear transformations $x\mapsto\lambda \cdot x + b$ of $k$.
We denote them by $\lambda.b$, see also section~\ref{s:Gelfand-Kazhdan}.
Indeed, sending $x_\lambda$ to $\lambda.0 \in\Gl_a(1)$ and
$s_b$ to $1.b\in\Gl_a(1)$ defines an isomorphism $TS\cong\Gl_a(1)$.
This gives an equivalence  between the category ${\cal C}_{TS}$ of smooth $TS$-modules
and the category ${\cal C}_1$ of smooth $\Gl_a(1)$-modules.
This being said, let $\CCC=\CCC^{\fin}_{TS}$ denote the full subcategory
of $\CCC_{TS}$ of modules of finite length.

\subsection{Examples}

All the examples listed below do frequently occur later.
We introduce them here to fix certain conventions and notations in this respect.
Most statements are well known. For the convenience of the reader we give proofs.

\begin{example}\label{example 1}
The spaces $C^\infty_c(k)$, $C_b^\infty(k^\times)$ of smooth complex valued functions $g$ on $k$ resp.\ on $k^\times$ with bounded support in $k$ are $TS$-modules by $(s_b g)(x)\! =\! \psi(bx)g(x)$ and $(x_\lambda g)(x)\! =\! g(x \lambda)$. 
An invariant irreducible $TS$-subspace contained in both $C^\infty_c(k)$ and $C_b^\infty(k^\times)$ is the Schwartz space
$$\mathbb S= C_c^\infty(k^\times)\ .$$
It is invariant under twists by smooth characters $\chi$ of $k^\times$ by the isomorphism 
$i_*(\chi)\otimes\mathbb{S}  \cong {\mathbb S}$, $g\mapsto \chi g$.
\end{example}
\begin{lem}\label{lem:C_b_perfect}
For submodules $M$ of $C^\infty_c(k)$ or $C_b^\infty(k^\times)$ we have $M^\chi=0$ for every $T$-character $\chi$ and ${\Danismanfunctor}(M)=0$,
and the quotient $C^\infty_b(k^\times)/\mathbb{S}$ has a one-dimensional $\chi$-eigenspace for every smooth $T$-character $\chi$.
\end{lem}
\begin{proof}
$M^\chi=0$ is implied by boundedness of the support.
For every non-zero $g\in M$
there is $x_0\!\in\! k^\times$ with $g(x_0)\neq 0$.
Evaluation at $x_0$ is an $S$-linear morphism $M \to (\C,{\psi_{x_0}})$,
so $g$ is not in ${\Danismanfunctor}(M)\subseteq \ker(M\to M_{\psi_{x_0}})$.

For a $\chi$-eigenvector $g+\mathbb{S}$ in $C^\infty_b(k^\times)/\mathbb{S}$ we have $g(x \lambda)=\chi(\lambda)g(x)$ for every $\lambda\in k^\times$ and sufficiently small $x$. Hence there is $C\in\C$ with $g(x)=C \chi(x)$ for small $x$
and this $g$ spans a one-dimensional $\chi$-eigenspace in $C_b^\times(k^\times)/\mathbb{S}$.
\end{proof}
\begin{lem}\label{lem:S_submodule_Cb}
Every non-zero $TS$-submodule $M$ of $C^\infty_b(k^\times)$ contains $\mathbb{S}$ and the action of $S$ on $M/\mathbb{S}$ is trivial.
\end{lem}
\begin{proof}
For every $g\in M$ and every $s_b\in S$, the difference $s_b g-g$ has compact support by smoothness of $\psi$.
If there is $x_0\in k^\times$ with $g(x_0)\neq0$, then there is $b\in K$ with $\psi(bx_0)\neq1$, so $s_b g-g$ is a non-trivial element of $M\cap\mathbb{S}$. The $TS$-action on $s_bg-g$ generates $\mathbb{S}=C_c(k^\times)$ as an irreducible submodule.
\end{proof}

\begin{example}\label{example 3}
Let $\mathbb{E}\! =\! C_c^\infty(S)$ denote the space of smooth complex valued functions with compact support in $S$;
we identify $k$ and $S$ by sending $b$ to $s_b$, so $\lambda s_b = s_{\lambda b}$.
Consider $\mathbb{E}$ as a  $TS$-module via $(s_b f)(s) = f(s+s_b)$ and $(x_\lambda f)(s) = \vert\lambda\vert^{-1} f(\lambda^{-1} s)$.
Integration $f\!\mapsto\! \int_S f(s)\mathrm{d}s$ defines a $TS$-map $\mathbb{E} \to \C$ to the trivial $TS$-module with irreducible kernel $\mathbb{E}^0$
$$ 0 \to \mathbb{E}^0 \to \mathbb{E} \to \C \to 0 \ . $$
The Fourier transform ${\cal F}: \mathbb{E}\to C_c^\infty(k)\,,\ {\cal F}(f)(x)\! =\! \int_S f(s) \psi(-xs) \mathrm{d}s$ gives an isomorphism
with the $TS$-module $C_c^\infty(k)$ in example~\ref{example 1} and $\mathcal{F}(\mathbb{E}^0)=\mathbb{S}$. 
Especially, ${{\Danismanfunctor}}(\mathbb{E})= 0$ and ${{\Danismanfunctor}}(\mathbb{E}^0)= 0$.  
\end{example}

\begin{example}\label{example 4}
Let $V\in {\cal C}_{\Gl(2)}$ be a smooth $\Gl(2)$-module.
Pullback along the embedding $\Gl_a(1)\hookrightarrow\Gl(2)$ by $\lambda.b\mapsto \left(\begin{smallmatrix}\lambda&b\\0&1\end{smallmatrix}\right)$
defines a $\Gl_a(1)$-module structure on $V$ in a natural way.
By our conventions this also defines a $TS$-module structure on $V$,
where $x_\lambda$ acts by $\diag(\lambda,1)$ and $s_b$ acts by
$\left(\begin{smallmatrix}1&b\\0&1\end{smallmatrix}\right)$.
\end{example}

 \begin{lem}\label{lem:almost_S-invariant}
$M^S=\bigcap_{\psi\neq1}\ker(M\to M_{S,\psi})$ for every $M\in \CCC_{TS}$. In other words, $m\in M$ is invariant under $S$ if and only if for every smooth non-trivial character $\psi$ of $S$, the element $m$ is a linear combination of terms $s\cdot w-\psi(s)w$ with $s\in S$ and $w\in M$.
\end{lem}

\begin{proof}
$M':=\bigcap_{\psi\neq1}\ker(M\to M_{S,\psi})$ is clearly an $S$-submodule of $M$.
For every $\psi$ the action of $t\in T$ sends the term $s\cdot w-\psi(s)w$ to
$$t\cdot m=s'\cdot w' -\psi'(s') w'$$
with $w'=tw$, $s'=tst^{-1}$ and $\psi'(s')=\psi(s)$. Hence $m\in M'$ implies $t\cdot m\in M'$ and $M'$ is a $TS$-submodule of $M$.

Now fix a non-trivial $\psi$.
The functor $j^!$ of $(S,\psi)$-coinvariants is exact, so there is a commutative square of vector spaces
\begin{equation*}
 \xymatrix{
 M'     \ar@{^{(}->}[r]\ar@{->>}[d]  &  M\ar@{->>}[d] \\
 j^!(M') \ar@{^{(}->}[r]       &  j^!(M)\ 
 }
\end{equation*}
with surjective vertical and injective horizontal arrows.
The composition morphism $M' \to j^!(M)$  vanishes by construction, so $j^!(M')=0$.
Lemma~\ref{lem:Gelfand-Kazhdan} applied to $M'$ yields
an isomorphism $M'\cong i_*i^*(M')$, hence $M'\subseteq M^S$ is a trivial $S$-module.
The converse assertion $M^S\subseteq M'$ is obvious.
\end{proof}

\begin{lem}\label{lem:kappa_pi_0_injective} For $M\in \CCC_{TS}$, the canonical map $M^S\to M_S$ is injective.
\end{lem}
\begin{proof}
$S$ is compactly generated, so the functor of $S$-coinvariants is exact.
Thus there is a commutative diagram in $\CCC_{TS}$ with surjective vertical arrows and injective horizontal arrows
\begin{equation*}
 \xymatrix@C+3mm{
 M^S\;\ar@{->>}[d]\ar@{^{(}->}[r] &  M        \ar@{->>}[d]\\
 (M^S)_S\;\ar@{^{(}->}[r] & M_S \ .
 }
\end{equation*}
The left vertical arrow is clearly an isomorphism, so the canonical morphism $M^S\to M\to M_S$ is injective.
\end{proof}

\begin{lem}\label{Example 2}
Every irreducible $V\in\mathcal{C}_{TS}$ with $V^S\!\neq\! V$
is isomorphic to $\mathbb{S}$.
\end{lem}
\begin{proof}
$V^S\neq V$ implies $V_\psi\neq 0$ for some $\psi\neq 1$ by lemma~\ref{lem:almost_S-invariant}.
By Frobenius reciprocity there is an embedding $V\hookrightarrow \Ind_S^{TS}(\psi)\subseteq C^\infty(T)$.
Every $f\in \Ind_S^{TS}(\psi)$ is $TS$-smooth, so $s_bf\! =\! f$ holds for $s_b\in S$ with sufficiently small $b\in k$.
Thus $\psi(b\lambda)f(x_{\lambda})\! =\! f(x_{\lambda})$ for these $b$ implies that $f(x_{\lambda})$ vanishes for large values of $\lambda\in k^\times$.
We obtain a $TS$-embedding $\Ind_S^{TS}(\psi)\hookrightarrow C_b^\infty(k^\times)$.
Identify $V$ with its image in $C_b^\infty(k^\times)$, then the statement follows from lemma~\ref{lem:S_submodule_Cb}.
\end{proof}

\subsection{Objects of finite length} %
\label{s:catC}

The category ${\cal C} = \CCC_{TS}^{\fin}$ of smooth $TS$-modules of finite length can be identified with ${\cal C}_{\Gl_a(1)}^{\fin}$, as explained in the last section. %
In this sense the irreducible objects in $\CCC$ are isomorphic to 
either $j_!(\C)$ and one of the objects $i_*(\chi)$, corresponding
to smooth $\Gl(1)$-characters $\chi$, extended to $\Gl_a(1)$ being trivial on the unipotent radical;
see \cite[\S5.13]{Bernstein-Zelevinsky76}.
Here the character $i_*(\chi)$ corresponds to the $TS$-character $$x_\lambda s_b\mapsto \chi(\lambda)\ ,$$ and we often write $\chi\in\CCC$ instead of $i_*(\chi)$.
The infinite-dimensional irreducible module $j_!(\C)$ corresponds to the $TS$-module $\mathbb{S}\subseteq C_b^\infty(k^\times)$ by lemma~\ref{Example 2}.

The degree $\deg(M)$ of $M\in{\cal C}$ is
the number of Jordan-H\"older constituents isomorphic to ${\mathbb S}$. For example, every non-trivial submodule $M\subseteq C_b^{\infty}(k^\times)$ of finite length has degree one by lemma~\ref{lem:S_submodule_Cb}.
For a non-trivial $S$-character~$\psi$, the functor $M\mapsto M_\psi\cong \Hom_{S}(M,{\psi})$ of $(S,\psi)$-coinvariants is exact and corresponds to the functor $j^!:\CCC^{\fin}_1\to \CCC^{\fin}_0$ in section~\ref{s:Gelfand-Kazhdan}.
Since $\chi_\psi \!=\! 0$ and ${\mathbb S}_\psi\! =\! \C$, exactness implies
$$  \deg(M) \ =\ \dim(M_\psi) \ .$$
The functor of $S$-coinvariants $$\pi_0:\, {\cal C} \to {\cal C}^{\fin}_T \,, \ V\mapsto V_S$$ is exact and left adjoint to the inclusion functor $\CCC^{\fin}_T \hookrightarrow \CCC$. By abuse of notation, the canonical morphism $M\to \pi_0(M)$ is also denoted $\pi_0$.
These functors correspond to the functors $i^*, i_*$, introduced in section \ref{s:Gelfand-Kazhdan}
in greater generality. The kernel $M^0:= \ker(M \to \pi_0(M))$ corresponds to $j_!j^!(M)$ in $\CCC_1$ by lemma~\ref{lem:Gelfand-Kazhdan},
and it does not contain one-dimensional Jordan-H\"older constituents.

Notice that $\pi_0(\chi)\! =\! \chi$ and $\pi_0({\mathbb S})\! =\! 0$ holds.
So for $M\in \CCC$ the Jordan-H\"older constituents of $\pi_0(M)\in \CCC_T^{\fin}$ are
the characters $\chi$ occuring in $M$ with multiplicity $a_\chi(M)\in\Z_{\geq0}$.
Using Tate's $L$-factor $L(\chi,s)$, we define the $L$-factor of an object $M\in\CCC$ by
$$L(M,s) = \prod_\chi L(\chi,s)^{a_\chi(M)}\ .$$

\begin{lem}\label{maxfin}\label{FINITE}
For a $TS$-module $M\in\CCC$ of finite length,
the submodule ${{\Danismanfunctor}}(M)$ is the maximal finite-dimensional $TS$-submodule of $M$.
\end{lem}

\begin{proof}
We first show that every finite-dimensional $TS$-submodule $F$ of $M$ is a trivial $S$-module and thus contained in $\Danismanfunctor(M)$.
Indeed, lemma~\ref{lem:Gelfand-Kazhdan} yields an exact sequence $0\to j_!j^!(F)\to F\to i_*i^*(F)\to0$
in $\CCC_{TS}$.
If $j^!(F)$ is non-zero, then $F$ contains
the infinite-dimensional submodule $j_!j^!(F)$, which contradicts the assumption that $F$ is finite-dimensional.
Therefore $F\cong i_*i^*(F)$ is a trivial $S$-module.

It remains to be shown that $\Danismanfunctor(M)$ is itself finite-dimensional.
By lemma~\ref{lem:kappa_pi_0_injective}, it is sufficient to show that $i_*i^*(M)$ is finite-dimensional.
But $i^*(M)$ is a $T$-module of finite length and thus finite-dimensional. The functor $i_*$ preserves the dimension.
\end{proof}

\begin{lem} \label{lem:Ext(S,F)}
$\Ext^1_{{\cal C}}({\mathbb S}^s, X)\! =\! 0$ for finite-dimensional $TS$-modules $X\in \CCC$.
\end{lem}
\begin{proof}
Indeed, for an exact sequence in ${\cal C}$ given as in the upper row of
$$ \xymatrix{ 0 \ar[r] & X \ar@{->>}[d]^{\pi_0}_\cong \ar[r] &  E   \ar@{->>}[d]^{\pi_0}\ar[r] &  {\mathbb S}^s  \ar@{->>}[d]^{\pi_0}\ar[r] & 0 \cr  
0 \ar[r] & X \ar[r] &  E_S  \ar[r] &  0 \ar[r] & 0 \cr}  \ ,$$
the left vertical arrow is an isomorphism by lemma \ref{FINITE} and the exactness of $\pi_0$. The snake lemma provides
a splitting $E(S)  = \ker(E\to\pi_0(E)) \cong {\mathbb S}^s$.
\end{proof}
Of course $\Hom_{{\cal C}}({\mathbb S}^s, X)\! =\! 0$ and $\Hom_{{\cal C}}(X,{\mathbb S}^s)\! =\! 0$ for finite-dimensional $X\in\CCC$.
Furthermore $$\Ext^1_{\cal C}(\mathbb S,{\mathbb S})= 0 \ ,$$
since %
by lemma~\ref{lem:Gelfand-Kazhdan} the full subcategory of ${\cal C}$
of modules annihilated by $i^*$ is equivalent to the category of finite dimensional complex vector spaces via the functors $j^!,\! j_!$.
However, we will see in lemma~\ref{lem:Ext(X,S)} that $\Ext^1_{{\cal C}}(X,{\mathbb S}^s)$ is non-trivial in general.
For $M$ in ${\cal C}$ the exact sequence
$$ 0 \to {\mathbb S}^{\deg(M)} \to M \to \pi_0(M) \to 0 \ ,$$
defining $M^0=\ker(M\to\pi_0(M))$ exhibits $M^0$
as the unique maximal subgroup ${\mathbb S}^{\deg(M)}$ of $M$ generated by $\deg(M)$ copies of ${\mathbb S}$.

\begin{example}\label{ex:cusp_as_TS_module}
A cuspidal irreducible $\Gl(2)$-module $\pi\in \CCC_{\Gl(2)}^{\fin}$, considered as a $TS$-module as in example~\ref{example 4},
is isomorphic to $\mathbb{S}$.
Indeed, $\pi_S=0$ by cuspidality and the degree of $\pi$ is one by the uniqueness of Whittaker models.
\end{example}

\begin{lem}\label{lem:RR}
For every $M\in \mathcal{C}$ and $T$-character $\chi$ we have
$$\dim M_{T,\chi}-\dim M^{T,\chi}=\deg M\ .$$
\end{lem}
\begin{proof}
By lemma~\ref{lem:proto_long_exact_sequence} we have a long exact sequence
\begin{equation*}
0\to(\mathbb{S}^{\deg(M)})^{\chi}\to M^\chi\to\pi_0(M)^\chi\to\mathbb{S}^{\deg(M)}_{\chi}\to M_\chi\to \pi_0(M)_\chi\to0\ .
\end{equation*}
Note that $\dim\pi_0(M)_\chi=\dim\pi_0(M)^\chi$ since $\pi_0(M)$ is finite-dimensional.
We have $\mathbb{S}^{\chi}=0$ by lemma~\ref{lem:C_b_perfect}
and $\dim\mathbb{S}_{\chi}=1$ by proposition~4.3.2 of Bump~\cite{Bump},
so counting dimensions implies the statement.
\end{proof}

\textit{The category $\CCC^{\fin}_T$ of $T$-modules with finite length}.
The irreducible objects in $\CCC^{\fin}_T$ are the smooth characters $\chi: T \to \C^\times$.
Thus every $X\in\CCC^{\fin}_T$ is finite-dimensional and can be decomposed as a direct sum of Jordan blocks in the following sense:
Choose a prime element $\pi\in k^\times$ so that $k^\times=\mathfrak{o}^\times \times \pi^\Z$.
For the isomorphism $k^\times\cong T$ defined by $\lambda\mapsto x_\lambda$
let $x_\pi$ be the image of $\pi$.
The action of the maximal compact subgroup of $T$ is semisimple
and commutes with $x_\pi$,
so the finite dimensional $T$-module $X$ decomposes
as a $T$-module $X \!=\! \bigoplus_\chi X^{(\chi)}$
into its generalized eigenspaces $X^{(\chi)}$
with respect to the smooth characters $\chi$.
On each $X^{(\chi)}$ the monodromy operator
$\tau_\chi =  x_{\pi}\! -\! \chi(\pi)\cdot\id$ is defined and nilpotent.
We write $\tau_\pi$ for $\bigoplus_\chi \tau_\chi$.
Thus $X^{(\chi)}$ decomposes as a direct sum of Jordan blocks under the action of $x_\pi$. Let $\chi^{(m)}$ be the Jordan block of length $m$ attached to the character $\chi$.
We say that $X$ is \emph{cyclic} if every $X^{(\chi)}$ is an indecomposable Jordan block. In that case $\dim X^\chi=\dim X_\chi\leq1$ for every $\chi$.

\subsection{Perfect modules}
A module $M\in\CCC$ is called \emph{perfect} if $M^\chi = 0$ for all smooth characters $\chi$ of $T$. Submodules of perfect modules and extensions of perfect modules are perfect.

\begin{lem} \label{PERF}
For $M\in {\cal C}$ the following assertions are equivalent:
\begin{enumerate}
\item $M$ is perfect.
\item ${{\Danismanfunctor}}(M)=0$.
\item $\dim(M_{\chi})=\deg(M)$ for all smooth $T$-characters $\chi$.
\end{enumerate}
\end{lem}

\begin{proof}
1. $\Longrightarrow$ 2.\ If ${{\Danismanfunctor}}(M)\neq 0$, there exists a character $\chi$
such that ${{\Danismanfunctor}}(M)^\chi \neq 0$, because ${\Danismanfunctor}(M)$ is finite dimensional.
But that implies ${{\Danismanfunctor}}(M)^\chi\subseteq M^\chi\neq 0$.

2. $\Longrightarrow$ 1.\ By lemma \ref{FINITE}, $M^\chi$ is a trivial $S$-module for every $\chi$,
hence a $TS$-submodule of $M$ of finite dimension.
By lemma~\ref{maxfin} ${{\Danismanfunctor}}(M)$ is the maximal finite dimensional $TS$-submodule of $M$,
therefore $M^\chi\subseteq {{\Danismanfunctor}}(M)$.

Equivalence between 1.\ and 3.\ follows by lemma~\ref{lem:RR}.
\end{proof}

\begin{example}\label{ex:S_perfect}
Non-zero finite-dimensional modules $M\in\CCC$ are not perfect.
$\mathbb S\in \CCC$ is perfect of degree one by lemma~\ref{lem:C_b_perfect} and lemma~\ref{lem:S_submodule_Cb}.
\end{example}

\begin{example}\label{ex:mathbb E}
The $TS$-module $\mathbb E = C_c^\infty(S)$ of example~\ref{example 3} admits a non-split exact sequence
\begin{equation*}
0 \to \mathbb S \to \mathbb {E} \to 1 \to  0 \quad , \quad %
\end{equation*} 
and is perfect of degree one by lemma~\ref{lem:C_b_perfect}.
Since $\mathbb{S}\cong\chi\otimes\mathbb{S}$, the twist $\mathbb{E}[\chi]=\chi\otimes\mathbb{E}$ is also perfect of degree one for every $T$-character $\chi$ and one has a non-split exact sequence  %
$$ 0 \to \mathbb S \to \mathbb{E}[\chi] \to \chi \to  0  \ .$$ 
\end{example}

\begin{lem}\label{lem:embedding_criterion}
For a $TS$-module $M\in\CCC$ of degree one and a $T$-character $\chi$
with $\dim \pi_0(M)_\chi\leq1$ and $\Danismanfunctor(M)^{\chi}\neq0$,
there is no embedding $\mathbb{E}[\chi]\hookrightarrow M$ in $\CCC$.
\end{lem}
\begin{proof}
Suppose such an embedding exists.
Since $\Danismanfunctor(M)^{\chi}\neq0$, we obtain another embedding $\chi\hookrightarrow {\Danismanfunctor}(M)\subseteq M$.
The images of both embeddings have zero intersection in $M$ by perfectness of $\mathbb{E}[\chi]$.
This yields an embedding $\mathbb{E}[\chi]\oplus \chi\hookrightarrow M$.
Exactness of $\pi_0$ implies $\dim X^{\chi}\geq2$, thus contradicts the assumption.
\end{proof}

\begin{lem}\label{lem:perfect_model}\label{Kriterium}\label{lem:equivalences2}
For $M\in{\cal C}$ of degree one the following assertions are equivalent:
\begin{enumerate}
\item $M$ is perfect.
\item $M$ admits a model, i.e.\ an embedding into $C_b^\infty(k^\times)$ (see example~\ref{example 1}).
\item for every $T$-character $\chi$ the $T$-module $\pi_0(M)$
is cyclic in the sense of section~\ref{s:catC}
and if $\pi_0(M)_\chi\neq 0$ then
there is an embedding $\chi\otimes\mathbb{E} \hookrightarrow M$.
\end{enumerate}
\end{lem}
\begin{proof}
1.\ $\Longrightarrow$ 2.
The degree of $M$ is one, so $\dim M_\psi=1$ for every non-trivial character $\psi$ of $S$.
Then the argument of lemma~\ref{Example 2} shows the existence of a nontrivial
$TS$-linear map $\ell\!:\! M \to C_b^\infty(k^\times)$.
By lemma~\ref{lem:S_submodule_Cb} the image of $\ell$ has degree one.
Therefore the kernel of $\ell$ is finite-dimensional and thus contained in ${\Danismanfunctor}(M)$ by lemma~\ref{maxfin}. But ${{\Danismanfunctor}}(M)\!=\!0$ vanishes by lemma~\ref{PERF}, so $\ell$ is injective.

2.\ $\Longrightarrow$ 3. By lemma~\ref{lem:C_b_perfect} and exactness of $\pi_0$, as a submodule of $C_b^\infty(k^\times)$ of finite length, $M$ satisfies $\dim(\pi_0(M)_\chi)=\dim(\pi_0(M)^\chi)\leq1$.
If $\dim(\pi_0(M)^\chi)=1$, there is an embedding $\chi\hookrightarrow \pi_0(M)$.
The preimage of $\chi$ under the projection $M\to M/\mathbb{S}\cong\pi_0(M)$ has length two with constituents $\mathbb{S}$ and $\chi$.
By the uniqueness statement of lemma~\ref{lem:extension E_X}, this preimage is isomorphic to $\chi\otimes\mathbb{E}$ as a $TS$-module.

3.\ $\Longrightarrow$ 1. If $M$ is not perfect, there is a character $\chi$ that embeds into  $\Danismanfunctor(M)$ by lemma~\ref{PERF}.
By lemma~\ref{lem:embedding_criterion} there is no embedding $\mathbb{E}[\chi] \hookrightarrow M$, in contradiction to the assumption.
\end{proof}

\begin{example}[Kirillov model]\label{ex:Kirillov}
An infinite-dimensional irreducible $\Gl(2)$-module $V$, 
considered as a $TS$-module (example \ref{example 4}), is perfect.
Indeed, if there is a $\chi$-eigenspace $V^\chi\neq0$ for a $T$-character $\chi$, then this eigenspace is a $\Gl(2)$-submodule by lemma \ref{lem:torus_action_GL2} and this contradicts the irreducibility.
By uniqueness of Whittaker models, $V$ has degree one.
Lemma~\ref{lem:equivalences2} provides an embedding $V\hookrightarrow C_b^\infty(k^\times)$ of $TS$-modules and the image contains $\mathbb{S}=C_c^\infty(k^\times)$. This is the \emph{Kirillov-model} of $V$.
The quotient of $V$ by $\mathbb{S}$ is the finite-dimensional unnormalized Jacquet-quotient $\pi_0(V)=V_S$ as a $T$-module.
\end{example}
\begin{cor}[Waldspurger-Tunnell]\label{cor:Waldspurger-Tunnell}
For generic irreducible $\pi\in\CCC_{\Gl(2)}$
and every smooth character $(\begin{smallmatrix}\lambda & 0 \cr 0 & 1
\end{smallmatrix})\mapsto \rho(\lambda)$ of $\{(\begin{smallmatrix}\lambda & 0 \cr 0 & 1
\end{smallmatrix})\}$ 
we have $\dim\pi_\rho=1$.
\end{cor}
\begin{proof}
Consider $\pi$ as a $TS$-module as in lemma~\ref{example 4}.
Then $\pi^\rho=0$ by lemma~\ref{lem:torus_action_GL2} and $\deg(\pi)=1$ by the uniqueness of Whittaker models (Kirillov-theory).
Lemma~\ref{lem:RR} implies $\dim \pi_\rho=1$.
\end{proof}
\begin{lem}\label{lem:torus_action_GL2}
Let $\pi$ be a smooth $\Gl(2)$-module.
If a split maximal torus of $\Gl(2)$ acts on $v\in \pi$ by a character $\chi$,
then $v$ is an eigenvector of $\Gl(2)$.
Especially, $\chi=\mu\circ\det$ holds for a character $\mu$ of $k^\times$.
\end{lem}
\begin{proof}
By smoothness it is sufficient to show the assertion for a dense subset of $\Gl(2)$.
By conjugation, $v$ is an eigenvector of the standard torus.
For the embedding $TS\hookrightarrow\Gl(2)$ in example~\ref{example 4}, lemma~\ref{FINITE} implies that $v$ is invariant under $S$, so it is an eigenvector of the standard Borel $B$. 
By conjugation, $v$ is also an eigenvector of the opposite Borel $\overline{B}$.
By Bruhat decomposition,
$B\overline{B}$ is dense in $\Gl(2)$ and this implies the statement.
\end{proof}

\begin{lem}\label{lem:extension E_X}
For a cyclic $T$-module $X\in\mathcal{C}_{T}^{\fin}$ there is a unique $TS$-module $M\subseteq C_b^\infty(k^\times)$ of finite length and degree one with $\pi_0(M) \cong X$.
\end{lem}
\begin{proof}
Without loss of generality we can assume that $X=X^{(\chi)}=\chi^{(m)}$ is a Jordan block
 of length $m=\dim X$ attached to a character $\chi$. By a twist we can assume $\chi=1$.
For every $M$ that satisfies the requirements, we have a commutative diagram with exact rows
\begin{equation*}
 \xymatrix{
0\ar[r] & \mathbb{S}\ar[r]\ar@{=}[d] & M\ar[r]\ar@{_{(}->}[d] & X\ar[r]\ar@{_{(}->}[d] & 0\\
0\ar[r] & \mathbb{S}\ar[r] & C_b^\infty(k^\times)\ar[r] & C_b^\infty(k^\times)/\mathbb{S}\ar[r] & 0\ .
 }
\end{equation*}
$X\cong M/\mathbb{S}$ is in the kernel of the monodromy operator $\tau_\chi^{m}$ on $C^\infty_b(k^\times)/\mathbb{S}$. Further, $(C_b^{\infty}(k^\times)/\mathbb{S})^\chi$ is one-dimensional by
lemma~\ref{lem:C_b_perfect}, so by Jordan normal forms $\dim\ker(\tau_\chi^{m}) \leq m$.
By assumption $X$ has dimension $m$, so it coincides with this kernel and therefore $M$ is unique.
It remains to show the existence.
The generalized $\chi$-eigenspace of $C_b^\infty(k^\times)/\mathbb S$ consists of functions
uniquely determined on each $\pi^n \cdot \mathfrak{o}^\times$ up to $a(n)=f(\pi^n)$,
so using $k^\times = \pi^\Z \times \mathfrak{o}^\times$ it can be identified with the subspace of $C(\N,\C)/C_c(\N,\C)$ representated by polynomial functions $a(n)$. Then $X$ corresponds to polynomials of degree $<m$.
Up to a sign, $\tau_\pi=(x_\pi - 1)$ acts on $C(\N,\C)/C_c(\N,\C)$ by the difference operator $a(n) \mapsto a(n)-a(n-1)$ and thus is nilpotent on $X$ of the correct order.
\end{proof}

\textit{Universal extensions}.
Perfect modules $M\in\CCC$ of degree one are uniquely determined by $X=\pi_0(M)$ up to isomorphism by lemma~\ref{lem:equivalences2} and lemma~\ref{lem:extension E_X}.
The same lemmas show that such an $M$ exists if and only if the finite-dimensional $T$-module $X$ is cyclic, i.e. satisfies $$\dim X^\chi=\dim X_\chi\leq1$$ for every $T$-character $\chi$.
The extensions $M$ are universal and we write
\begin{equation*}
0\to \mathbb{S} \to \mathbb{E}[X] \to X \to 0\ .
\end{equation*}
By construction, each isomorphism class $\mathbb{E}[X]$ admits a unique representative in $C_b^\infty(k^\times)$.
For example $\mathbb{E}[0] = \mathbb S$ and $\mathbb{E}[\chi] = \chi\otimes\mathbb{E}$ for $T$-characters $\chi$, see example~\ref{ex:mathbb E}.

\textit{Filtrations}. A perfect non-trivial $M\in \CCC$ has degree at least one. It admits quotients $E$ of degree one.
For example, there always exists a quotient $E$ with degree one and 
$\pi_0(M)\!\cong\! \pi_0(E)$. A quotient $E$ of degree one
and minimal dimension of $\pi_0(E)$ is perfect and the kernel of $M\twoheadrightarrow E$ is also perfect.
Repeating this process gives a filtration of $M$ by perfect modules of degree one.
If $M$ is not perfect, the above construction applied to the perfect module $M/{\Danismanfunctor}(M)$ provides a filtration whose first term is finite-dimensional.

This provides a filtration whose 
graded components $E_i\in \CCC$ are perfect of degree one.
Such a filtration always exists and the quotients $E_i$ are unique up to ordering.
Then $\pi_0(M)^{ss} \ = \ \oplus_{i} \ \pi_0(E_i)^{ss}$
follows from exactness of $\pi_0$.

\subsection{Modules of degree one}\label{s:universal extensions}

\begin{lem}\label{lem:Ext(X,S)} For characters $\chi\in \CCC_T^{\fin}$ we have
$\Ext^1_{\cal C}(\chi,{\mathbb S})=  \C\cdot cl({\mathbb E}[\chi])$.
\end{lem}
\begin{proof}
Fix an exact sequence $0\to \mathbb{S}\to M\to \chi\to 0$ in $\CCC$.
If ${\Danismanfunctor}(M)$ is non-zero, then it is isomorphic to $\chi$ because ${\Danismanfunctor}(\mathbb{S})=0$.
In that case the embedding ${\Danismanfunctor}(M)\to M$ splits the sequence.
Otherwise $M$ is perfect by lemma~\ref{PERF}, so it admits an embedding $M\hookrightarrow C_b^\infty(k^\times)$ by lemma \ref{lem:equivalences2}.
A submodule of $C_b^\infty(k^\times)$ is uniquely determined by $\pi_0(M)\cong X$ because of lemma~\ref{lem:extension E_X}.
\end{proof}

\textit{Fiber products}. For finite-dimensional $T$-modules $X,Y\in \CCC_T^{\fin}$ and a $T$-morphism $f:X\to Y$, assume $Y$ is cyclic in the sense of section \ref{s:catC}.
Then the fiber product
\begin{equation*}
 \xymatrix@+2mm{
 \mathbb{E}[f]\ar[d]_a\ar[r]^b & \mathbb{E}[Y]\ar[d]^{\pi_0}\\
 X\ar[r]^f                     & Y                            
}
\end{equation*}
is given by the $TS$-module $\mathbb{E}[f]=\{(x,m)\in X\times\mathbb{E}[Y]\,|\, f(x)=\pi_0(m)\}$ and the obvious projections $a,b$.
We can (and will) always assume that $f$ is surjective. Indeed, replacing $Y$ by the image of $f$ does not change $\mathbb{E}[f]$. By construction, $\ker(b)\cong \ker(f)$.

The next lemma shows that every $M\in \CCC$ of degree one is uniquely determined, up to isomorphism, by the $T$-modules $\pi_0(M)$ and ${\Danismanfunctor}(M)$.
\begin{lem}\label{lem:universal extensions classification}
Every fiber product $\mathbb{E}[f]$ as above has degree one and satisfies
$${\Danismanfunctor}(\mathbb{E}[f])\cong\ker(f) \quad,  \qquad\pi_0(\mathbb{E}[f])\cong X\ .$$
Conversely, every $M\in \CCC$ of degree one is isomorphic to $\mathbb{E}[f]$
for the natural projection $f:X\to Y$ from $X=\pi_0(M)$ to the quotient $Y=X/{\Danismanfunctor}(M)$.
Especially, $Y$ always satisfies the monodromy bound $\dim Y^{\chi}\leq 1$
for every $T$-character $\chi$.
\end{lem}

\begin{proof}
$\mathbb{E}[f]$ modulo
$\{(0,m)\,|\, m\in \mathbb{S}\}\cong\mathbb{S}$ is finite-dimensional, 
so $\mathbb{E}[f]$ has degree one and $$\pi_0(\mathbb{E}[f])\cong\{(x,y)\in X\times Y\,|\, f(x)=y\}\cong X\ .$$
By lemma~\ref{maxfin} and lemma~\ref{PERF},
the maximal finite-dimensional submodule of $\mathbb{E}[f]$ is the kernel of the projection $b$,
so ${\Danismanfunctor}(\mathbb{E}[f])= \ker (b) \cong \ker(f)$.
For the converse statement,
let $Q$ be a quotient of $M$ of minimal length and degree one. The kernel of the projection 
$$K=\ker(M\twoheadrightarrow Q) \ $$ has degree zero and is thus a finite-dimensional submodule $K\subseteq {\Danismanfunctor}(M)$. We claim that $K={\Danismanfunctor}(M)$. Indeed, otherwise the length of $M/{\Danismanfunctor}(M)$ would be smaller than that of $Q$.
This implies that $Q\cong M/{\Danismanfunctor}(M)$ is perfect of degree one. 
Indeed, if $Q$ was not perfect, then it would admit a non-trivial finite-dimensional submodule which contradicts the minimality assumption.
By exactness of the $\pi_0$-functor, $Q\cong\mathbb{E}[Y]$ is the universal extension attached to the $T$-module $Y = \pi_0(M)/{\Danismanfunctor}(M)$.
By lemma~\ref{lem:equivalences2}, $Y\cong \pi_0(Q)$ satisfies the monodromy bound. For the projection $f$ from $X=\pi_0(M)$ to $Y$ we have a commutative diagram
\begin{equation*}
\xymatrix@+1mm{   M\ar@/^1pc/[rrd]\ar@/_1pc/[rdd]_{\pi_0}\ar@{.>}[rd]^\varphi  & & \cr                 & \mathbb{E}[f]\ar[d]^{\pi_0}\ar[r] & \mathbb{E}[Y]\ar[d]^{\pi_0}\\
                            & X\ar[r]_f                        & \ Y\ .
}
\end{equation*}
By the universal property of the fiber product, we obtain a unique morphism $\varphi:M\to\mathbb{E}[f]$.
Since $\pi_0(\varphi)$ is an isomorphism $X\to X$ and since the kernel of $\varphi$ has degree zero, $\varphi$ is an isomorphism by the five-lemma.
\end{proof}

\begin{lem}\label{lem:hom-lemma}\label{SCREEN}\label{screening}
For a universal extension $\mathbb{E}[X]$ and a module $M\in\CCC$ of degree one,
every morphism
$\varphi:\mathbb{E}[X]\to M$ either
\begin{enumerate}
 \item has finite-dimensional image in ${\Danismanfunctor}(M)$ or
 \item is injective.
 \end{enumerate}
 If $\varphi$ is injective, then $\pi_0(\varphi)$ is injective and
$\coker(\varphi) \cong \coker(\pi_0(\varphi))$ in $\CCC$.
\end{lem}
\begin{proof}
If the image of $\varphi$ is not finite-dimensional, then it has degree one. Then the kernel is a finite-dimensional submodule of $\mathbb{E}[X]$ and therefore zero. This shows the proposed dichotomy.
For injective $\varphi$, the snake lemma applied to
$$ \xymatrix{ 0 \ar[r] & {\mathbb S}  \ar[d]_\cong \ar[r] &  \mathbb{E}[X]   \ar[d]^\varphi\ar[r] & X \ar[d]^{\pi_0(\varphi)}\ar[r] & 0 \cr  
0 \ar[r] & {\mathbb S} \ar[r] &  M  \ar[r] &  \pi_0(M) \ar[r] & 0 \cr}  $$
implies the final statement.
\end{proof}
\begin{lem}
Fix a $T$-character $\chi$ and a surjection $f:X\to Y$ of in $\CCC^\fin_T$ with $\dim Y^\chi\leq1$ for every $T$-character $\chi$.
Morphisms $\varphi:\mathbb{E}[\chi]\to \mathbb{E}[f]$ with finite-dimensional image are in one to one correspondence with morphisms $\chi\to \ker(f)$.
An embedding $\varphi:\mathbb{E}[\chi]\hookrightarrow\mathbb{E}[f]$ can only exist if $Y^\chi\neq0$ and the following exact sequence of smooth $T$-modules (with the obvious morphisms) splits:
\begin{equation*}
0\to \ker(f)\to f^{-1}(Y^\chi)\to Y^\chi\to0\ .
\end{equation*}
\end{lem}
\begin{proof}
Morphisms $\varphi$ with finite-dimensional image factor over $\pi_0(\mathbb{E}[\chi])$.
If $\varphi$ is an embedding, then the composition $$\mathbb{E}[{\chi}]\stackrel{\varphi}{\hookrightarrow} \mathbb{E}[f]\to \mathbb{E}[f]/{\Danismanfunctor}(\mathbb{E}[f])\cong \mathbb{E}[Y]$$ is injective by lemma~\ref{lem:hom-lemma}.
The exact functor $\pi_0$ applied to $\mathbb{E}[\chi]\hookrightarrow \mathbb{E}[Y]$ implies $Y^\chi\neq0$. By assumption, $Y^\chi$ is one-dimensional and we can fix a generator $0\neq y\in Y^\chi$ in the image of $\mathbb{E}[\chi]$.
Note that $x\in f^{-1}(y)$ is unique up to elements in $\ker(f)\cong {\Danismanfunctor}( \mathbb{E}[f])$, but ${\Danismanfunctor}( \mathbb{E}[f])\cap im(\varphi)=\{0\}$. Hence there is a unique $\chi$-eigenvector $x\in f^{-1}(y)$ in the image of $\pi_0\circ\varphi$.
The $T$-morphism $Y^\chi\rightarrow f^{-1}(Y^\chi)\ ,\ \lambda y\mapsto \lambda x$ splits the sequence.
\end{proof}

\begin{lem} \label{Ext-Lemma} 
Let $\mathbb E[X]\in\CCC$ be perfect of degree one.
For inclusions
$i\!:\! P\! \hookrightarrow\! M$ and $\!j\! : \mathbb E[X]\! \hookrightarrow\! M$ 
in $ {\cal C}$ let  $p\!:\! M\to Q\!=\!M/P$ be the quotient.
Then $\varphi\!  =\! p\circ j$ is either injective
and $i$ induces an inclusion $P \!\hookrightarrow\! M/im(j)$, or there exists a morphism 
$\psi: X \!\to\! {{\Danismanfunctor}}(Q)$ and an inclusion
$\mathbb{E}[Y] \hookrightarrow P$ for $Y\!=\!\ker(\psi)$.
\end{lem}
\begin{equation*}
\xymatrix@-1mm{
  &  & \mathbb{E}[X]\ar@{^{(}->}[d]_j\ar[dr]^\varphi & & \\
 0\ar[r] & P\ar[r]_i                            & M\ar[r]_p               & Q\ar[r] & 0
}
\end{equation*}
\begin{proof}
If $K=\ker(\varphi)\neq 0$, then $\deg(K)= 1$ and $im(\varphi)  \subseteq  {{\Danismanfunctor}}(Q)$. Thus $\varphi$ factorizes over $\psi: X\to{\Danismanfunctor}(Q)$.
We obtain $K\cong \mathbb{E}[Y]$ for $Y\!= \ker(\psi)$.
By construction $K$ injects into $i(P)= \ker(p)$.
\end{proof}

\begin{lem}\label{lem:indecomposability}
For $M\in \CCC$ of degree one the following assertions are equivalent:
\begin{enumerate}
 \item $M$ is indecomposable,
 \item Every Jordan block in $\pi_0(M)$ has non-trivial image in $\pi_0(M)/{\Danismanfunctor}(M)$,
 \item $\dim\pi_0(M)^\chi\leq1$ for every character $\chi$ and if $\pi_0(M)^\chi\neq0$ then there is an embedding $\mathbb{E}[\chi] \hookrightarrow M/{\Danismanfunctor}(M)$.
\end{enumerate}
\end{lem}
\begin{proof}
1.\ $\Longrightarrow$ 2.
If there is a Jordan block $X$ of $\pi_0(M)$ contained in the submodule ${\Danismanfunctor}(M)$,
then we obtain decompositions of $T$-modules $\pi_0(M) \cong X\oplus \pi_0(M)/X$ and ${\Danismanfunctor}(M) \cong X\oplus {\Danismanfunctor}(M)/X$ as direct sums.
By the uniqueness statement implicit in lemma~\ref{lem:universal extensions classification}, $M$ is isomorphic to $X\oplus M/X$.

2.\ $\Longrightarrow$ 3.
Fix a character $\chi$ that occurs in $\pi_0(M)$.
If $\dim\pi_0(M)^\chi\geq2$, then there are two $\chi$-Jordan blocks and both have non-trivial image in $\pi_0(M)/{\Danismanfunctor}(M)$. But $\pi_0(M)/{\Danismanfunctor}(M)$ satisfies the monodromy property by lemma~\ref{lem:equivalences2}, this gives a contradiction.
Since $\chi$ occurs in $\pi_0(M)/{\Danismanfunctor}(M)\cong \pi_0(M/{\Danismanfunctor}(M))$, the embedding exists by lemma~\ref{lem:equivalences2}.

3.\ $\Longrightarrow$ 1.
Assume that $M=M_1\oplus M_2$ splits as a direct sum where $M_2$ is finite-dimensional.
Fix a character $\chi$ that occurs in $M_2=\pi_0(M_2)$.
If there is an embedding $\mathbb{E}[\chi]\hookrightarrow M/{\Danismanfunctor}(M)\cong M_1/{\Danismanfunctor}(M_1)$, then $\pi_0(M_1)^\chi\neq0$. Hence $\dim \pi_0(M)^\chi\geq2$ implies a contradiction.
\end{proof}
\begin{cor}
Every perfect module $M\in\CCC$ of degree one is indecomposable.
\end{cor}
\begin{proof}
If $M$ is perfect, then ${\Danismanfunctor}(M)=0$ by lemma~\ref{PERF}.
The second assertion of lemma~\ref{lem:indecomposability} is obviously satisfied.
\end{proof}

\subsection{$L$-functions and zeta integrals}\label{s:L-series}

Recall that $M\in {\CCC}$ is perfect of degree one if and only if $M$ has finite length
and $\dim(\chi\otimes M)_{T}=1$ holds for all smooth characters $\chi$ of $k^\times$.
Then $\dim(M_\psi)=1$ and every non-zero $\ell\in \Hom_\C(M_\psi,\C)$
defines an embedding $p: M\hookrightarrow C_b^\infty(k^\times)$ by $p(m)(\lambda)=\ell(x_\lambda m)$, see lemma~\ref{lem:equivalences2}. Conversely, if $M\in \CCC$ has finite length and is non-zero,
every such embedding $p$ defines a functional $\ell(m)=p(m)(1)$.

\begin{defn}\label{def:L-factor_TS_module}
Suppose a $TS$-submodule $M\subseteq C_b^\infty(k^\times)$ has finite length but is not zero.
For a smooth $\Gl(1)$-character $\chi$ and $f\in M$ define the \emph{zeta integral}
$$ Z(f,\chi,s) = \int_{k^\times} f(x) \chi(x) \vert x \vert^s \ \Diff^\times x \ ,$$
with Haar measure $\Diff^{\times}x$ normalized by $\mathrm{vol}({\mathfrak o}^{\times})=1$.
The decomposition  of the semisimplified $T$-module $\pi_0(M)^{ss}
\cong \bigoplus_\mu a_\mu(M)\cdot \mu$ defines the product of Tate $L$-factors
$$L(M,s)= \prod_\mu L(\mu,s)^{a_\mu(M)}\ .$$
\end{defn}

\begin{lem}\label{entire}
For smooth $\chi$ and $f\in M$ the quotients $\frac{Z(f,\chi,s)}{L(\chi\otimes M,s)}$
are entire.
\end{lem}

\begin{proof}
Replacing $M$ by $\chi\otimes M$, we may assume $\chi=1$.
Indeed, $f\in M$ if and only if $\chi f\in \chi\otimes M$
for $\chi f(x)=\chi(x)f(x)$.
Then integration over $\mathfrak{o}^{\times}$ allows to assume $f\in M^{T(\mathfrak{o})}$, so $f$ is fixed by $T(\mathfrak{o}) \cong \mathfrak{o}^{\times}$ and hence
$f(x)=\sum_n a_n 1_{\pi^n \mathfrak{o}^{\times}}(x)$.
For $f(x)=1_{\pi^n\mathfrak{o}^{\times}}(x)$ observe 
$$Z(f,1,s)= q^{-ns} \ ,$$ where $q$ is the cardinality of the residue field $\mathfrak{o}/\pi\mathfrak{o}$.
For $f$ in $M^{T(\mathfrak{o})}$ there exist polynomials
$P_\mu(t)$ of degree $< a_\mu(M)$ for all $\mu$ that occur in $\pi_0(M)$,
such that $a_n=\sum_\mu \mu(\pi)^n P_\mu(n)$ for $n\geq m_0(f)$ and $a_n=0$ for $n\leq n_0(f)$.
Let the Laurent ring $\Lambda=\C[t,t^{-1}]$ in the indeterminate $t$
act on $C_b^\infty(k^{\times})$ by $t=x_{\pi}^{-1} \in T$.
Using power series expansion $P(t)= \sum_{n\geq m_0} a_n t^n$ and  $t^n \cdot 1_{\mathfrak{o}^{\times}} = 1_{\pi^n \mathfrak{o}^{\times}}$ one defines an embedding $P(t)\mapsto P(t)\cdot 1_{\mathfrak{o}^{\times}}$  of 
$\tilde\Lambda$ in $C_b^\infty(k^{\times})$
\begin{equation*}
\tilde\Lambda:= \C[t,t^{-1},\tfrac{1}{1-\mu_1(\pi)t}, \cdots, \tfrac{1}{1-\mu_r(\pi)t} ] \hookrightarrow \C[[t]][t^{-1}] \hookrightarrow C_b^\infty(k^{\times})\ .
\end{equation*}
The subspace $W=M^{T(\mathfrak{o})}$ of $M$
is finitely generated as a $\Lambda$-submodule of the localization $\tilde\Lambda$ of the Dedekind ring $\Lambda$, hence defines a fractional ideal.
On $\tilde\Lambda$ the zeta integral $Z(f,1,s)$ is obtained from the evaluation $t\mapsto q^{-s}$.
Since  $\frac{d}{ds} P(q^{-s}) = -\log(q) t\frac{d}{dt} P(t)$, the submodule $\prod_\mu (1-\mu(\pi)t)^{a_\mu(M)}W$ is equal to $\Lambda$.
Since $1-\mu(\pi)t$ specializes to $L(\mu,s)^{-1}=  1-\mu(\pi)q^{-s}$, the product $\prod_\mu L(\mu,s)^{-\deg(P_\mu)}Z(f_P,\chi,s)$ is entire.
\end{proof}

\begin{lem} For non-zero $M\subseteq C_b^\infty(k^{\times})$, the functional $I_\chi: M \to \C$\,,  $$I_\chi(f)\ =\ \lim_{s\to 0}
\frac{Z(f,\chi,s)}{L(\chi\otimes M,s)} \quad , \quad f\in M\ , $$ generates the one dimensional space $\Hom_T(\chi\otimes M,\C)$.
\end{lem}

\begin{proof} Observe $Z(x_\lambda f, \chi,s)= \chi^{-1}(\lambda)\vert \lambda \vert^{-s}Z(f,\chi,s)$
and that the limit $s\to 0$ is defined by lemma \ref{entire}. By a suitable twist we can assume $\chi=1$.
If $d=a_\chi(M)-1 \geq 0$, nonvanishing follows from $I_\chi(f_P)=(-1)^{d} d!$ for $P_\chi(t)=t^{d}$.
For $a_\chi(M)=0$, one can use $I_\chi(1_{\mathfrak{o}^{\times}})\neq 0$ and $1_{\mathfrak{o}^{\times}} \in C_c^\infty(k^{\times})\subseteq M$
(lemma \ref{lem:S_submodule_Cb}).
\end{proof}

\begin{rmk}\label{rmk:Remark 1 L-series} The proof of the last lemma shows that $L(\chi\otimes M,s)$ can be characterized as the regularizing $L$-factor for all zeta integrals $Z(f,\chi,s)$, where $f$ runs over the functions in $M \subseteq C_b^{\infty}(k^{\times})$.
\end{rmk}

\begin{rmk}\label{rmk:Remark 2 L-series} Fixing $0\neq \ell\in M_\psi$ for a perfect module $M$, this defines a model $p: M \hookrightarrow C_b^\infty(k^{\times}) $, and this model defines $p_\chi(m)= I_\chi(p(m))$.
Conversely, if $\ell(m) = p(m)(1)\neq 0$, then the normalized value $p_\chi(m)/\ell(m)$ does not depend on the choice of $\ell$.
\end{rmk}

\begin{lem} For every nontrivial $TS$-submodule $M$ of $C_b^\infty(k^{\times})$ with finite length,
the vector space $\Hom_T(\chi\otimes M, 1^{(n)})$ has dimension $n$.
The exact sequence $0\to 1^{(n-1)} \to 1^{(n)} \to \C \to 0$
of $T$-modules induces an injection $\Hom_T(\chi\otimes M, 1^{(n-1)}) \hookrightarrow \Hom_T(\chi\otimes M, 1^{(n)})$.
Let $(\chi\otimes M)^{(n-1)} \subseteq \chi\otimes M$
be the $T$-submodule annihilated by the functionals in $\Hom_T(\chi\otimes M, 1^{(n-1)})$. 
Then $\Hom_T((\chi\otimes M)^{(n-1)}, \C)$ is spanned as a complex vector space by the functional $I_\chi^{(n)}: \chi\otimes M \to \C$ that is defined as
$$ I_\chi^{(n)}(f)= \frac{d^{n-1}}{ds^{n-1}}\Bigl(\frac{Z(f,\chi,s)}{L(M\otimes\chi,s)}\Bigr)\vert_{s=0}\ .$$
\end{lem}

\begin{proof} Without restriction of generality $\chi=1$.
The subspace $W=M^{T(\mathfrak{o})} \subseteq M$ in $\tilde\Lambda$ is a fractional $\Lambda$-ideal such that $\prod_\mu (1-t)^{a_\mu(M)} =\Lambda$.
Notice  $(t-1) = -t \tau $ holds for the monodromy operator $\tau$ and  $-t\in \Lambda^\times$ is a unit.
Then $\tau^n=0$  on $1^{(n)}$ implies
$$   \Hom_{T}(M, 1^{(n)}) = \Hom_\Lambda(W, 1^{(n)}) =
\Hom_\Lambda(W/(t-1)^nW, 1^{(n)})\ .$$
Since $\varphi: W/(t-1)^nW \cong \Lambda/(t-1)^n$,
using the identification by the isomorphism $\varphi(P(t)) = \prod_\mu (1-\mu(\pi) t)^{a_\mu(M)}P(t)$ we obtain
$$\Hom_\Lambda(W/(t-1)^nW, 1^{(n)}) \cong \Hom_\Lambda(\Lambda/(t-1)^n, 1^{(n)})  = \Hom_{\Lambda}(1^{(n)},1^{(n)}) \ ,$$ which is $\Lambda/(t-1)^n \cong 1^{(n)}$ as vector space, hence
$ \dim(\Hom_{T}(M, \chi^{(n)})) = n$.
On  $\prod_\mu (1-\mu(\pi) t)^{a_\mu(M)} W =\Lambda$ the zeta integral $Z(f,1,s)$
is the evaluation $t\mapsto q^{-s}$, and  $\frac{d}{ds} P(q^{-s}) = -\log(q) t\frac{d}{dt} P(t)$ holds. %
Since the higher derivations $D^i$ with respect to $D=t\frac{d}{dt}$ at $t=1$ up to order $i=n-1$  are linear independent on $\Lambda /(t-1)^n$, the lemma now easily follows.
\end{proof}

\subsection{Bessel functionals} \label{Bessel models}

For $I\in {\cal C}_G^{\fin}({\omega})$, a \emph{Bessel functional} attached to
the split Bessel datum $(\Lambda,\psi)$ with values in a $\C$-vector space $X$
is a $\C$-linear map $$\ell: I \to X$$ such that $\ell(I(\tilde r s) w)= \Lambda(\tilde r)\psi(s) \ell(w)$
 for all $w\in I$, all $\widetilde{r}\in \widetilde{R}$ and $s\in S$; i.e.
$$  \ell \in \Hom_{S \widetilde R}( I, (\psi\boxtimes \Lambda) \otimes_\C X) \ .$$
\textit{Universal Bessel functionals}. %
$  \Hom_{S\widetilde R}( I, (\psi\otimes\Lambda) \otimes X)$ can be identified with $\Hom_{S}(\beta_\rho(I), \psi\otimes X) = \Hom_{\C}((\widetilde I)_\psi, X) $
for $\rho$ defined by $\Lambda(\tilde t)= \rho(\frac{t_1}{t_2}) \omega(t_2)$
and the $TS$-module $\widetilde I =\beta_\rho(I)$. Hence
$X_{\univ}(I)=(\widetilde I)_\psi$ defines the universal Bessel functional
in the sense that every $\C$-valued Bessel functional on $I$ is obtained
from the universal $X_{\univ}$-valued Bessel functional $\ell_{\univ}:I \to X_{\univ}$ by composition with
a $\C$-linear form $X_{\univ} \to \C$.

\bigskip\noindent
\textit{Bessel models}. A nonvanishing Bessel functional $\ell$ on $I$
defines the Bessel model  $ {\cal K}_\ell(I)$ of $I$, i.e.
a surjective $\C$-linear map
$$ p_\ell: I\twoheadrightarrow {\cal K}_\ell(I) \ \subseteq \ C_b^\infty(k^{\times},X) $$
defined by $w \mapsto p_\ell(w)\! =\! \ell(I(x_\lambda) w)$,
the latter viewed as $X$-valued function $\varphi_w(\lambda)$ of the variable $\lambda\in k^{\times}$.
Obviously $p_\ell$ factorizes  over the quotient map $I \to \widetilde I$ and defines  a $TS$-linear
map %
$$\tilde p_\ell:   \widetilde I \longrightarrow {\cal K}_\ell(I) \ .$$
Notice $\tilde p_{\univ}(w)(1)=0$ if and only if $w\in \ker(\widetilde I \to \widetilde I_\psi)$. Further,
$\tilde p_{\univ}(w)(\lambda)=0$
if and only if $ w \in \ker(\widetilde I \to \widetilde I_{\psi_{\lambda}})$ for $\psi_\lambda(s_b)=\psi(s_{\lambda b})$. Since $\psi_\lambda, \lambda\in k^{\times}$ runs over all nontrivial characters of $S$, lemma~\ref{lem:almost_S-invariant} implies that the kernel of the universal map
$\tilde p_{\univ}: \widetilde I \to {\cal K}_{\univ}(I)$ for the universal Bessel functional
$\ell_{\univ}: I\to X_{\univ}={\widetilde I}_\psi$
is ${{\Danismanfunctor}}(\widetilde I)$. (See also \cite[4.1,\,4.2,\,4.7]{Danisman}).
Hence we obtain the next two lemmas

\begin{lem} \label{exactomega}
There is an exact sequence of functors ${\cal C}^{\fin}_G(\omega)\to {\cal C}$
$$ \xymatrix@+3mm{ 0 \ar[r] & {{\Danismanfunctor}}\circ\beta_\rho \ar[r] & \beta_\rho  \ar[r]^-{\tilde p_{\univ}} &  {\cal K}_{\univ} \ar[r] & 0} \ .$$
\end{lem}

\begin{lem} \label{0.4} Suppose $\dim(\widetilde I)_\psi \! =\! 1$. Then
the $TS$-module  ${\cal K}_\univ(I)$ contains $C_c^\infty(k^{\times})$ as $TS$-submodule of finite codimension.
\end{lem}

Indeed $\dim(\widetilde I)_\psi \! =\! 1$, or equivalently $\deg(\widetilde I)=1$,
implies that the quotient ${\cal K}_{\univ}(I)$ in ${\cal C}$ is perfect. 
Further, ${\cal K}_{\univ}(I)$ must contain $C_c^\infty(k^{\times})\cong \mathbb S$
by lemma \ref{lem:S_submodule_Cb}.
This implies (see also Dani\c{s}man in \cite{Danisman}, lemma 4.1):
\begin{lem}\label{lem:cal_L}
For $I\in {\cal C}^{\fin}_G(\omega)$ with $\deg(\widetilde I)=1$
the following holds:
\begin{enumerate}
\item  $ {\cal L}_\Lambda(I) :=  \pi_0({\cal K}_{\univ}(I))\cong{\cal K}_{\univ}(I)/C_c^\infty(k^{\times})$ in ${\cal C}_T^{\fin}$.
\item  $ 0 \to {\Danismanfunctor}({\widetilde I}) \to \pi_0(\widetilde I) \to {\cal L}_\Lambda(I) \to 0$ is exact in ${\cal C}_T^{\fin}$.
\end{enumerate}
\end{lem}

\begin{rmk}\label{rmk:universal functional}
Let $X$ be an $R$-module on which $R=\widetilde{R}S$ acts by the character
$(\Lambda,\psi)$. For $M$ in ${\cal C}_{TR}$, by definition 
$ \Hom_R(M,X) = \Hom_S(\widetilde M, X) = \Hom(\widetilde M_\psi,X)$.
For $X=\widetilde M_\psi$ and $id\in End(\widetilde M_\psi)$ this defines a universal
functional $\ell_{\univ}$ in $ \Hom_R(M,\widetilde M_\psi)$ which 
uniquely corresponds to a functional $\tilde\ell_{\univ} \in \Hom_S(\widetilde M, \widetilde M_\psi)$.
\end{rmk}

\subsection{Regular Poles of Piatetskii-Shapiro $L$-functions}\label{Poles}

For infinite-dimensional irreducible representations $\Pi\in {\cal C}_G(\omega)$ of $G=\GSp(4)$ and a smooth character $\mu$ of $\Gl(1)$,
Piatetskii-Shapiro and Soudry \cite{PS-L-Factor_GSp4}, \cite{Soudry_Piatetski_L_Factors} have defined
local $L$-factors $L^{\mathrm{PS}}(s,\Pi,\mu,\Lambda)$, attached to a choice of a Bessel model for $\Pi$.
For split Bessel models, this $L$-factor can be described as follows:
For each $v\in\Pi$ let $W_v(g):=\ell(\Pi(g)v)$ with the Bessel functional $\ell$.
Attached to $v\in\Pi$ and Schwartz-Bruhat functions $\Phi\in \mathcal{S}(V)$ for $V=k^4$ is the zeta-integral
\begin{equation*}
 Z^{\PS}(s,v,\Phi,\mu,\Lambda)=\int_{\widetilde{N}\backslash H}W_v(h)\Phi((0,0,1,1)h)\mu(\lambda_G(h))
 |\lambda_G(h)|^{s+\tfrac12}\Diff h\ ,
\end{equation*}
where $H$ is the subgroup generated by $T\widetilde{T}$, $\widetilde{N}$ and $J=\left(\begin{smallmatrix}0&E\\-E&0\end{smallmatrix}\right)$,
\begin{equation*}
 H=\left\{\begin{pmatrix}*&0&*&0\\0&*&0&*\\ *&0&*&0\\0&*&0&*\end{pmatrix}\right\}\cap G\ .
\end{equation*}
The zeta-integral converges for sufficiently large $\Re(s)>0$ and
admits a unique meromorphic continuation to $\C$, also denoted $Z^{\PS}(s,v,\Lambda,\Phi,\mu)$.
The $L$-factor $$L^{\PS}(s,\Pi,\mu,\Lambda)$$ is then the regularization $L$-factor of $Z^{\PS}(s,v,\Phi,\mu,\Lambda)$ varying over all $v\in\Pi$ and $\Phi\in\mathcal{S}(V)$.
It decomposes as a product
$$  L^{\mathrm{PS}}(s,\Pi,\mu,\Lambda) = L^{\mathrm{PS}}_{\mathrm{reg}}(s,\Pi,\mu,\Lambda) L^{\mathrm{PS}}_{\mathrm{ex}}(s,\Pi,\mu,\Lambda) \ ,$$
where the regular part $L^{\mathrm{PS}}_{\mathrm{reg}}(s,\Pi,\mu,\Lambda)$ is the regularizing $L$-factor of the zeta functions subject to the condition $\Phi(0,0,0,0)=0$, see \cite{PS-L-Factor_GSp4}. The right-most term is the exceptional factor described in \cite{W_Excep}.
In the previous section we defined the perfect $TS$-module
$${\cal K}_{\ell}(\Pi) \cong \widetilde{\Pi}/\widetilde{\Pi}^S\ .$$
The next propositions relates the $L$-factor of $\nu^{-3/2}\mu\otimes\mathcal{K}_\ell(\Pi)$ with the regular factor $L_{\mathrm{reg}}^{\mathrm{PS}}(s,\Pi,\mu,\Lambda)$.
In the analogous case of anisotropic Bessel models, these $L$-factors coincide as shown by Dani\c{s}man
\cite[prop.~2.5]{Danisman}.
In the situation of split Bessel models however, these factors do not coincide in general. This is due to the presence of so-called subregular poles; for details see \cite{Subregular}.

\begin{prop} \label{L-reg}
Fix a smooth character $\mu$ of $\Gl(1,k)$ and an irreducible $\Pi\in {\cal C}_G(\omega)$ with a non-zero Bessel functional
$\ell:\Pi\to(\Lambda,\psi)$ for a smooth Bessel character $\Lambda=\rho\boxtimes\rho^\divideontimes$ of $\widetilde{T}$.
Then the $L$-factor $L(s,\nu^{-3/2}\mu\otimes {\cal K}_{\ell}(\Pi))$ equals the regularizing $L$-factor of the zeta functions $Z^\PS(s,v,\Lambda,\Phi,\mu)$ varying over every $v\in\Pi$ and $\Phi\in\mathcal{S}(V)$ subject to the conditions $\Phi(\ast,0,\ast,0)=0$ and $\Phi(0,\ast,0,\ast)=0$.
\end{prop}
\begin{proof}
Fix a zeta-function $Z^\PS(s,v,\mu,\Phi,\Lambda)$ where $\Phi(x_1,y_1,x_2,y_2)$ vanishes on the hyperplanes $x_1=y_1=0$ and $x_2=y_2=0$.
Without loss of generality we can assume that
$\mu=1$, otherwise replace $\Pi$ with $\mu\otimes\Pi$. Further, we can assume
 $\Phi(x_1,x_2,y_1,y_2)=\Phi_1(x_1,y_1)\Phi_2(x_2,y_2)$ for certain $\Phi_i\in \mathcal{S}(k^2\setminus\{0\})$, because these $\Phi$ span the tensor product $$\mathcal{S}(k^2\setminus\{0\}\times k^2\setminus \{0\})\cong\mathcal{S}(k^2\setminus\{0\})\otimes \mathcal{S}(k^2\setminus\{0\})\ .$$
By Iwasawa decomposition, every $h\in H$ is a product
$h=\widetilde{n}x_\lambda \widetilde{t}(k_1,k_2)\in \widetilde{N} T \widetilde{T} H(\mathfrak{o}_k)$,  where $H(\mathfrak{o}_k) = H\cap \GSp(4,\mathfrak{o}_k)$. The element $(k_1,k_2)\in H(\mathfrak{o}_k)$ is the image of certain $k_1,k_2\in \Gl(2,\mathfrak{o}_k)$ under the standard embedding $(\Gl(2)\times_{\Gl(1)} \Gl(2))\hookrightarrow H$.
Up to a constant volume factor, the zeta-integral is
\begin{align*}
 Z^\PS(s,v,\mu,\Phi,\Lambda)
 = \int_{\widetilde{N}\backslash H}W_v(h) \Phi((0,0,1,1)h)
 |\lambda_G(h)|^{s+\tfrac12} \Diff h =\\
  Z^{\PS}_{\mathrm{reg}}(s,W_{v}) \int_{H(\mathfrak{o}_k)}\int_{\widetilde{T}}  \Phi_1^{k_1}((0,t_2))\Phi_2^{k_2}((0,t_1)) \Lambda(\widetilde{t})|t_1t_2|^{s+\tfrac{1}{2}}\Diff \widetilde{t}\Diff (k_1,k_2)\ ,
\end{align*}
with the regular zeta-integral $Z^{\PS}_{\mathrm{reg}}(s,W_{v})=\int_{T}W_v(x_\lambda)\vert \lambda\vert^{s-3/2}\Diff x_\lambda$.
The action of $H(\mathfrak{o}_k)$ replaces $\Phi_i$ by $\Phi^{k_i}_i(g)=\Phi_i(gk_i)$ and thus permutes $\mathcal{S}(k^2\setminus\{0\})$.
For each $k_i\in\Gl(2,\mathfrak{o}_k)$, the function $\Phi^{k_i}_i(0,\ast)$ has compact support in $k^\times$,
so the integral over $\widetilde{T}$ is a finite sum.
By smoothness of $\Phi_i$, the integral over the compact group $H(\mathfrak{o}_k)$ is also a finite sum.
For suitable choices of $\Phi_i$ one can arrange that the integral over $H(\mathfrak{o}_k)\widetilde{T}$ is non-zero.
Hence for each fixed $v\in\Pi$, the poles of $Z^\PS(s,v,\mu,\Phi,\Lambda)$ varying over $\Phi$ as above are exactly the poles of the regular zeta integral $Z^{\PS}_{\mathrm{reg}}(s,W_{v})$.
By definition, $L(s-\tfrac{3}{2},\mathcal{K}_\ell(\Pi))$ is the regularizing $L$-factor of the regular zeta integrals $Z^{\PS}_{\mathrm{reg}}(s,W_{v})$ varying over $v\in\Pi$.
\end{proof}

\begin{thm}\label{thm:reg_L_factor}
 For every irreducible $\Pi\in\CCC_G(\omega)$ with split Bessel model $(\Lambda,\psi)$ and every smooth character $\mu$ of $\Gl(1,k)$, the regular $L$-factor $L_\mathrm{reg}^\PS(s,\Pi,\mu,\Lambda)$ is independent of the choice of the split Bessel model and coincides with the product of Tate $L$-factors given in table \ref{tab:regular_poles}.
\end{thm}
\begin{proof}
Proposition~\ref{L-reg} implies that $L(s,\nu^{-3/2}\mu\otimes \mathcal{K}_\ell(\Pi))$ divides the regular factor $L^\PS_{\mathrm{reg}}(s,\Pi,\mu,\Lambda)$
and thus yields a product decomposition
\begin{equation*}
L_{\mathrm{reg}}^{\mathrm{PS}}(s,\Pi,\mu,\Lambda) 
= L(s-\tfrac{3}{2},\mu\otimes \mathcal{K}_\ell(\Pi))L^\PS_{\mathrm{sreg}}(s,\Pi,\mu,\Lambda) \ ,
\end{equation*}
where the right-most term is the subregular factor described in \cite{Subregular}.

In section~\ref{s:SiegelInd} we determine the $S$-coinvariants $\pi_0(\widetilde\Pi)=\widetilde{\Pi}_S$ of the Bessel module, see theorem~\ref{thm:monodromy} and lemma~\ref{ugly}.
In section~\ref{c:Main} we find the $S$-invariants $\kappa(\widetilde{\Pi})=\widetilde{\Pi}^S$, see theorem~\ref{thm:sigma_tilde_Pi}.
By lemma~\ref{lem:universal extensions classification} and thm.~\ref{Mult1} this determines the Bessel module uniquely, see 
tables~\ref{tab:Bessel_module_split_Bessel_model} and \ref{tab:Bessel_module_no_split_Bessel_model}.
By lemma~\ref{lem:cal_L} this determines
\begin{equation*}L(s,\mu\otimes {\cal K}_\ell(\Pi))
=\frac{L(s,\mu\otimes \widetilde{\Pi})}{L(s,\mu\otimes \widetilde{\Pi}^S)}
\ .\qedhere
\end{equation*}
\end{proof}

\textbf{Strategy of proof.}
There is a simple but important consequence of the transitivity property for coinvariant functors:
\begin{lem}\label{lem:Bessel_Jacquet}
For $I\in\CCC_G^{\fin}(\omega)$ and the unnormalized Siegel-Jacquet module $J_P(I)=I_N$ we have a natural equivalence
$\pi_0(\widetilde I) \cong J_P(I)_{\widetilde{T},\rho}$.
\end{lem}
We obtain the following commutative diagram in $\CCC_T$, functorial in $I\in\CCC_G(\omega)^{\fin}$,
where the arrows are the obvious natural transformations.
The upper left and the lower right diagonal arrows correspond to exact coinvariant functors. The functors corresponding to the upper right and lower left arrows are only right-exact, but they are exact for anisotropic Bessel models.
\begin{equation*}  \xymatrix@-3mm{
& J_P(I)  \ar@{->>}[dr]^{\ \ \ \textit{Tunnell}}  &  &    \cr
\ I \ \ \ar@{->>}[rd]_{\textit{Bessel}\ \ \ }  \ar@{->>}[ru]^{\textit{Jacquet}\ \ \ } & &    \bigl(\pi_0(\widetilde{I})\ , {{\Danismanfunctor}}(\widetilde I)\bigr)\    \ar@{->>}[r]& \   \pi_0(\widetilde{I})/{{\Danismanfunctor}}(\widetilde I)=\mathcal{L}_\Lambda(I) \ .  \cr
& \ \  \bigl(\widetilde I \ ,   {{\Danismanfunctor}}(\widetilde I)\bigr) \ \  \ar@{->>}[ur]_{\ \ \ \ \ \pi_0\ \ \ }   & &
}
\end{equation*}

The object on top of the diagram is the unnormalized Siegel-Jacquet module $J_P(I)=I_N$. 
The $T$-action on $J_P(I)$ commutes with the action of $\Gl(2)\cong M^1$,
but $T$ does not act by a central character.
Endowed with the trivial action of $S$, we view $J_P(I)$ as a $T\widetilde{T}S$-module.
The object on the bottom is the $TS$-module $\widetilde I=\beta_\rho(I)$ with its maximal finite-dimensional submodule ${\Danismanfunctor}(\widetilde{I})$.
The diagram commutes by lemma \ref{lem:Bessel_Jacquet}.
In the quotient $\pi_0(\widetilde I)\cong J_P(I)_{\widetilde{T},\rho}$ there is inherited from $\widetilde I$
the $T$-submodule ${{\Danismanfunctor}}(\widetilde I)\subseteq \pi_0(\widetilde I)$ by lemma~\ref{lem:kappa_pi_0_injective}.
This object is not visible in the Jacquet module $J_P(I)$.
\pagebreak
\section{Representations of affine linear groups} \label{s:Gelfand-Kazhdan}

We now review certain results from \cite{Bernstein-Zelevinsky76}, \S5.1-\S5.15.
For a finite dimensional vector space $V$ over $k$ let 
$\Gl_a(V) =\Gl(V).V$ be the corresponding affine linear group, i.e. the semi-direct product of the general linear group $\Gl(V)$
with the group of translations by $V$.
Elements of $\Gl_a(V)$ are denoted $(g.v)$ for $g\in\Gl(V)$ and $v\in V$; and the group law is $(g_1.v_1) (g_2.v_2)=(g_1g_2.v_1+g_1v_2)$.
There are adjoint  exact functors
$$    \xymatrix@+3mm{ {\cal C}_{\Gl_a(V)}\   \ar@<0,5ex>[r]^{i^*}  &             \  {\cal C}_{\Gl(V)}  \ar@<0,5ex>[l]^{i_*} }  $$
such that $  \Hom_{\Gl_a(V)}(X,i_*Y) \! =\! \Hom_{\Gl(V)}(i^*X,Y) $.
They are defined by
\begin{description} 
\item  $i^*(M) \!=\! M_{V}$, \ coinvariants under the translation group $V\subseteq\Gl_a(V)$,
\item  $i_* \! =\!  $\ \ pullback under the quotient map $\Gl_a(V) \to\Gl(V),\ (g.v)\mapsto g$.  
\end{description}
Fix a nontrivial $k$-linear form $\ell\!:\! V\! \to\! k$
and a nontrivial smooth character $\psi\!:\! k\! \to\! \C$.
The kernel $V'=\ker(\ell)$ is a subspace of codimension one. 
This defines a nontrivial character $\psi_V\!: \!V\! \to\! \C$, by $\psi_V\! =\! \psi\!\circ\! \ell$.
The group $\Gl_a(V)$ acts on its normal subgroup $V$ by conjugation.
The little group, the stabilizer in $\Gl_a(V)$ of the character $\psi_V$, is the subgroup $\widetilde{\Gl}_a(V').V$ of $\Gl_a(V)\! =\!\Gl(V).V$.
Here the \emph{mirabolic} subgroup $\widetilde{\Gl}_a(V')$ is the image of the natural embedding
$$\Gl_a(V')\ni (g.v)\mapsto
\begin{pmatrix}g&v\\0&1\end{pmatrix}\in\Gl(V)$$
with respect to the decomposition $V\cong V'\oplus k$.
In the following we relax notation and usually write $\Gl(V)$ instead of $\Gl(V).0$ and $V$ instead of $0.V$ and for elements $g.v$ instead of $(g.v)$.
There are two exact adjoint functors
$$    \xymatrix@+3mm{ {\cal C}_{\Gl_a(V)}\   \ar@<0,5ex>[r]^{j^!}  &             \  {\cal C}_{\widetilde{\Gl}_a(V')}  \ar@<0,5ex>[l]^{j_!} }  $$
such that  $  \Hom_{\Gl_a(V)}(j_!X,Y)\! =\! \Hom_{\widetilde{\Gl}_a(V')}(X,j^!Y) $.
They are defined by
\begin{description} 
\item  $j^!(M) =  M_{V,\psi_V}$ for the $(V,\psi_V)$-coinvariants of the translations $V\!\subseteq\!\Gl_a(V)$,
\item  $j_!(M) = \ind_{\widetilde{\Gl}_a(V').V}^{\Gl_a(V)}(M \boxtimes \psi_V)$ for unnormalized compact induction.
\end{description}
Bernstein and Zelevinski \cite[\S5.12]{Bernstein-Zelevinsky76} have shown:
\begin{lem}\label{lem:Gelfand-Kazhdan}
There is a functorial exact sequence
\begin{equation*}
\xymatrix{
0 \ar[r] & j_!j^!\ar[r] &  \id\ar[r]  &  i_*i^*\ar[r] & 0 \ .
}
\end{equation*}
The functors $j_!$ and $i_*$ induce equivalences 
${\cal C}_{\widetilde{\Gl}_a(V')} \!\cong\! j_! ({\cal C}_{\widetilde{\Gl}_a(V')})$ and ${\cal C}_{\Gl(V)} \!\cong\! i_* ({\cal C}_{\Gl(V)})$
of abelian categories.
It holds that $j^! i_*\! =\! 0$ and $i^*j_!\! =\! 0$, so
these abelian subcategories are closed under extensions in $\CCC_{\Gl_a(V)}$.
\end{lem}
For $V = 0$ we identify ${\cal C}_{\Gl_a(V)}$ with the category of complex vector spaces. Then ${\mathbb S}_V\!=\! (j_!)^n(\C)$ for $n=\dim V$ is defined for a choice of nontrivial characters
$\psi_V$, $\psi_{V'}, \cdots$ that together define a generic character $\psi_{gen}$ of the unipotent radical $U$ of strict upper triangular matrices in group $\Gl(V)$. It is easy to see that
$$ {\mathbb S}_V \ \cong\ \ind_{U.V}^{\Gl_a(n)}(\psi_{gen}\boxtimes \psi_V) \ $$
is irreducible. In particular $j_!({\mathbb S}_{V'})\!=\!{\mathbb S}_V$ and $j^!({\mathbb S}_V) \!=\! {\mathbb S}_{V'}$. %
The functors above preserve irreducible objects, hence objects of finite length.

Put ${\cal C}_n$ for ${\cal C}_{\Gl_a(V)}$ if $n\!=\!\dim(V)$ and write $\mathbb{S}_n$ for $\mathbb{S}_V$.
Let ${\cal C}_n^{\fin} \subseteq {\cal C}_n$ be the full subcategory of objects of finite length.
Then ${\cal C}_1$ is equivalent to the category ${\cal C}_{TS}$ of smooth $TS$-modules,
and we often write ${\mathbb S}$ for ${\mathbb S}_1$. %
It is easy to see inductively that every irreducible object of ${\cal C}_n$ is of the form
$(j_!)^k i_*(M)$ for some irreducible $M\in {\cal C}_{\Gl(n-k)}$.
The irreducible objects of ${\cal C}_1$ for example are $\mathbb S= j_!(\C)$ for $\C\in {\cal C}_0$ and $i_*(\chi)$ for smooth characters $\chi:Gl(1)\to \C^{\times}$.

\textit{Twist}. Characters of $\Gl_a(V)$ factorize over $\det_a:\ g.v\mapsto \det(g)\in\Gl(1)$.
For a smooth character $\chi:\Gl(1)\to \C^{\times}$, the twist functor
$$\CCC_{n}\to\CCC_{n}\ ,\quad M\mapsto\chi\otimes M = (\chi\circ\det\nolimits_a)\otimes M
$$
is denoted by $\chi$ again, if no confusion is possible.
The natural equivalences
\begin{gather*} i_*\circ\chi \cong \chi \circ i_* \qquad, \qquad
   i^*\circ\chi \cong \chi \circ i^* \ ,\\
   j^!\circ\chi \cong \chi\circ j^! \qquad, \qquad 
   j_!\circ\chi \cong \chi\circ j_!  \ ,
 \end{gather*}
are obvious except for the last one which follows
by partial induction and the formula $ \ind_{\{1\}}^{k^{\times}}(\chi) \cong \chi\!\otimes\! \ind_{\{1\}}^{k^{\times}}(1) \cong \ind_{\{1\}}^{k^{\times}}(1)$.

\subsection{The functors $k_\rho$ and $k^\rho$} \label{Computation of betarho}\label{the functors beta}

Fix a finite-dimensional $k$-vector space $V$ and a functional $\ell:V\to k$ defining a subspace $V'=\ker\ell$ of codimension one as before.
We write elements of $\Gl_a(V)$ as block matrices with respect to a fixed decomposition $V\cong V'\oplus k$. 
We identify $H=Gl_a(V')\times\Gl(1)$ with its image under the embedding
\begin{equation*}
H \hookrightarrow\Gl_a(V)\qquad (g.v', x)\mapsto \begin{pmatrix}g & 0\\0&x \end{pmatrix}.\begin{pmatrix}v'\\0\end{pmatrix} \ .
\end{equation*}
$H$ normalizes the subgroup 
 $\widetilde{V'}=
\{ \left(\begin{smallmatrix}\id &\ast\\0&1\end{smallmatrix}\right)\}\subseteq\Gl(V)$.
For every $M\in \CCC_{\Gl_a(V)}$
the $\widetilde{V'}$-coinvariant quotient
is a smooth $H$-module $M_{\widetilde{V'}}\in\CCC_H$.
The corresponding functor $\CCC_{\Gl_a(V)}\to \CCC_H$
is exact because $\widetilde{V'}$ is compactly generated.
A smooth character $\rho$ of $\Gl(1)$ defines a character of $\{\id\}\times\Gl(1)\subseteq H$,
explicitly given by 
$(\id.0 , x)
\mapsto \rho(x)$.
By abuse of notation we denote this character by $\rho$ again.
The $\rho$-(co)invariant spaces
\begin{equation*}
k_\rho(M) =  (M_{\widetilde{V}'})_{\rho}\qquad , \qquad
k^\rho(M) =  (M_{\widetilde{V}'})^{\rho}
\end{equation*}
are preserved by the subgroup $\Gl_a(V')\times\{1\}$ of $H$
and thus define smooth modules $k_\rho(M)$ and $k^\rho(M)$ of $\Gl_a(V')$.
We obtain right exact functors $k_\rho$ and left exact functors $k^\rho$
$$    \xymatrix@+3mm{ {\cal C}_{\Gl_a(V)}\   \ar@<0,5ex>[r]^{k^\rho} \ar@<-0,5ex>[r]_{k_\rho} &             \  {\cal C}_{{\Gl}_a(V')} }   \ . $$
For twists by characters $\chi$ of $k^\times$ there are natural equivalences
$$ k_\rho\circ \chi \cong \chi\circ k_{\rho\chi^{-1}}\qquad, \qquad
   k^\rho\circ \chi \cong \chi\circ k^{\rho\chi^{-1}} \ . $$
Lemma~\ref{lem:proto_long_exact_sequence} implies
\begin{lem}\label{lem:long exact sequence_affine}
A short exact sequence $0\to E\to F\to G\to 0$ in $\CCC_{\Gl_a(V)}$
gives rise to a long exact sequence in $\CCC_{\Gl_a(V')}$,
\begin{equation*}
0\to k^\rho(E)\to k^\rho(F)\to k^\rho(G)\to k_\rho(E) \to k_\rho(F)\to k_\rho(G) \to 0 \ .
\end{equation*}
\end{lem}

We say an object $M\in {\cal C}_n$ for $n\geq1$ is \emph{perfect} if it has finite length and  
$k^\rho(M)\!=\! 0$ holds for all characters $\rho$.
The functors $j_!, j^!$ send perfect objects to perfect objects by lemma~\ref{lem:proto_small_extension_lemma}.
However, objects need not remain perfect after applying the functor
$k_\rho$, see theorem \ref{thm:ME_split}.

\begin{lem}\label{3.1} %
For every smooth $\rho$ there is a natural equivalence
 $$k_\rho\circ i_*  \quad\cong\quad
 k^\rho\circ i_*\ $$
 of functors from $\CCC_{\Gl(V)}^{\fin}$ to the full subcategory $i_*(\CCC^\fin_{\Gl(V')})$ of $\CCC^\fin_{\Gl_a(V')}$.
\end{lem}
\begin{proof}For a smooth $\Gl(V)$-module $\pi$
the isomorphisms $k_\rho(i_*(\pi)) \cong i_*((\pi_{\widetilde{V}'})_\rho)$ and $k^\rho(i_*(\pi))\cong i_*((\pi_{\widetilde{V}'})^\rho)$ are obviously functorial. It remains to be shown that $(\pi_{\widetilde{V}'})_\rho$ and $(\pi_{\widetilde{V}'})^\rho$ are naturally equivalent as $\Gl(V')$-modules.
The unnormalized Jacquet quotient $M=\pi_{\widetilde{V}'}$ is a module of finite length under the action of $\Gl(V')\times\Gl(1)\cong H\cap\Gl(V)$.
Therefore $M$ is a finite direct sum of products $X\boxtimes \chi^{(m)}$ where $X$ is an indecomposable $\Gl(V')$-module and $\chi^{(m)}$ is a $\Gl(1)$-Jordan block of length $m$.
The $\rho$-(co)invariant space in $M$ is then the direct sum of those $X$ where $\chi=\rho$,
so $M^\rho$ and $M_\rho$ are isomorphic.
\end{proof}

\begin{lem}\label{3.3}\label{lem:proto_small_extension_lemma}
For smooth $\rho$
there is a functorial exact sequence in  $\CCC_n$, $n\geq1$,
\begin{equation*}
0\longrightarrow j_!\circ k_{\nu\rho} \longrightarrow k_\rho\circ j_! \longrightarrow i_*\circ i^* \longrightarrow 0
\end{equation*}
and a natural equivalence $ j_! \circ k^{\nu\rho} \cong k^\rho \circ j_!  \ $.
Here the functors $k_{\nu\rho}$ and $k^{\nu\rho}$ are defined with respect to the twisted character $\nu\rho$.
\end{lem}

\begin{proof}
Fix a non-trivial additive character $\psi$ on $V=k^{n+1}$ and a subspace $V'\subseteq V$ of codimension one in the kernel of $\psi$.
As before we write elements of $\Gl_a(V)$ as block matrices with respect to a fixed decomposition $V\cong V'\oplus k$.
Fix subgroups of $\Gl_a(V)$
\begin{equation*}P=M\ltimes V \qquad,\qquad Q=H\ltimes \widetilde{V'}\end{equation*}
 with the mirabolic group $M=\widetilde{\Gl}_a(V')$ and $H\cong\Gl_a(V')\times\Gl(1)$ as above.
Bruhat decomposition implies that the double coset space $P \setminus \Gl_a(V)/Q$ is generated by the identity and a Weyl reflection $w\in \Gl(V)$ that does not preserve $\Gl(V')$. Note that $V'=H\cap wVw^{-1}=H\cap V$ and that $\psi^w=\psi\circ w^{-1}$ is non-trivial on $V'$.
A dimension estimate implies that $PwQ$ is open in $\Gl_a(V)$.
Theorem~5.2 of Bernstein and Zelevinski~\cite{Bernstein-Zelevinsky77} gives a short exact sequence of normalized functors $\CCC_{M}\to \CCC_H$
\begin{equation*}
 \xymatrix{
 0\ar[r] & i_{V',\psi^w} \circ w \circ r_{M\cap w^{-1}\widetilde{V'}w} \ar[r] & r_{\widetilde{V'}} \circ i_{V,\psi}\ar[r] & i_{V',\psi} \circ r_{M\cap \widetilde{V'}} \ar[r] & 0\ .
 }
\end{equation*}
These functors are explicitly given by
\begin{align*}
&i_{V,\psi} = j_!\circ\nu^{1/2}                                            \quad ,&       &\CCC_M\to\CCC_{\Gl_a(V)}\ ,\\
&r_{\widetilde{V'}}= (\nu\boxtimes\nu^{n})^{-1/2}\circ (-)_{\widetilde{V'}}\quad ,&       &\CCC_{\Gl_a(V)}\to\CCC_{H}\ ,\\
&i_{V',\psi^w} \circ w = (j_!\boxtimes \id)\circ (\nu\boxtimes1)^{1/2}            \quad ,&       &\CCC_{\widetilde{H}}\to\CCC_{H}\ ,\\
&r_{M\cap w^{-1}\widetilde{V'}w} = (\nu \boxtimes \nu^{1-n})^{-1/2}\circ (-)_{M\cap w^{-1}\widetilde{V'}w}\quad ,& &\CCC_{M}\to\CCC_{\widetilde{H}}\ ,\\
&i_{V',\psi} = (i_*\boxtimes \mathcal{S})\circ \nu^{1/2}                   \quad ,&       &\CCC_{\Gl(V')}\to\CCC_{H}\ ,\\
&r_{M\cap \widetilde{V'}} = \nu^{-1/2}\circ i^\ast                         \quad ,&       &\CCC_{M}\to\CCC_{\Gl(V')}\ .
\end{align*}
with $\widetilde{H}=M\cap w^{-1}H\cong\Gl_a(V'\cap wV')\times\Gl(1)$ and the regular $\Gl(1)$-module $\mathcal{S}=\ind_{\{\id\}}^{\Gl(1)}(1)\cong \mathbb{S}|_{\Gl(1)}$.
Twist the above sequence by the $\Gl(1)$-character $\nu^{-(n-1)/2}$. Note that $\mathcal{S}$ is invariant under twists. This yields a short exact sequence of functors $\CCC_{M}\to\CCC_{H}$
\begin{equation*}
\xymatrix{
0 \ar[r] & (j_!\boxtimes \nu^{-1})\circ (-)_{M\cap w^{-1}\widetilde{V'}w} \ar[r] & (-)_{\widetilde{V'}}\circ  j_! \ar[r] & (i_*\circ i^*)\boxtimes\mathcal{S} \ar[r] & 0\ .
}
\end{equation*}
Finally, lemma~\ref{lem:proto_long_exact_sequence} gives a long exact sequence of functors $\CCC_M\to \CCC_{\Gl_a(V')}$
\begin{equation*}
\xymatrix@-1mm{
0\ar[r]& j_!k^{\nu\rho} \ar[r] & k^\rho j_! \ar[r]& i_*i^* \otimes \mathcal{S}^\rho \ar[r]^\delta &
   j_!k_{\nu\rho}  \ar[r] & k_\rho j_! \ar[r]& i_*i^* \otimes \mathcal{S}_\rho \ar[r] & 0\ .
}
\end{equation*}
Since $\mathcal{S}^\rho=0$ and $\mathcal{S}_\rho\cong\C$ by example \ref{ex:S_perfect}, this implies the statement.
\end{proof}

Especially, lemma~\ref{3.1} and lemma~\ref{3.3} imply
\begin{cor}\label{cor:k_rho_fin}
 The functors $k_\rho$ and $k^\rho$ send modules of finite length to modules of finite length.
\end{cor}

\begin{cor}\label{cor:small_extension_lemma_preliminary}
For $X\in \CCC^{\fin}_{\Gl(1)}$ there is an exact sequence in ${\cal C}=\CCC_1^{\fin}$
$$0\to \mathbb{S}\otimes (\nu^{-1}X)_{\rho} \to k_\rho j_!i_*(X) \to i_*(X) \to 0$$
and an isomorphism $ \mathbb{S}\otimes (\nu^{-1}X)^{\rho} \cong k^\rho j_!i_*(X) \ .$
\end{cor}
Recall that $\mathbb{S}_n=(j_!)^n (\C)\in \CCC_n^{\fin}$.
\begin{cor}\label{3.2}\label{cor:k_rho_S2}
 $k_\rho({\mathbb S}_n) \cong {\mathbb S}_{n-1}$ and
$k^\rho({\mathbb S}_n) = 0 $ holds for $n\geq1$ and smooth $\rho$.
\end{cor}
\begin{proof}
The proof is by induction over $n$.
Indeed, $\mathbb{S}_1=\mathbb{S}$ is perfect by lemma~\ref{lem:C_b_perfect}.
$k_\rho(\mathbb{S})=\mathbb{S}_{\rho}$ is one-dimensional by lemma~\ref{lem:RR}.
For the induction step use lemma~\ref{lem:proto_small_extension_lemma} and note that $i^*(\mathbb{S}_n)=0$ by lemma~\ref{lem:Gelfand-Kazhdan}.
\end{proof}

On the level of the Grothendieck groups we obtain the following statement.
\begin{prop}\label{prop:K0k}\label{20.5}
For $M\in \CCC_n^{\fin}$ with $n\geq1$, in the Grothendieck group holds
$$  [k_\rho(M)] - [k^\rho(M)] = [j^!(M)]\quad \in K_0(\CCC_{n-1}^{\fin}) \ .$$
\end{prop}
\begin{proof}
The left hand side is well-defined and additive in $M$
by lemma~\ref{lem:long exact sequence_affine}, 
so we can assume that $M$ is irreducible.
Then $M\cong j_!^{m}i_*(\rho)$ for irreducible $\rho\in\CCC_{\Gl(n-m)}^{\fin}$ with $0\leq m\leq n$, see \cite[\S5.13]{Bernstein-Zelevinsky76}.
For $m=0$ the claim follows from lemma~\ref{3.1} and the assertion $j^!i_*=0$.
The general case follows by induction over $m$ using lemma \ref{lem:proto_small_extension_lemma},
the assertions $i^*j_!=0$, $j^!j_!\cong\id$ and the functorial exact sequence in lemma~\ref{lem:Gelfand-Kazhdan}.
\end{proof}

\begin{cor}
 $M$ in ${\cal C}^{\fin}_n$ is perfect if and only if $[k_\rho(M)] \in K_0({\cal C}^{\fin}_{n-1})$ does not depend on $\rho$.
\end{cor}
By lemma~\ref{3.1} and corollary~\ref{3.2}, $k^\rho(M)$ is zero for almost all $\rho$.

\subsection{The functor $\eta$} \label{functor eta}

The normal subgroup $\tilde N$ of the split Bessel group $\tilde R$
decomposes into two eigenspaces with respect to the action of $T \tilde R$ by conjugation.
The subgroup $S_A$ is generated by the matrices $s_a, a\in k$; the other is generated by the matrices $s_c,c\in k$. The normalizer in $G$ of $S_A$ is the Klingen parabolic subgroup $Q$ 
of $G$. By definition $Q$ consists of all elements $g\in G$ with the property $g(e_1)\in \C \cdot e_1$, where $e_1$ denotes the first vector of the fixed symplectic basis of $k^4$.
The Levi subgroup $L_Q\subseteq Q$ is chosen
to consist of all $g\in Q$ such that  $g(e_3)\in\C \cdot e_3$.
We construct an exact sequence
\begin{equation*}
\xymatrix{
0 \ar[r] &\Gl_a(2) \ar[r]^-{q} &  Q/S_A \ar[r]^-{\mu} &  k^{\times} \ar[r] &  0 }\ .\end{equation*}
Indeed, every $g\in Q$ is a unique product $g\! =\! \mu(g) l(g) n(g) $
of some $\mu(g)$ in the center $Z\cong k^\times$, some $l(g)\in L_Q$ with $l(g)e_3=e_3$, and $n(g)$ in the unipotent radical $N_Q$ of $Q$
$$l(g)\!=\! \begin{pmatrix}
\alpha\delta-\beta\gamma & 0 & 0 & 0 \cr
0 & \alpha & 0 & \beta \cr
0 & 0      & 1 & 0 \cr
0 & \gamma & 0 & \delta 
\end{pmatrix} \!\in\! L_Q  \quad , \quad  n(g) \! =\! 
\begin{pmatrix} 1 & -y & * & b \cr
0 & 1      & b & 0 \cr
0 & 0      & 1 & 0 \cr
0 & 0      & y & 1 
\end{pmatrix} \!\in\! N_Q   \  $$
The homomorphism $q$ is well-defined by
$$q\left(\begin{pmatrix} \alpha & \beta \cr \gamma & \delta \end{pmatrix}.0\right) = l(g)
\quad \mbox{ and }
\quad q\left(\id. \begin{pmatrix} b \cr y \end{pmatrix}\right)= n(g)\ .$$
For every $I\in \CCC_G(\omega)$, the factor group $Q/S_A$ acts naturally on the $S_A$-coinvariant space $I_{S_A}$.
Restriction along $q$ defines a smooth representation 
$\overline{I} = I_{S_A}\circ q$ of $\Gl_a(2)$. %
The constituents of $\overline{I}$ are described in lemma~\ref{Structure of TILDE I}.
This defines an exact functor
\begin{equation*}
\eta: {\cal C}_G({\omega})  \longrightarrow {\cal C}_{2} \quad ,\quad I\mapsto \overline{I}\ . \end{equation*}
For more information, see \cite[\S2.5]{Roberts-Schmidt}.

\textit{Twist}.
Recall that the twist by a smooth character $\mu$ of $k^\times$ sends the central character ${\omega}$ to ${\omega}\mu^2$.
There is a natural equivalence of functors $\CCC_G(\omega)\to\CCC_{2}$
$$\eta\circ  \mu\ \quad\cong\ \quad \mu\circ\eta\ .$$

For the unipotent radical $U=N_Q\cap M$ of $B_G\cap M$, let $J_P(-)_\psi=J_P(-)_{U,\psi}$ be the twisted coinvariant space of the Siegel Jacquet module
with respect to a non-trivial smooth character $\psi:U\to \C$.
Fix equivalences of categories $\CCC_T\cong\CCC_{\Gl(1)}$ and $\CCC_{L_Q}(\omega)\cong\CCC_{\Gl(2)}$ by pullback along $\Gl(1)\to T\,,\ \lambda\mapsto x_\lambda$ and $q:Gl(2)\to L_Q$, respectively.
\begin{lem} \label{Dual action}
There are natural equivalences of functors
\begin{align*}
J_P(-)_\psi\ \cong\ i^*j^!\eta 
\quad:&\quad \CCC_G(\omega)\to \CCC_T\cong\CCC_{\Gl(1)} \ , \\
J_Q(-)\ \cong\ i^*\eta
\quad:&\quad \CCC_G(\omega)\to \CCC_{L_Q}(\omega)\cong\CCC_{\Gl(2)} \ .
\end{align*}

In other words, there are commutative diagrams
\begin{equation*}
\xymatrix@C+4mm{
{\cal C}_G({\omega})\ar[d]_{J_P(-)_\psi} \ar[r]^\eta &  {\cal C}_2 \ar[d]^{i^*j^!}   \cr
  \CCC_T\ar[r]^\simeq &   {\cal C}_{\Gl(1)}  {\ ,}
} \qquad
\xymatrix@C+4mm{
{\cal C}_G({\omega})\ar[d]_{J_Q(-)} \ar[r]^\eta &  {\cal C}_2 \ar[d]^{i^*}   \cr
  \CCC_{L_Q}(\omega)\ar[r]^{\simeq} &   {\cal C}_{\Gl(2)}  {\ .}
} 
\end{equation*}
\end{lem}
\begin{proof}
This is clear by transitivity of the involved coinvariant functors.
\end{proof}

\begin{lem} \label{Structure of TILDE I}
For $I\in\CCC_G(\omega)$ there are functorial exact sequences
$$       0\to j_!j^!(\overline{I}) \to  \overline{I} \to  i_*(J_Q(I)|_{\Gl(2)})           \to 0 \quad \text{in}\ \CCC_2\ ,$$
$$       0\to {\mathbb S}^{m_I}    \to  j^!(\overline{I}) \to  i_*(J_P(I)_\psi) \to 0 \quad \text{in}\ \CCC_1\ ,$$
where $m_I$ denotes the dimension of the space of Whittaker functionals of $I$.
\end{lem}
\begin{proof}
Apply the short exact sequence of lemma~\ref{lem:Gelfand-Kazhdan} to $\overline{I}$ and to $j^!(\overline{I})$. 
The terms on the right hand side are given by lemma~\ref{Dual action}.
Finally, note that by construction $\dim j^!j^! (\overline{I}) = {m_I}$. 
\end{proof}

\begin{cor}\label{cor:eta_fin}
If $I\in\CCC_G(\omega)^\fin$ has finite length, $\overline{I}\in\CCC_2^\fin$ has finite length.
\end{cor}
\begin{proof}
Lemma~\ref{Structure of TILDE I} gives
an exact sequence in ${\cal C}_2$
$$  0\longrightarrow   j_!(i_*(J_P(I)_\psi)) \longrightarrow   \overline{I}/{\mathbb S}_2^{m_I} \longrightarrow   i_*(J_Q(I)|_{\Gl(2)}) \longrightarrow  0  \ .$$
It is well known that the Jacquet modules are has finite length.
Recall that $i_*$ and $j_!$ are exact and send irreducible modules to irreducible ones.
\end{proof}
In other words, the functor $\eta$ is well-defined between categories of smooth modules with finite length.
\begin{lem}\label{lem:long_exact_main}
For every $I\in \CCC_G(\omega)$ there is a long exact sequence in $\CCC_1$
\begin{equation*}
0\to k^\rho j_!j^!(\overline{I})  \to k^\rho(\overline{I}) \to k^\rho i_*(J_Q(I)) \to k_\rho j_!j^!(\overline{I})  \to k_\rho (\overline{I}) \to k_\rho i_*(J_Q(I)) \to 0\ .
\end{equation*}
If $I$ has finite length, then
$k^\rho i_*(J_Q(I))$ and $k_\rho i_*(J_Q(I))$ are finite-dimensional and it holds that 
$$\deg(k_\rho(\overline{I}))=\deg(k_\rho j_!j^!(\overline{I}))\quad,\quad \deg(k^\rho(\overline{I}))=\deg(k^\rho j_!j^!(\overline{I}))\ .$$
\end{lem}
\begin{proof}
Lemma~\ref{lem:long exact sequence_affine} applied to the first exact sequence of lemma~\ref{Structure of TILDE I} shows the first statement.
For the second statement note that the functors $k_\rho, k^\rho$ send objects in $i_* (\CCC^\fin_{\Gl(2)})$ to the full subcategory $i_*(\CCC^\fin_{\Gl(1)})$ of ${\CCC}^\fin_{1}$ by lemma~\ref{3.1}.
Counting degrees shows the third statement.
\end{proof}

\begin{lem}\label{lem:alpha_perfect}\label{20.8}
For irreducible generic representations $\Pi\in\CCC_G(\omega)$ the module
$\alpha(\Pi)=j^!(\overline\Pi)$ in $\CCC$ is perfect of degree one.
\end{lem}

\begin{proof} 
Notice, $\deg(\alpha(\Pi))=\deg(j^! \eta(\Pi))=m_\Pi=1$ holds
by the uniqueness of Whittaker models.
Therefore $\mathbb S$ embeds into $A=\alpha(\Pi)$.
By lemma~\ref{PERF}, $A$ is perfect if and only if ${\Danismanfunctor}(A)= 0$.
Since $\mathbb S$ does not contain nontrivial finite dimensional submodules,
the natural map ${\Danismanfunctor}(A) \oplus \mathbb S \to A$ is injective
and by degree reasons its cokernel $Y=\pi_0(A)/{\Danismanfunctor}(A)$ is finite-dimensional.
From lemma \ref{lem:long exact sequence_affine} we obtain for every $\rho$ a long exact sequence in ${\cal C}_1$
\begin{equation*}
\cdots\to k^\rho\circ j_!(Y)\to k_\rho\circ j_!\bigl({\Danismanfunctor}(A)\oplus \mathbb{S}\bigr) \to k_\rho\circ j_!(A) \to k_\rho\circ j_!(Y)\to 0\ .
\end{equation*}
Assume that $A$ is not perfect, then there is a smooth $\Gl(1)$-character $\rho$ such that $(\nu^{-1}{\Danismanfunctor}(A))_\rho$ is non-zero.
Corollary~\ref{cor:small_extension_lemma_preliminary} implies $\deg(k_\rho\circ j_!({\Danismanfunctor}(A))>0$,
so the degree of $k_\rho\circ j_!({\Danismanfunctor}(A)\oplus\mathbb{S})$
is at least two by corollary~\ref{cor:k_rho_S2}.
Since $Y$ is finite-dimensional, proposition~\ref{prop:K0k} implies $\deg(k^\rho\circ j_!(Y))= \deg(k_\rho\circ j_!(Y))$. By counting degrees in the long exact sequence we obtain $\deg(k_\rho\circ j_!(A))\geq 2$.
However, lemma~\ref{lem:long_exact_main} and lemma~\ref{lem:beta=k_rho_eta} imply $\deg(k_\rho\circ j_!(A)) = \deg(\widetilde \Pi)$ and this is one by corollary~\ref{cor:gen_Mult1}.
A contradiction, so $A=\alpha(\Pi)$ is indeed perfect.
\end{proof}

\subsection{Bessel functors $\beta_\rho$ and $\beta^\rho$} \label{Bessel functor}

For a fixed central character $\omega$ and a $\Gl(1)$-character $\rho$, the Bessel functors $\beta_\rho,\beta^\rho:\CCC_G(\omega)\to\CCC_{TS}$ are defined in the introduction.
We first show in lemma~\ref{lem:beta=k_rho_eta} that the Bessel functors factorize over $\eta$.
Moreover, for irreducible $I\in \CCC_G(\omega)$ the degree of $\beta_\rho(I)$ is at most one 
by uniqueness of split Bessel models \cite[6.3.2 i)]{Roberts-Schmidt_Bessel}.
We give an independent proof in theorem~\ref{Mult1}.

\begin{lem}\label{lem:long exact sequence}
An exact sequence $0\to E\to F\to G\to 0$ in $\CCC_{G}(\omega)$
gives rise to a long exact sequence in $\CCC_{TS}$
$$ 0\to \beta^\rho(E)\to \beta^\rho(F)\to \beta^\rho(G)\to \beta_\rho(E) \to \beta_\rho(F)\to
\beta_\rho(G) \to 0 \ .$$
\end{lem}
\begin{proof}
 This is implied by lemmas~\ref{lemma 3} and~\ref{lemma 2} because $\widetilde{N}$ is compactly generated and $\widetilde{T}/Z_G\cong k^\times$. Alternatively, use lemma~\ref{lem:beta=k_rho_eta} and exactness of $\eta$.
\end{proof}

\begin{lem}\label{lem:beta=k_rho_eta}
There are natural equivalences $\beta_\rho \cong k_{\rho} \circ \eta $
and $\beta^\rho \cong k^\rho \circ \eta$ of functors $\CCC_G(\omega)\to \CCC$. In other words, there are commutative diagrams
\begin{equation*}
\xymatrix@C+5mm{
{\cal C}_G({\omega})\ar[d]_{\beta_\rho} \ar[r]^\eta &  {\CCC}_2 \ar[d]^{k_\rho}   \cr
     \CCC_{TS}       \ar[r]^\simeq                &   \CCC_{1} }
\quad , \quad  
\xymatrix@C+5mm{
{\cal C}_G({\omega})\ar[d]_{\beta^\rho} \ar[r]^\eta &  {\CCC}_2 \ar[d]^{k^\rho}   \cr
       \CCC_{TS}  \ar[r]^\simeq                   &   \CCC_1 
}
\end{equation*}
with $k_\rho$ and $k^\rho$ defined in section \ref{the functors beta} and the equivalence of categories $\CCC_{TS}\cong\CCC_1$ described in example~\ref{example 4}.
\end{lem}
\begin{proof}
Transitivity of coinvariant functors gives an isomorphism
$I_{\widetilde{N}}\cong \overline{I}_{L_Q\cap N}$ of smooth $T\widetilde{T}S$-modules, functorial in $I\in \CCC_G(\omega)$. The rest is straightforward.
\end{proof}

\begin{cor}\label{cor:Bessel_finite}
The Bessel functors define functors $\CCC_G^\fin(\omega)\to \CCC^\fin_{TS}$ between the full subcategories of representations with finite length.
\end{cor}
\begin{proof}
The corresponding assertions hold for $\eta$ by corollary~\ref{cor:eta_fin} and for $k_\rho$, $k^\rho$ by corollary~\ref{cor:k_rho_fin}.
\end{proof}

\begin{rmk}\label{rmk:Bessel_twist}
Recall that the twist $\CCC_G(\omega)\to\CCC_G(\mu^2\omega)$ by a smooth character $\mu$ of $\Gl(1)$ satisfies natural equivalences
$\beta_\rho\circ \mu \cong \mu\circ\beta_{\rho\mu^{-1}}$ and
$\beta^\rho\circ \mu\cong \mu\circ\beta^{\rho\mu^{-1}}$.
Recall that $\rho\mapsto\rho^\divideontimes=\omega/\rho$ defines an involution on characters of $\Gl(1)$.\end{rmk}
\begin{lem}[Duality] \label{duality}
There are natural equivalences
$\beta_\rho \cong \beta_{\rho^\divideontimes}$ and
$\beta^\rho \cong \beta^{\rho^\divideontimes}$.
For every $I\in {\cal C}_G^{\fin}({\omega})$
there is a commutative diagram of $TS$-modules
\begin{equation*}
\xymatrix@+2mm{ I  \ar[d]_{I(\mathbf{s}_1)} \ar[r] &    \beta_\rho(I) \ar[d]^\cong   \cr
I  \ar[r] &    \beta_{\rho^\divideontimes} (I)\ .}
\end{equation*}
\end{lem}
\begin{proof}
The Weyl group element $\mathbf{s}_1\in G$
normalizes $\tilde{T}$ and $\widetilde{N}$ and centralizes $TS$. Conjugation with $\mathbf{s}_1$
permutes $t_1$ and $t_2$ and thus sends $\Lambda=\rho\boxtimes\rho^\divideontimes$ to $\Lambda^{\mathbf{s}_1}=\rho^\divideontimes\boxtimes\rho$, see section \ref{Bessel data} for the notation.
\end{proof}

\begin{prop}\label{lasst}
For $I\in {\cal C}_G^{\fin}({\omega})$ let
$$d(I,\rho) = \dim(\Hom_{T}(\delta_P^{-1/2}J_P(I)_{\psi},\nu^{-1/2} \rho))\ ,$$
then the degree of $\widetilde{I}=\beta_\rho(I)$ satisfies the estimate
 $$\ d(I,\rho) \leq\ \deg(\widetilde I)
\ \leq\ m_I + d(I,\rho) \ $$ where $m_I$ denotes the multiplicity of Whittaker models of $I$.
If $J_P(I)_\psi=0$, then $\deg(\widetilde I)= m_I$. If $J_P(I)=0$ vanishes, $\widetilde I$ is the direct sum of $m_I$ copies of $\mathbb{S}$.
\end{prop}
\begin{proof} 
The $T$-module $X=J_P(I)_\psi$ is isomorphic to $i^*j^!(\overline{I})$ by lemma~\ref{Dual action}.
The second  exact sequence of corollary~\ref{Structure of TILDE I} is
$$ 0 \to {\mathbb S}^{m_I} \to  j^!(\overline{I})  \to  i_*(X) \to 0 \ .$$
Right exactness of the functor $k_\rho\circ j_!$ and the assertion
$k_\rho\circ j_!({\mathbb S})={\mathbb S}$ by corollary~\ref{cor:k_rho_S2}
yield a long exact sequence
$$ \cdots\to k^\rho j_!i_*(X)\to \mathbb{S}^{m_I} \to  k_\rho j_!j^!(\overline{I})  \to  k_\rho j_!i_*(X) \to 0 \ .$$
Lemma~\ref{cor:small_extension_lemma_preliminary} and lemma~\ref{Dual action} imply %
$$\deg\bigl(k_\rho j_!i_*(X)\bigr) = \deg\bigl(j_!k_{\nu\rho} i_*(X)\bigr)= \dim(\Hom_{T}(\nu^{-1} X,\rho))=d(I,\rho)\ .$$
By counting degrees in the long exact sequence we obtain the estimate
$$\ d(I,\rho) \leq\ \deg(k_\rho j_!j^!(\overline{I}))\ \leq\ m_I + d(I,\rho) \ .$$
Now lemma~\ref{lem:long_exact_main} and lemma~\ref{lem:beta=k_rho_eta} imply the first statement because
\begin{equation*}\deg(k_\rho j_! j^!(\overline{I}))
= \deg( k_\rho(\overline{I})) = \deg(\widetilde I) \ .
\end{equation*}

If $X=0$ vanishes, then ${\mathbb S}^{m_I} \cong k_\rho j_!j^!(\overline{I})$ by the long exact sequence above and this has degree $m_I=\deg(\widetilde{I})$. 

If finally $J_P(I)=0$ vanishes, then so does the Borel-Jacquet module $J_P(I)_{U}$.
Lemma~\ref{3.1} shows that then $k_\rho(i_*(J_Q(I))) = 0$ and $k^\rho(i_*(J_Q(I)))=0$ also vanish. Lemma~\ref{lem:long_exact_main} and lemma~\ref{lem:beta=k_rho_eta} imply $\widetilde{I}\cong k_\rho(\overline{I})\cong k_\rho j_!j^!(\overline{I})$.
Clearly $X=0$ vanishes, so $k_\rho j_!j^!(\overline{I})\cong  \mathbb{S}^{m_I}$.
\end{proof}

\begin{cor}\label{cor:gen_Mult1}
For irreducible generic $\Pi\in \CCC_G(\omega)$ and every smooth Bessel character $\rho$, we have
 $\deg(\widetilde{\Pi})=1$.
\end{cor}
\begin{proof}
By proposition~\ref{lasst},
  $\deg(\widetilde{\Pi})\leq m_\Pi+d(\Pi,\rho)$
where $m_\Pi=1$ is the multiplicity of Whittaker models.
Note that $d(\Pi,\rho)$ can be non-zero only if $\rho$ is in the multiset $\Delta_{\pluss}(\Pi)$
defined in section \ref{s:Combinat}.
Replacing $\rho$ by $\rho^\divideontimes$
does not change $\beta_\rho(\Pi)$ up to isomorphism (lemma~\ref{duality}),
so without loss of generality we can assume that $\rho\notin \Delta_{\pluss}(\Pi)$ because $\Delta_{\pluss}(\Pi)\cap \Delta_{\pluss}^\divideontimes(\Pi)$ is empty by lemma~\ref{lem:intersection--*}.
This implies $d(\Pi,\rho)=0$ and therefore $\deg(\widetilde{\Pi})\leq m_\Pi$.
Proposition~\ref{BETA1} implies $\deg(\widetilde{\Pi})\geq m_\Pi$.
\end{proof}

\begin{prop}\label{BETA1}
For $I\in {\cal C}_G^{\fin}({\omega})$ it holds in the Grothendieck group $K_0(\CCC)$
$$[\beta_\rho(I)]-[\beta^\rho(I)]=[\alpha(I)]$$ for $\alpha(I)=j^!(\overline I)$ constructed in section~\ref{functor eta}.
Especially, the difference of degrees $\deg(\beta_\rho(I))-\deg(\beta^\rho(I)) = m_I$
is the multiplicity of Whittaker models of $I$.
\end{prop}

\begin{proof}
The first statement follows from proposition~\ref{prop:K0k} and lemma~\ref{lem:beta=k_rho_eta}.
For $\overline{I}=\eta(I)\in \CCC_2^{\fin}$,
the degree of $\beta^\rho(I)$ is $\dim j^! k^\rho \overline{I}$ and the degree of $\beta_\rho(I)$ is $\dim j^! k_\rho(\overline{I})$.
By definition, $m_I=\dim (j^!)^2(\overline{I})=\deg(\alpha(I))$.
\end{proof}

\begin{cor} \label{20.7}
If $\Pi$ in ${\cal C}_G({\omega})$ is generic irreducible, then
$\deg(\beta^\rho(\Pi)) = 0$.
\end{cor}
\begin{proof}
This is clear by corollary~\ref{cor:gen_Mult1} and proposition~\ref{BETA1}.
\end{proof}

\begin{lem}\label{lem:beta^rho_Jacquet}
For every $I\in\CCC_G^{\fin}(\omega)$ and every Bessel character $\rho$,
there is a functorial isomorphism $\pi_0(\widetilde{I})\cong J_P(I)_{\widetilde{T},\rho}$ of $T$-modules.
In the Grothendieck group of $T$-modules of finite length it holds that
\begin{equation*}[\pi_0(\beta^\rho(I))] = [J_P(I)^{\widetilde{T},\rho}]\ ,\qquad\text{in}\  K_0(\CCC_T^{\fin})\ .
\end{equation*}
\end{lem}
\begin{proof}
The first assertion is lemma~\ref{lem:Bessel_Jacquet}.
For the second statement, consider $J_P(I)$ as a $\Gl_a(1)\times T$-module where $\alpha.v\in\Gl_a(1)$ acts by $m_A$ for
$A=\left(\begin{smallmatrix}\alpha&v\\0&1\end{smallmatrix}\right)$ and $T$ by its natural embedding $T\hookrightarrow M$. 
Proposition~\ref{prop:K0k} implies
 \begin{equation*}
  [J_P(I)_{\widetilde{T},\rho}]-[J_P(I)^{\widetilde{T},\rho}] =[j^!(J_P(I))]= [J_P(I)_{U,\psi}]\quad\in K_0(\CCC_T^{\fin})\ .
 \end{equation*}
We have $J_P(I)_{\widetilde{T},\rho} \cong \pi_0(\widetilde{I})$ by the first assertion. 
There is an isomorphism $J_P(I)_{U,\psi}\cong \pi_0(\alpha(I))$ by lemma~\ref{Dual action}.
The exact functor $\pi_0$ applied to the equation of proposition~\ref{BETA1} yields
\begin{equation*}
[\pi_0(\beta^\rho(I))]=[\pi_0(\widetilde{I})]-[\pi_0(\alpha(I))]
\end{equation*} and this shows the statement.
\end{proof}

\begin{prop}\label{5.3}
For $I\in {\cal C}_G^{\fin}({\omega})$ one has 
$   L(\beta_\rho(I),s)\! =\! L({J_P(I)}_{\widetilde{T},\rho},s) $ and
\begin{equation*}  L(J_P(I)_{\psi},s) \ =\ \frac{L(\beta_\rho(I),s)}{L(\beta^\rho(I),s)}\ . \end{equation*}
The finite dimensional $T$-module $J_P(I)_{\psi}$ is viewed as a $TS$-module with the trivial $S$-action; it is independent of $\rho$.
\end{prop}
\begin{proof}
This is a restatement of lemma~\ref{lem:Bessel_Jacquet} and proposition~\ref{BETA1} in terms of section~\ref{s:L-series}.
\end{proof}

\subsection{Mellin functors $M_\rho$ and $M^\rho$}\label{s:Mellin}

The Mellin functors attached to a smooth $\Gl(1)$-character $\rho$ are the additive endofunctors $M_\rho=k_\rho\circ j_!$ and $M^\rho=k^\rho\circ j_!$ of $\CCC_n^{\fin}$.
Corollary~\ref{cor:k_rho_S2} implies $M_\rho(\mathbb{S}_n)\cong\mathbb{S}_n$ and $M^\rho(\mathbb{S}_n)=0$ for every $\rho$.
In the Grothendieck group of $\CCC_n^\fin$,
proposition~\ref{prop:K0k} implies $[X] = [M_\rho(X)]-[M^\rho(X)]$ for every $[X]$ .
\begin{lem}\label{lem:nat_eq_Mellin}
There are natural equivalences $i^*\!\circ\! M_\rho\cong i^*$ and $i^*\!\circ\! M^\rho\cong 0$ of functors $\CCC_{n}\to \CCC_{\Gl(n)}$.
In other words, there are commutative diagrams
\begin{equation*}
\xymatrix@R+0mm{
  \ \CCC_n\ar[rr]^{M_\rho}\ar[drr]_{i^*}&&\CCC_n\ar[d]^{i^*}\\
  &&\CCC_{\Gl(n)}  
}\ ,\qquad
\xymatrix@R+0mm{
  \ \CCC_n\ar[rr]^{M^\rho}\ar[drr]_{0}&&\CCC_n\ar[d]^{i^*}\\
  &&\CCC_{\Gl(n)}    
}\ .
\end{equation*}
\end{lem}
\begin{proof}
Apply the exact functor $i^*$ to the short exact sequence and the natural isomorphism of lemma~\ref{lem:proto_small_extension_lemma}.
Note that $i^*j_!=0$ and $i^*i_*=\id$.
\end{proof}

In the following section we determine $M_\rho(X)$ and $M^\rho(X)$ for certain $X\in \CCC_1^\fin$.

\begin{thm}\label{thm:big_extension}\label{thm:ME_split}
For a universal extension $\mathbb{E}[\mu]\in \CCC_1^\fin$ of a $\Gl(1)$-character $\mu$ and for an arbitrary Bessel character $\rho$, we have 
\begin{equation*}
 M_\rho (\mathbb{E}[\mu])\cong
 \begin{cases}
  \mathbb{E}[\mu]           & \mu\!\neq\!\nu^2\rho\ ,\\
  \mathbb{S}\oplus i_*(\mu) & \mu\!=\!\nu^2\rho\ .
 \end{cases}
\end{equation*}
We have $M^\rho (X) = 0$ for every perfect $X\in\CCC_1^\fin$.
\end{thm}

\begin{proof}
Without loss of generality we assume $\mu=\nu^{3/2}$ as a $T$-character by a suitable twist.
Let $\Pi=\Ind_P^G(\pi\boxtimes 1)\in\CCC^{\fin}_G(\omega)$ be the normalized fully Siegel induced irreducible representation of case \nosf{X} in the sense of \cite{Sally_Tadic} and table \ref{tab:List}, for
an irreducible cuspidal representation $\pi\in\CCC_{\Gl(2)}^{\fin}(\omega_{\pi})$ of $\Gl(2)$ with central character $\omega=\omega_{\pi}\neq\nu^{\pm1}$. We can arrange that $\omega\neq1,\nu^{-1/2}\rho$ by choice of $\pi$.
The normalized Siegel-Jacquet module $\delta_P^{-1/2}J_P(\Pi)$ has constituents $\pi\boxtimes 1$ and $\pi^\vee\boxtimes\omega$ by \cite[table~A.3]{Roberts-Schmidt}.
Note that $\delta_{P}^{1/2}(x_\lambda)=\nu^{3/2}(\lambda)$. Since $\omega\neq1$, the $\psi$-coinvariant space of the Siegel-Jacquet module is $J_P(\Pi)_\psi\cong \delta_P^{1/2}(1\oplus\omega)$ as a direct sum of $T$-characters.
By lemma~\ref{lem:alpha_perfect}, $A=\alpha(\Pi)=j^!(\overline{\Pi})\in \CCC$ is perfect of degree one. Since $\pi_0(A)$ is isomorphic to $J_P(\Pi)_\psi$ as a $T$-module (lemma~\ref{Dual action}), we obtain a short exact sequence in $\CCC$
$$ 0\to \mathbb{E}[\nu^{3/2}]\to A\to i_*(\nu^{3/2}\omega) \to 0 \ .$$ 
Lemma~\ref{lem:long exact sequence_affine} provides a long exact sequence
$$\cdots\to M^\rho(i_*(\nu^{3/2}\omega)) \to M_\rho(\mathbb{E}[\nu^{3/2}]) \to M_\rho(A) \to M_\rho(i_*(\nu^{3/2}\omega))\to0\ .$$
The left term is zero and the right term is
$M_\rho(i_*(\nu^{3/2}\omega))\cong i_*(\nu^{3/2}\omega)$ by lemma~\ref{3.3}.
Therefore, we have an exact sequence
\begin{equation}\label{eq:big_Mellin1}
0\to M_\rho(\mathbb{E}[\nu^{3/2}]) \to M_\rho(A) \to i_*(\nu^{3/2}\omega) \to0\ .\tag{*}
\end{equation}
Every smooth character $\rho$ defines a split Bessel model of $\Pi$.
Since $\Pi$ is fully induced and generic, we have constructed
in section~\ref{s:Bessel_filt} a long exact filtration sequence
$$\cdots\to\beta^\rho(I_{\leq1})\stackrel{\delta}\to \widetilde{I_{\geq2}}\to\widetilde{\Pi}\to\widetilde{I_{\leq1}}\to0\ .$$
We distinguish two cases:
Either $(\sigma_\Pi,\rho)$ is ordinary with $\rho\neq\rho_\pm(\sigma_\Pi)$ or it is non-ordinary with $\rho=\rho_+(\sigma_\Pi)=\nu^{-1/2}$ because we assumed $\rho\neq\rho_-(\sigma_\Pi)= \nu^{1/2}\omega$ by choice of $\omega$.

\bigskip\noindent
\textit{Ordinary case $\rho\neq\rho_\pm(\sigma_\Pi)$}.
By lemmas~\ref{tildemod} and \ref{tildeKoh} the long exact filtration sequence in $\CCC$ is
$$0\to \delta_P^{-1/2}\otimes\mathbb{E}[1]  \to  \widetilde{\Pi} \to  \delta_P^{-1/2}\otimes i_*(\omega)\to 0\ .$$
because $\widetilde{I}_2=0$, $\widetilde{I_1}=0$ and $\beta^\rho(I_{\leq1})=0$.
Lemma~\ref{Structure of TILDE I} provides an isomorphism $\widetilde{\Pi}\cong M_\rho(A)$ since $J_Q(\Pi)=0$. This yields an exact sequence in $\CCC$
$$0\to \mathbb{E}[\nu^{3/2}]\to   M_\rho(A ) \to  i_*(\nu^{3/2}\omega)\to 0\ .$$
Comparing this with \eqref{eq:big_Mellin1} yields the required assertion
$$ M_\rho(\mathbb{E}[\nu^{3/2}])\cong \mathbb{E}[\nu^{3/2}] \ .$$
\textit{Non-ordinary case $\rho=\rho_+(\sigma_\Pi)=\nu^{-1/2}$}.
By lemmas~\ref{tildemod} and \ref{tildeKoh} the filtration sequence is
$$\dots \to \beta^\rho(I_1)\stackrel{\delta}\to\widetilde{I_3}\to\widetilde{\Pi}\to \widetilde{I_{\leq1}} \to 0\ .$$
with an extension
$0 \to \widetilde{I_1} \to \widetilde{I_{\leq1}} \to i_*(\nu^{3/2}\omega)\to0$ of degree one.
Note that $\beta^\rho(I_1)\cong\widetilde{I_1}\cong\delta_P^{1/2}\delta_B^{-1/2}\pi \cong \mathbb{S}$ by example~\ref{ex:cusp_as_TS_module}.
By uniqueness of Bessel models $\widetilde{\Pi}$ has degree one (corollary~\ref{cor:gen_Mult1}), so
counting degrees shows that the image of $\delta$ has degree one and thus is isomorphic to $\mathbb{S}$.
Hence $\widetilde{I_3}/\mathbb{S}\cong i_*(\nu^{3/2})$ yields
a short exact sequence in $\CCC$
$$0\to i_*(\nu^{3/2})\to\widetilde{\Pi}\to \widetilde{I_{\leq1}} \to 0\ .$$
Since $J_Q(\Pi)=0$, lemma~\ref{Structure of TILDE I} provides an isomorphism
$\widetilde\Pi\cong k_\rho(\overline\Pi)\cong M_\rho(A)$.
Together with \eqref{eq:big_Mellin1} this yields a commutative diagram in $\CCC$ with exact sequences

\begin{equation*}
 \xymatrix{
         & 0  \ar[d]                       &      0\ar[d]    &                   \mathbb{S}\ar[d]& \\
 0\ar[r] & i_*(\nu^{3/2})\ar[r]\ar@{.>}[d]_\varphi & \widetilde{\Pi}\ar[r]\ar[d]_\cong & \widetilde{I_{\leq1}}\ar[r]\ar[d] & 0\\
 0\ar[r] & M_\rho(\mathbb{E}[\nu^{3/2}])\ar[r]\ar@{.>}[d] & M_\rho(A)\ar[r]\ar[d] & i_*(\nu^{3/2}\omega)\ar[r]\ar[d] & 0\\
         & \coker(\varphi)                    &0                &0 
 }
\end{equation*}
where $\varphi$ is the uniquely determined injection that makes the left square commute.
The snake lemma gives $\coker(\varphi)\cong\mathbb{S}$ and thus an exact sequence
$$0 \to i_*(\nu^{3/2}) \to M_\rho(\mathbb{E}[\nu^{3/2}]) \to \mathbb{S} \to 0  $$
which splits by lemma \ref{lem:Ext(S,F)}. This shows the required assertion.%

Finally, $M^\rho(X)=0$ vanishes for perfect $X\in\CCC_1^\fin$ by lemma~\ref{lem:proto_small_extension_lemma}.
\end{proof}

We now determine $M_\rho(i_*(X))$ and $M^\rho(i_*(X))$ for finite-dimensional smooth $\Gl(1)$-modules $X\in\CCC_{\Gl(1)}^\fin$.
Since the Mellin functors are additive,
it suffices to consider indecomposable Jordan blocks.

\begin{lem}\label{lem:small extension}
For an indecomposable Jordan block $X=\mu^{(m)}\in\CCC_{GL(1)}^{\fin}$ of finite length $m>0$ attached to a smooth character $\mu$ of $\Gl(1)$, it holds that
\begin{equation*}
M_\rho(i_*(X))  =  \begin{cases}
                     \mathbb{E}[{X\to \mu}]        & \mu\! =  \!\nu\rho\ ,\\
                     i_*(X)                            & \mu\!\neq\!\nu\rho\ ,
                  \end{cases}\qquad
M^\rho(i_*(X))  =  \begin{cases}
                     \mathbb{S}                    & \mu\! =  \!\nu\rho\ ,\\
                     0                                 & \mu\!\neq\!\nu\rho\ ,
                  \end{cases}
\end{equation*} 
where $X\to \mu$ is the natural projection to $\mu$-coinvariants.
\end{lem}
\begin{proof}
For $\mu\neq\nu\rho$ the statement follows directly from lemma~\ref{cor:small_extension_lemma_preliminary},
so we can assume $\mu=\nu\rho$.

We first assume $m=1$. Lemma~\ref{cor:small_extension_lemma_preliminary} shows everything
except that the exact sequence $0 \to \mathbb{S} \to M_\rho(i_*(\mu))\to i_*(\mu) \to 0 $ does not split.
Since $M_\rho$ is right exact, the exact sequence $0\to\mathbb{S}\to\mathbb{E}[\mu]\to i_*(\mu)\to0$ yields
a partial exact sequence
$$
\cdots\to 
M_\rho(\mathbb S) \to M_\rho(\mathbb E[\mu]) \to M_\rho(i_*(\mu)) \to 0 \ .$$
Theorem~\ref{thm:big_extension} yields $M_\rho(\mathbb E[\mu]) \cong \mathbb E[\mu]$.
Since $\mathbb E[\mu]$ and $M_\rho(i_*(\mu))$ both have length two,
the right surjection is an isomorphism $ M_\rho(\mathbb E[\mu]) \cong M_\rho(i_*(\mu)) \ .$

We now assume $m>1$.
For the Jordan block $Y=\ker(X\to \mu)\cong\mu^{(m-1)}$ there is a short exact sequence
$$0\to Y\to X\to \mu\to 0$$
and thus by lemma~\ref{lem:long exact sequence_affine} a long exact sequence
\begin{equation*}
\xymatrix{
0\ar[r] & M^\rho(i_*(Y))\ar[r] & M^\rho(i_*(X))\ar[r] & M^\rho(i_* (\mu))\ar@{->}`r[d]`d[l]`^d[lll]_\delta`d[l] [dll]\\
        & M_\rho(i_*(Y))\ar[r] & M_\rho(i_*(X))\ar[r] & M_\rho(i_*( \mu))\ar[r] & 0\ .
        }
\end{equation*}

By lemma~\ref{cor:small_extension_lemma_preliminary} every term in the long exact sequence has degree one.
The same lemma furnishes a second short exact sequence
\begin{equation*}
0\to \mathbb{S}\to M_\rho(i_*(X)) \to i_* (X) \to 0\ .
\end{equation*}

Since $M^\rho(i_*(\mu))\cong\mathbb{S}$, counting degrees shows $\im(\delta) \cong \mathbb{S}$.
We obtain a short exact sequence in $\CCC_1^\fin$
\begin{equation*}
\xymatrix{
0\ar[r]  &M_\rho(i_*(Y))/\mathbb{S}\ar[r] & M_\rho(i_*( X))\ar[r] & M_\rho(i_*(\mu))\ar[r] & 0\ .\\
}
\end{equation*}
The second short exact sequence applied to $Y$ yields $M_\rho(i_*(Y))/\mathbb{S}\cong i_*(Y)$.
We have already shown that $M_\rho(i_*(\mu))\cong \mathbb{E}[\mu]$,
hence $\Danismanfunctor(M_\rho(i_*(X)))\cong Y$.
Lemma~\ref{lem:nat_eq_Mellin} implies $i^* \circ M_\rho(i_*(X)) \cong X$,
so lemma~\ref{lem:universal extensions classification} implies the statement $M_\rho(i_*(X))\cong \mathbb{E}[X\to \mu]$.

The last assertion $M^\rho(i_*(X))\cong \mathbb{S}$ is clear by lemma~\ref{cor:small_extension_lemma_preliminary}.
\end{proof}

\begin{lem}\label{lem:extended_Mellin_lemma}
For a smooth character $\mu$ of $\Gl(1)$ let $X=\mu^{(2)}\in\CCC_{\Gl(1)}$ be the associated Jordan block of length two.
Then $M^\rho (\mathbb{E}[X])=0$ vanishes for every smooth $\rho$ and there is an isomorphism of $\Gl_a(1)$-modules
\begin{equation*}
 M_\rho  (\mathbb{E}[X]) \cong
 \begin{cases}
  \mathbb{E}[X\to\mu] & \text{if}\ \mu\!=\!\nu\rho\ \text{or}\ \mu\!=\!\nu^2\rho\ ,\\
  \mathbb{E}[X]       & \text{otherwise} \ .
 \end{cases}
\end{equation*}
\end{lem}
\begin{proof}
Let $\pi\in\CCC_{\Gl(2)}$ be cuspidal irreducible with trivial central character.
The Siegel induced representation $\Pi=\pi \rtimes \mu\nu^{-3/2}\in\CCC_G$ is generic irreducible of type~\nosf{X}.
Note that $j^!(\overline{\Pi})$ is perfect of degree one by lemma~\ref{lem:alpha_perfect} and there is an isomorphism of $T$-modules $i^*j^!(\overline{\Pi})\cong J_P(\Pi)_\psi$ by lemma~\ref{Dual action}, so $j^!(\overline{\Pi})$ is isomorphic to $\mathbb{E}[X]$ by lemma~\ref{lem:extension E_X}.
The Klingen Jacquet module $J_Q(\Pi)$ vanishes,
so there is an isomorphism $j_!j^!\overline{\Pi}\cong\overline{\Pi}$ of $\Gl_a(2)$-modules by lemma~\ref{Structure of TILDE I}.
This yields an isomorphism of $\Gl_a(1)$-modules
$k_\rho j_!(\mathbb{E}[X])\cong k_\rho(\overline{\Pi})$.
The assertion follows from theorem~\ref{thm:sigma_tilde_Pi} and theorem~\ref{thm:monodromy}.
\end{proof}

\section{Siegel induced representations}\label{s:SiegelInd}

We want to determine the $T$-module $\pi_0(\widetilde{\Pi})$ for every irreducible $\Pi\in\CCC_G(\omega)$.
By lemma~\ref{lem:Bessel_Jacquet} it suffices to consider those $\Pi$ whose Siegel-Jacquet module $J_P(\Pi)=\Pi_N$ does not vanish.
By dual Frobenius reciprocity, $\Pi$ is then isomorphic to a quotient of a Siegel induced representation $\Ind_P^G(\sigma_\Pi)$ from an irreducible $\sigma_\Pi\in\CCC_M$.
This $\sigma_\Pi$ is not uniquely determined by $\Pi$,
we may choose it as in table~\ref{tab:List} and give the full list of possible $\sigma_\Pi$ in table~\ref{tab:list_sigma}.
For simplicity, we fix a representative of each equivalence class under twisting.

For the standard Levi subgroup of the Siegel parabolic
we use the decomposition $\Gl(2)\times\Gl(1)\cong M$ via $(A,\lambda)\mapsto m_A t_\lambda$ with notations as in section~\ref{Bessel data}.
Every irreducible $\sigma_\Pi\in\CCC_M$ is of the form
$$\sigma_\Pi(m_At_\lambda)=\pi(A)\chi_\Pi(\lambda)$$
for irreducible smooth representations
$\pi=(\pi,V)$ of $\Gl(2)$ and $\chi_\Pi$ of $\Gl(1)$.
The twist of $\sigma_\Pi=\pi\boxtimes\chi_\Pi$ by a $\Gl(1)$-character $\mu$ is $\mu\otimes\sigma_\Pi=\pi\boxtimes\chi_\Pi\mu$
and the contragredient is $\sigma_\Pi^\vee\cong\pi^\vee\boxtimes\chi_\Pi^{-1}$.
Recall that the modulus factor of the Siegel parabolic $P$ is given on $M$ by
\begin{equation*}
 \delta_P(m_A t_\lambda)=\left\vert\frac{\det(A)}{\lambda}\right\vert^3\quad,\ \qquad\delta_P(x_\lambda)=\vert\lambda\vert^3\ .
\end{equation*}
The normalized Siegel induced representation $\pi\rtimes\chi_\Pi=\Ind_P^G(\sigma_\Pi)$ is the right-regular action of $G$ on the space of smooth functions $f:G\to V$ with compact support modulo $P$ such that 
$$ f(ms g)= \delta_P^{1/2}\sigma_\Pi(m) f(g)$$
for $s\in N$, $m\in M$ and $g\in G$.
The central character of $\pi\rtimes\chi_\Pi$ is $\omega=\omega_\pi\chi_\Pi^2$,
where $\omega_\pi$ is the central character of $\pi$.
There are natural isomorphisms 
$$\mu\otimes \Ind_P^G(\sigma_\Pi)\cong \Ind_P^G(\mu\otimes\sigma_\Pi)\qquad \text{and}\qquad \Ind_P^G(\sigma_\Pi)^\vee\cong \Ind_P^G(\sigma_\Pi^\vee)\ .$$
For details, see Tadi\'{c} \cite{Tadic}.

\begin{thm}\label{thm:sigma_Pi_existence}
For every irreducible $\Pi\in\CCC_G(\omega)$ and $\sigma_\Pi\in \CCC_M(\omega)$, the following assertions are equivalent:
\begin{enumerate}
\item $\Pi$ is a quotient of $\Ind_P^G(\sigma_\Pi)$,
\item $\Pi$ is a submodule of $\Ind_P^G(\omega\otimes\sigma_\Pi^\vee)$,
\item $\omega\otimes \sigma_\Pi^\vee$ is a quotient of $E=\delta_P^{-1/2}\otimes J_P(\Pi)$ for the unnormalized Siegel-Jacquet module $J_P(\Pi)=\Pi_N$.
\end{enumerate}
For $\Pi$ normalized as in table~\ref{tab:List}, $\sigma_\Pi$ satisfies these assertions if and only if it appears in table~\ref{tab:list_sigma}.
\end{thm}
\begin{proof}
The involution $I\mapsto \omega\otimes I^\vee$ of $\CCC^{\fin}_G(\omega)$ is a contravariant exact functor that preserves the irreducible constituents, so the first two assertions are equivalent.
By dual Frobenius reciprocity, the second and third assertion are equivalent.

Since $T$ is in the center of $M$, the decomposition $E=\bigoplus_{\chi_{\norm}}E^{(T,\chi_\norm)}$ into generalized $T$-eigenspaces is preserved under the action of $M$.
It is easy to see that for $\Pi$ not of type \nosf{VIa} or \nosf{VId}, every constituent of $E$ occurs as a quotient, see \cite[table~A.3]{Roberts-Schmidt}.
For type \nosf{VIa}, normalized as in table \ref{tab:List},
the constituents of $E$ are $\Sp(\nu^{1/2})\boxtimes\nu^{-1/2}$ (twice) and  $(\nu^{1/2}\circ \det)\boxtimes\nu^{-1/2}$.
However, the one-dimensional constituent is not a quotient of $E$ because $\Pi$ is not a constituent of
$\Ind_P^G((\nu^{1/2}\circ \det)\boxtimes\nu^{-1/2})$, see \cite[(2.11)]{Roberts-Schmidt}.
For type \nosf{VId} the argument is analogous.
\end{proof}

\begin{cor}\label{cor:sigma_Pi_existence}
For irreducible $\Pi\in\CCC_G^{\fin}(\omega)$ and
for every $\chi_{\norm}\in\Delta(\Pi)$,
there is an irreducible $\sigma_\Pi=\pi\boxtimes\chi_\Pi\in \CCC_M^{\fin}(\omega)$
such that $\Pi$ is a quotient of $\Ind_P^G(\sigma_\Pi)$ and $\chi_\Pi(t_\lambda)=\chi_{\norm}(x_\lambda)$.
\end{cor}

\begin{defn}
For irreducible $\sigma_\Pi = \pi \boxtimes \chi_\Pi\in\CCC_M(\omega)$ as above and a smooth character $\rho$ of $\Gl(1)$,
the pair $(\sigma_\Pi,\rho)$
will be called \emph{ordinary} if either $\dim(\pi)\! =\! 1$ holds or if $\rho$ is different from the two characters
$$ \rho_+ \!=\! \nu^{-1/2}\chi_\Pi \quad , \quad \rho_- =  \rho_+^\divideontimes = {\omega}/\rho_+\ .$$
If $(\sigma_\Pi,\rho)$ is not ordinary and $\rho_+\neq\rho_-$, it will be called \emph{non-ordinary}.
The case where $\rho=\rho_+=\rho_-$ is called \emph{extraordinary}.
\end{defn}
Notice that $\rho_+ \! =\! \rho_-$ if and only if
$\nu^{-1}\chi_\Pi\! =\! {\omega}\chi_\Pi^{-1}$ if and only if $\omega_\pi\!=\!\nu^{-1}$.
The case $\sigma_\Pi \!=\! (\chi_1\times \nu) \boxtimes \nu^{-1/2}$ and $\rho\! =\!\chi_1 \!=\! \nu^{-1}$
is non-ordinary by definition but behaves like an ordinary case.
To be consistent we always assume for the family \nosf{IIIa, IIIb} that
the character $\chi_1$ satisfies $\vert \chi_1\vert = \nu^s$ with $s\geq 0$.
This can be achieved by the Weyl reflection that switches $\chi_1$ and $\chi_1^{-1}$.

\begin{lem}\label{lem_dual_ord}
If $(\sigma_\Pi,\rho)$ is non-ordinary or extraordinary and $\Ind_P^G(\sigma_\Pi)$ is not irreducible, then $(\omega\otimes\sigma^\vee_\Pi,\rho)$ is ordinary.
\end{lem}
\begin{proof} Note that $\omega\otimes \sigma_\Pi^\vee\cong \pi^\vee\boxtimes\chi_\Pi^\divideontimes$.
If $(\omega\otimes\sigma^\vee_\Pi,\rho)$ is not ordinary, then
either $\rho_+(\chi_\Pi) = \nu^{-1/2} \chi_\Pi^\divideontimes$ or
$\rho_-(\chi_\Pi) = \nu^{-1/2} \chi^\divideontimes_\Pi$.
This amounts to $\chi_\Pi=\chi_\Pi^\divideontimes$ resp.\ $\nu^{1/2}\chi_\Pi^\divideontimes = \nu^{-1/2}\chi_\Pi^\divideontimes$.
The latter is clearly impossible, so $\chi_\Pi = \chi_\Pi^\divideontimes$ implies
$\chi_{\Pi}^2={\omega}$, hence $\omega_\pi=1$. This only occurs for the case $\pi=\nu\times\nu^{-1}$ that has already been excluded.
\end{proof}

\subsection{Bessel filtration} \label{ordinary}\label{s:Bessel_filt}
For irreducible $\sigma_\Pi\in\CCC_M(\omega)$,
the Siegel induced representation $I=\Ind_P^G(\sigma_\Pi)$
carries a filtration of $P$-modules $I_{\cell}^{\bullet}$
corresponding to Schubert cells on $P\!\setminus\! G$ defined by double cosets in $P\! \setminus \! G\!/P$.
By Bruhat decomposition, these double cosets are represented by the relative Weyl group
$W_M\mspace{-4mu} \setminus \mspace{-4mu} W\!/W_M=\{id,\mathbf{s}_2,\mathbf{s}_2\mathbf{s}_1\mathbf{s}_2\}$. This filtration $I_{\cell}^{\bullet}$ is denoted $I^\bullet$ in \cite[\S5.2]{Roberts-Schmidt_Bessel}.
As modules of $RT=\widetilde{T}TN\subseteq P$, this induces a finer filtration
$$ 0\ \subseteq\ I_{\geq 3} \ \subseteq\ I_{\geq 2}\ \subseteq\ I_{\geq 1}\ \subseteq\ I_{\geq 0}\ =\ I $$
whose associated graded components
$$I_i= I_{\geq i}/I_{\geq i +1}$$
are described in section~\ref{s_app:filtration}.
Since the Bessel group $\tilde R$ is contained in $RT$,
this allows to analyze the structure of the 
$\beta_\rho(I)=\widetilde I$ in terms of this filtration.
We recall some notations and simple facts:
We identify $TS$ with the subgroup $\Gl_a(1)\subseteq\Gl(2)$
as in example~\ref{example 4}.
The normalization factor for the Borel group $B$ of
upper triangular matrices in $\Gl(2)$ with unipotent radical $U$ is
$\delta_B^{1/2}(x_\lambda)\!=\! \vert \lambda\vert^{1/2}$.
As a $T$-module, the semisimplified Jacquet quotient is
\begin{equation*}
\delta_B^{-1/2} (\chi_\Pi\otimes \pi)_U^{ss}
\quad\cong\quad
\begin{cases}
 \quad\chi_1\chi_\Pi\!+\! \chi_2\chi_\Pi & \pi\! =\! \chi_1 \! \times\! \chi_2\ ,\\
 \quad\nu^{1/2}\chi\chi_\Pi               & \pi\!=\! \Sp(\chi) \ ,               \\
 \quad\nu^{-1/2}\chi\chi_\Pi              & \pi\!=\! (\chi\circ \det)\ .
\end{cases}
\end{equation*}

For the next lemmas \ref{tildemod} and \ref{tildeKoh} we define the filtration index $m\! =\! m(\sigma_\Pi,\Lambda)=\delta_{\rho\rho_+}+\delta_{\rho\rho_-}\in\{0,1,2\}$ (Kronecker delta).
It is equal to zero for $\rho\!\neq\! \rho_\pm$,
to two for $\rho\! =\! \rho_+\! =\! \rho_-$, and to one otherwise.

\begin{lem} \label{tildemod} For $I=\Ind_P^G(\sigma_\Pi)$ with filtration as above, 
the $TS$-modules $\delta_P^{-1/2}\otimes {\widetilde I_i}\ $ are
\begin{enumerate}
\item[\qquad$i=3$:] $\quad\begin{cases}
                      0           &  \text{if } \pi=\mu\circ\det\text{ and }\ \rho\neq\mu\chi_\Pi\ ,\\
                      \chi_\Pi \otimes C_c^\infty(S)\ \  & \text{otherwise}\ ,
                      \end{cases}$
\item[\qquad $i=2$:]  $ \quad \chi_\Pi\delta_B^{-1/2}\otimes \pi_U\quad\quad \ $ on the space $V_U\ $ of $\dim(V_U)\leq 2$\ ,
\item[\qquad $i=1$:] $\quad  m\cdot \chi_\Pi\delta_B^{-1/2}\otimes \pi \quad\ $ as a $TS$-module (see example \ref{example 4}) \ ,
\item[\qquad $i=0$:] $\quad \begin{cases}
                      0   \qquad\qquad\quad     & \text{if }\ \pi = \mu\circ \det\text{ and }\ \rho\neq\mu\chi_\Pi\\
                      \chi_\Pi^\divideontimes\qquad\qquad\quad & \text{otherwise}\ .
                             \end{cases}$
\end{enumerate}
\end{lem}

Concerning $\mathbb{E}=C_c^\infty(S)$ and $\pi, \pi_U$ as $TS$-modules, see examples~\ref{example 3} and \ref{example 4}.
\begin{proof}
$\widetilde{I_i}\cong((I_{i})_{\widetilde{N}})_{\widetilde{T},\rho}$ holds by definition.
The $\widetilde{N}$-coinvariant spaces $(I_i)_{\widetilde{N}}$ are given in section \ref{s_app:filtration} and it remains to determine $\rho$-coinvariants.
 For $i=3$ the result follows from corollary~\ref{cor:Waldspurger-Tunnell} with the action on $C_c^\infty(S)$ given in example~\ref{example 3}.
 For $i=2$ the $\rho$-coinvariant space is given by integrating $f\in C_c^\infty(k^\times,V_{U})$ over $k^\times$. The action of $\widetilde{T}$ on $(I_1)_{\widetilde{N}}\cong V\oplus V$ is by multiplication with the character $\rho_+$ on one factor and $\rho_-$ on the other factor. Thus $\widetilde{I_1}$ is isomorphic to $m$ copies of $\nu\chi_\Pi\otimes\pi$ as a $\Gl_a(1)$-module.
 The $\rho$-coinvariant space of $(I_0)_{\widetilde{N}}$ is given by corollary~\ref{cor:Waldspurger-Tunnell}.
 Finally, recall $\delta_P(x_\lambda s_b)=\vert\lambda\vert^3$.
\end{proof}

\begin{lem} \label{tildeKoh} For $\beta^\rho(I_i) = ((I_i)_{\tilde N})^{\tilde T,\Lambda}$ the $TS$-modules 
$\delta_P^{-1/2} \otimes \beta^\rho(I_i)$ are
\begin{enumerate}
\item[\qquad $i=3$:] $\quad\begin{cases}
            \chi_\Pi\otimes C^\infty_c(S) \ \ & \text{if }\pi=\mu\circ\det \text{ and } \rho=\mu\chi_\Pi\ ,\\ 
            0             & \text{otherwise}\ ,
                       \end{cases}$
\item[\qquad $i=2$:] $\quad 0$\ ,
\item[\qquad $i=1$:] $\quad m\cdot \chi_\Pi\delta_B^{-1/2}\otimes \pi\  \quad\ $ as a $TS$-module (see example \ref{example 4}) \ ,
\item[\qquad $i=0$:] $\quad\begin{cases}
            \chi_\Pi^\divideontimes   \qquad\qquad\quad   & \text{if } \pi=\mu\circ\det \text{ and }\rho=\mu\chi_\Pi \ ,\\ 
            0          \qquad\qquad\quad   & \text{otherwise}\ .
                       \end{cases}$
\end{enumerate}
\end{lem}
\begin{proof}
The proof is analogous to lemma~\ref{tildemod} and uses section~\ref{s_app:filtration}.
\end{proof}

Note that for $\dim(\pi)\!\neq\! 1$ and $i\!\neq\! 1$, all $\beta^\rho(I_i)$ above vanish.
If $\dim(\pi)\!=\! 1$  or more generally if $(\sigma_\Pi,\Lambda)$ is ordinary,
then all  $\beta^\rho(I_i)$ with $i\neq3$ are finite dimensional.

\textit{Bessel filtration.}
Under the Bessel functor $\beta_\rho$ the inclusion maps
$I_{\geq i} \! \subseteq\! I$ induce morphisms $\widetilde {I_{\geq i}} \to \widetilde I$
since $\beta_\rho$ is only right exact.
The images $$F_{\geq i} := im(\widetilde{I_{\geq i}} \to \widetilde I)$$
define a filtration    
$0\ \subseteq\ F_{\geq 3}\ \subseteq\ F_{\geq 2}\ \subseteq\ F_{\geq 1}\ \subseteq\ F_{\geq 0}\ =\ \widetilde I \ $
of $\tilde I$ whose graded objects
$ gr^F_i(\widetilde I)\! =\! F_{\geq i}/F_{\geq i +1} $
are quotients  of  the $TS$-modules $\widetilde{I_i}$ in lemma \ref{tildemod}:
$$    \widetilde{I_i}  \twoheadrightarrow gr^F_i(\widetilde I) \ .$$
Under the exact functor $\pi_0$,  $\pi_0(\widetilde I)$ becomes filtered
by $\pi_0(F_{\geq i})$ with graded components $\pi_0(gr^F_i(\widetilde I))$ and
these are quotients of the $\pi_0(\widetilde{I_i})$. 
The monodromy operator $\tau_\pi$ acts on $\pi_0(\widetilde I)$ and respects the filtration
in the following sense: 
$\tau_\pi(\pi_0(\widetilde{I_i})\! \subseteq\! \pi_0(\widetilde{I_{i-1}})$ holds 
with the possible exceptions $i\!=\!2,1$, where sometimes only $\tau_\pi^2(\pi_0(\widetilde{I_i})\! \subseteq\! \pi_0(\widetilde{I_{i-1}})$ could hold, namely in the degenerate case where $\pi\! =\! \chi_1\! \times\! \chi_2$ for $\chi_1 \! =\! \chi_2$. 

\bigskip\noindent
The Bessel functors applied to $I_{\leq i}\!:\! =\! I/I_{\geq i +1}$ and the exact sequences
$$ 0 \to I_{\geq 2} \to  I \to  I_{\leq 1} \to 0  $$
as well as
$  0 \to {I_{1}} \to I_{\leq 1} \to  {I_0} \to 0 $ and  
$  0 \to {I_{3}} \to I_{\geq 2} \to  {I_2} \to 0 $ yield a diagram
\begin{equation*}
\xymatrix@R-3mm{
  & 0\ar[d]           & & 0\ar[d]             & 0\ar[d]                      &     &     &   \\
 & \beta^\rho(I_3)\ar[d]& & \beta^\rho(I_1)\ar[d]  & \widetilde{I_3}\ar[d] &                    &\widetilde{I_1} \ar[d]      &   \\
 0\ar[r] & \beta^\rho(I_{\geq2})\ar[r]\ar[d]& \beta^\rho(I)\ar[r] &   \beta^\rho(I_{\leq1})\ar[r]^{\delta}\ar[d] & \widetilde{I_{\geq2}}\ar[d]\ar[r] & \widetilde{I}\ar[r]  &\widetilde{I_{\leq1}}\ar[r]\ar[d] & 0\ .\\
        &  0       & & \beta^\rho(I_0)       & \widetilde{I_2}\ar[d]   &    &\widetilde{I_0}\ar[d]       &   \\
  &                  & &                    & 0                        &    &0    &   \\
}
\end{equation*}
\begin{lem}\label{Lemma1}\label{lem:filtration_diagram_exact}
The sequences in the diagram are exact.
If $\pi\not\cong(\nu^{-1/2}\circ\det)$, the sequences
$0\to \widetilde{I_{1}} \to \widetilde I_{\leq 1} \to  \widetilde{I_0} \to 0$ and
$0\to \beta^\rho(I_{1})\to \beta^\rho (I_{\leq 1}) \to  \beta^\rho(I_0)\to 0$ are exact.
\end{lem}

\begin{proof}
The first assertion holds by lemma \ref{lem:long exact sequence}, since $\beta^\rho(I_2)=0$ by lemma~\ref{tildeKoh}.
It remains to be shown that the boundary morphism
$$ \delta_2: \beta^\rho(I_0) \longrightarrow {\widetilde I_1} $$ vanishes for $\pi\not\cong\nu^{-1/2}\circ\det$.
Indeed, assume $\delta_2$ is non-zero. %
By lemma~\ref{tildeKoh} this is only possible if $\pi \! = \! \mu \circ\det$ and $\rho=\chi_\Pi\mu$ for a smooth character $\mu$ of $k^\times$.
But then lemma~\ref{tildemod} yields
$\beta^\rho(I_0) \cong \delta_P^{1/2}\chi_\Pi^\divideontimes$ and
$$\widetilde{I_1} \cong m\cdot \delta_P^{1/2} \delta_{B}^{-1/2}\chi_\Pi\otimes \pi
 = m\cdot\delta_P^{1/2}\nu^{-1/2}\chi_\Pi \mu  = m\cdot\delta_P^{1/2}\nu^{-1/2}\rho\ .$$
Since both sides are one-dimensional, $\delta_2\neq 0$ implies $\rho = \nu^{1/2}\chi_\Pi^\divideontimes$.
The central character is $\omega=\mu^2\chi_\Pi^2=\rho^2$, so this means $\mu=\nu^{-1/2}$.
\end{proof}

If $\pi\cong(\nu^{-1/2}\circ\det)$, the unique quotient $\Pi$ of $I$ is of type \nosf{VIb} and does not admit a split Bessel model.
In the following we tacitly assume $\pi\not\cong(\nu^{-1/2}\circ\det)$.

\begin{lem} \label{KEY} Suppose $(\sigma_\Pi,\rho)$ is ordinary.
Then $\widetilde I$ has degree one and $\widetilde I_3$ injects
$$ gr^F_3(\widetilde I)=\widetilde I_3 \cong   \delta_P^{1/2} \otimes  \mathbb E[\chi_\Pi] \ \ \ \hookrightarrow\ \ \ \widetilde I \ .$$ 

If $(\sigma_\Pi,\rho)$ is non-ordinary, $\widetilde{I}$ has degree one or two.
The image $gr_3^F(\widetilde I)$ of $\widetilde{I_3}$ is 
\begin{equation*}
gr_3^F(\widetilde I)\cong
\begin{cases}
\delta_P^{1/2}\otimes\chi_\Pi             & \deg(\widetilde{I})\!=\!1\ ,\\
\delta_P^{1/2}\otimes\mathbb{E}[\chi_\Pi] & \deg(\widetilde{I})\!=\!2\ .                                                                                                                                                                                                                                                                                                                  \end{cases}
\end{equation*}
\end{lem}

\begin{proof}
In the ordinary case, for $i\neq 3$ the modules $\beta^\rho(I_i)$ and $\widetilde{I_i}$ are of finite dimension by lemma \ref{tildeKoh} and lemma \ref{tildemod}.
On the other hand $\deg(\widetilde{I_3})\!=\!
\deg(C^\infty_c(S))\! =\! 1$ (example~\ref{example 3}).
By a diagram chase it is then easy
to see $\deg(\widetilde I)\! =\! \deg(\widetilde{I_3})\!=\!1$ and
the first claim follows. In particular, the image $\widetilde{I_3} \to gr^F_3$ must
have degree one.
Since $\widetilde{I_3} = C_c^\infty(S)=\mathbb{E}[\chi_\Pi]$ is perfect by lemma \ref{tildemod}
and example~\ref{example 3}, the map $\widetilde{I_3} \to \widetilde{I}$
is injective by lemma~\ref{SCREEN}.

In the non-ordinary case, the diagram of lemma~\ref{lem:filtration_diagram_exact} implies 
\begin{equation*}1=\deg(\widetilde{I}_{\leq1})\leq \deg(\widetilde{I})\leq \deg(\widetilde{I}_{\leq1})+\deg(\widetilde{I}_{\geq2})=2\ .\end{equation*}
By proposition~\ref{BETA1}, we have $\deg(\beta^\rho(I))=\deg(\widetilde{I})-1$.
If $\widetilde{I}$ has degree one, then $\beta^\rho(I)$ is finite-dimensional and perfect by lemma~\ref{lem:beta^rho_perfect} and thus vanishes.
The boundary map $\delta:\beta^\rho (I_{\leq1})\to \widetilde{I_{\geq2}}$ is then an injection.
By lemma~\ref{tildemod}, the cokernel of $\delta$ is the one-dimensional character $\pi_0(\widetilde{I_3})\cong \delta_P^{1/2}\chi_\Pi$ of $T$.
If $\widetilde{I}$ has degree two, then the image of $\widetilde{I}_3$ in $\widetilde{I}$ must have degree one and lemma~\ref{screening} provides the embedding.
\end{proof}

\begin{lem}\label{lem:key2}
Suppose $(\sigma_\Pi,\rho)$ is extraordinary. Then $\widetilde{I}$ has degree two.
Further, if $I$ admits a generic irreducible quotient $\Pi$, there is an embedding of $TS$-modules $\nu\chi_\Pi\otimes\pi\hookrightarrow \widetilde{\Pi}/{\Danismanfunctor}(\widetilde\Pi)$ with the $TS$-structure on $\pi$ as in example~\ref{example 4}.
\end{lem}

\begin{proof}
The filtration index is $m=2$, so $\widetilde{I_{1}}$ and thus $\widetilde{I_{\leq1}}$ have degree two.
The degree of $\widetilde{I}$ cannot exceed two by theorem~\ref{Mult1},
so the image of $\widetilde{I_{\geq2}}$ in $\widetilde{I}$ is finite-dimensional.
By lemma~\ref{cor:gen_Mult1} $\widetilde{\Pi}$ has degree one, so the composition $\widetilde{I_1}\to \widetilde{I_{\leq1}}\to \widetilde{I}/{\Danismanfunctor}(\widetilde{I})\to\widetilde{\Pi}/{\Danismanfunctor}(\widetilde\Pi)$ is non-trivial.
By lemma \ref{screening} and Kirillov theory,
one of the two copies of $\nu\chi_\Pi\otimes\pi$ in $\widetilde{I_1}$
injects into $\widetilde{\Pi}/{\Danismanfunctor}(\widetilde\Pi)$ .
\end{proof}

\begin{lem}\label{lem:beta^rho_perfect}
If $\sigma_\Pi$ is not one-dimensional, then $\beta^\rho(I)$ is perfect for every $\rho$.
Further, $\beta^\rho(\Pi)$ is perfect for every irreducible submodule or quotient $\Pi$ of $I$.
\end{lem}

\begin{proof}
By lemma \ref{tildeKoh}, we have $\beta^{\rho}(I_i)=0$ for every $i\neq1$ 
and ${\Danismanfunctor}(\beta^\rho(I_1))=0$ by Kirillov theory. This implies ${\Danismanfunctor}(\beta^\rho(I))=0$ by left-exactness of ${\Danismanfunctor}$.
By possibly replacing $I$ with $\omega\otimes I^\vee$, we can assume that $\Pi$ is a submodule of $I$ by theorem~\ref{thm:sigma_Pi_existence}. Then $\beta^\rho(\Pi)$ is a submodule of $\beta^\rho(I)$ and thus perfect.
\end{proof}

\begin{lem}\label{lem:beta^rho=0_generic}
If $\Pi\in\CCC_G(\omega)$ is generic irreducible, then  $\beta^\rho(\Pi)=0$ for all $\rho$.
\end{lem}
\begin{proof}
$\beta^\rho(\Pi)$ has degree zero by corollary~\ref{20.7}, so it is finite-dimensional.
If $J_P(\Pi)\neq0$, there is an irreducible $\sigma_\Pi\in\CCC_M$, not one-dimensional, such that $\Pi$ is a quotient of $\Ind_P^G(\sigma)$, see table~\ref{tab:list_sigma}.
By lemma~\ref{lem:beta^rho_perfect}, $\beta^\rho(\Pi)$ is perfect and thus zero.
If $J_P(\Pi)=0$, then proposition~\ref{lasst} and proposition~\ref{BETA1} imply the assertion.
\end{proof}

Now lemma~\ref{lem:long exact sequence} easily implies
\begin{cor}\label{Addition}
An exact sequence $0\to \Xi\to I\to\Pi\to 0$ in $\CCC_G^{\fin}(\omega)$ with irreducible generic quotient $\Pi$ gives an exact sequence $0\to\widetilde{\Xi}\to\widetilde{I}\to\widetilde{\Pi}\to0$ of $TS$-modules,
and especially
$\deg(\widetilde I) = \deg(\widetilde\Xi) + \deg(\widetilde \Pi)$.
\end{cor}

\subsection{Siegel-Jacquet module}\label{s:Siegel-Jacquet}

For irreducible $\Pi\in {\cal C}_G(\omega)$ 
the normalized Jacquet module with respect to the Siegel parabolic subgroup $P=MN$
$$E := \delta_P^{-1/2}\otimes J_P(\Pi)= \delta_P^{-1/2}\otimes \Pi_N \ $$
defines an admissible module of finite length of the Levi subgroup $M$ of $P$.
The Levi group $M$
decomposes
$$ M =  M^1 \times \{t_\lambda\mid\lambda\in k^{\times}\}$$
where  $M^1=M\cap \Sp(4)$ is isomorphic to $\Gl(2)$ via $m_A \mapsto A$.
In this section we describe $J_P(\Pi)_{T,\chi}$ as a smooth $\Gl(2)$-module for characters $\chi$ of $T$.
The irreducible Jordan-H\"older constituents of $E$ are of the form $$\sigma_\Pi=\pi \boxtimes \chi_\Pi$$
for $\sigma_\Pi(t_\lambda m_A)= \pi(A)\chi_\Pi(\lambda)$ with central character $\omega=\omega_\pi\chi_{\Pi}^2$.

\begin{rmk}\label{rmk:torus_action_on_Jacquet}
$\tilde T$ embeds into $M$ via $\tilde t= (t_1,t_2,t_2,t_1) = z t_\lambda m_A$
for $z=t_2\cdot\id$, $\lambda = t_1/t_2$ and $A= \diag(t_1/t_2,1)$.
It acts on $\pi\boxtimes\chi_{\Pi}$ by 
$$\widetilde{t}\mapsto
(\chi_\Pi\otimes\pi)(\diag(t_1,t_2))=\omega(t_2)\chi_{\Pi}(t_1/t_2)\pi(\diag(t_1/t_2,1))\ .$$
The torus $T$ embeds into $M$ via $x_\lambda= m_{\diag(\lambda,\lambda)} t_\lambda$, so on $\pi\boxtimes\chi_{\Pi}$ it acts by $$x_\lambda\mapsto \omega_\pi\chi_{\Pi}(\lambda)=\chi_\Pi^\divideontimes(\lambda)\ .$$
\end{rmk}
The multiset of smooth $T$-characters $\chi_{\norm}$ arising from the irreducible constituents of $E$ is called $\Delta(\Pi)$ in section~\ref{s:Combinat}.
For the unnormalized Jacquet modules $J_P(\Pi)$ the character $\chi_{\norm}$
must be replaced by the corresponding unnormalized character $\chi = \nu^{3/2}\chi_{\norm}$. 
The operation of $\widetilde T$ does not depend on the normalization because $\delta_P(\widetilde t)=1$ for every $\widetilde{t}\in\widetilde{T}$.

\begin{lem} \label{mysterious}
If a non-generic irreducible $\Pi\in\CCC_G(\omega)$ admits a split Bessel model for a Bessel character $\rho$,
then $\rho=\mu\chi_\Pi$ and $\omega=(\mu\chi_{\Pi})^2$ holds
for every one-dimensional Jordan-H\"older constituent
$\sigma_\Pi=(\mu\circ\det)\boxtimes\chi_{\Pi}$ of $E$.
\end{lem} 
\begin{proof} By a twist we can assume that $\Pi$ occurs in table \ref{tab:List}.
The assertion $\omega=(\mu\chi_{\Pi})^2$ is clear because the center of $G$ acts on every constituent of $E$ by $\omega$.
By \cite[table\,A.3]{Roberts-Schmidt} and theorem \ref{thm:sigma_Pi_existence}, $\Pi$ is a quotient of the Siegel induced representation $I=\Ind_P^G((\mu^{-1}\circ\det)\boxtimes\chi_\Pi\mu^2)$. If $\rho\neq \mu\chi_\Pi$, then lemma~\ref{tildemod} and lemma~\ref{lem:filtration_diagram_exact} implies that $\widetilde{I}$ and thus its quotient $\widetilde{\Pi}$ are finite-dimensional. This contradicts the assumption that $\rho$ provides a split Bessel model for $\Pi$.
\end{proof}

\begin{cor}
Suppose $\Pi\in {\cal C}_G(\omega)$ is irreducible and not a twist of \nosf{VIa}.
If $\Pi$ has a split Bessel model associated to the Bessel character $\rho$, the $\rho$-coinvariant space
$\sigma_\rho$ is one-dimensional for all Jordan-H\"older constituents $\sigma$ of $E$.
\end{cor}
\begin{proof}
Remark~\ref{rmk:torus_action_on_Jacquet} implies %
that $\sigma_\rho=\sigma_{\widetilde{T},\rho}$
is the space of $\rho$-coinvariants under the $\tilde T$-action
defined by $\widetilde{t}\mapsto\chi_{\Pi}\otimes \pi(\diag(t_1,t_2))$.
Hence, by corollary~\ref{cor:Waldspurger-Tunnell},
$$\dim(\sigma_\rho)=\begin{cases}0 & \pi\!=\!\mu\!\circ\!\det\ \text{and}\ \rho\!\neq\!\mu\chi_\Pi\ ,\\1&\text{otherwise.}\end{cases}$$
For non-generic $\Pi$, lemma \ref{mysterious} implies $\dim(\sigma_\rho)=1$.
The only generic $\Pi$ with a one-dimensional constituent in $E$ are the twists of type \nosf{VIa}.
\end{proof}

\bigskip\noindent
\textit{Monodromy filtration}. $T$ and $M^1$ commute, so $E$ can be considered as an object 
in the abelian category ${\cal C}_{M^1}^{\fin}$ endowed with $T$-action.  
For a smooth character $\chi$ of $T$ we define the maximal quotient $E_\chi \in {\cal C}_{M^1}^{\fin}$,
and the maximal subspace $E^\chi\in {\cal C}_{M^1}^{\fin}$ on which $T$ acts by $\chi$, respectively.
Recall that the generalized $\chi$-eigenspace $E^{(\chi)}\subseteq E$ is the subspace of elements $v$ on which $(t-\chi(t))^N v=0$ holds for all $t\in T$ and all $N > length(E)$. Since $T$ and $M^1$ commute,
this gives a decomposition as a direct sum of $M$-modules
$$E \cong \bigoplus_{\chi} E^{(\chi)}\ .$$
On each $E^{(\chi)}$ the compact group $T(\mathfrak{o}) \cong \mathfrak{o}^{\times}$ acts by the character $\chi$. 
Since $T(\mathfrak{o}) \cong k^{\times} = \mathfrak{o}^{\times} \times \pi^\Z$, we want to determine 
the action of the monodromy operator $\tau_\chi = \pi - \chi(\pi)$ on $E^{(\chi)}$.
The morphism $\tau_\chi$ is nilpotent of length $\leq length(E)$.
There exists a unique finite increasing filtration $F_\bullet$ of $V=E^{(\chi)}$ by $M^1$-submodules $F_i(V) \subseteq F_{i+1}(V)$ with the property:
$$ \tau_\chi(F_i(V)) \subseteq F_{i-2}(V) \quad , \quad \tau_\chi^i : Gr_i(V) \cong Gr_{-i}(V) $$
for all $i\geq 0$, where $Gr_i(V)= F_i(V)/F_{i-1}(V)$.
Put $P_i=\ker(\tau_\chi: Gr_i(V) \to Gr_{i-2}(V))$. Then $P_i(V)=0$ for $i>0$
and $Gr_i(V) \cong \bigoplus_{j\geq \vert i \vert, j\equiv i(2)} P_{-j}$ in ${\cal C}_{M^1}^{\fin}$. 
The kernel of $\tau_\chi$ is $E^\chi \subseteq E^{(\chi)}$. For the induced filtration on  $E^\chi$ we have $Gr_i(\ker(\tau_\chi)) \cong P_i(V)$. Similarly, for the induced increasing filtration on the cokernel $E_\chi$ of $\tau_\chi$ defined by $im(F_i(V)) \subseteq E_\chi$
 one has  $Gr_i(E_\chi) \cong P_{-i}(V)$;  
see e.g.~\cite[p.56]{Kiehl-Weissauer}, where the Tate twist can be omitted in our situation.
Notice that every constituent of $P_{-i}(V)$ has multiplicity $\geq i+1$ in $V$.

\textit{Generic constituents $\pi$}. For fixed characters $\rho$ of $\tilde T$ and $\chi$ of $T$  we have isomorphisms between coinvariants spaces  %
$$(J_P(\Pi)_\chi)_\rho \cong (\widetilde\Pi_S)_\chi \cong \pi_0(\widetilde{\Pi})_\chi $$
so that 
$  \dim((E_{\chi_{\norm}})_\rho) = \dim(\pi_0(\widetilde \Pi)_\chi)$
holds for $\widetilde \Pi = \beta_\rho(\Pi)$.
Now let $\rho$ vary. The Jordan-H\"older
constituents $\pi$ of $E_{\chi_{\norm}}$ define irreducible representations $\pi$ of $M^1\cong\Gl(2)$.
If $\dim(\pi)=1$, then $\pi_\rho \cong \pi^\rho$ vanishes for almost all $\rho$.
If $\dim(\pi)\neq 1$, then $\dim(\pi_\rho)=1$ and $\dim(\pi^\rho)=0$
for all smooth characters $\rho$,
see example~\ref{ex:Kirillov} and corollary~\ref{cor:Waldspurger-Tunnell}.
Hence for \emph{almost all} $\rho$,
$$\dim E_{\rho,\chi_{\norm}} = \#\{ \text{generic constituents}\ \pi\ \text{of}\ E_{\chi_{\norm}}  \}\ .$$
By the analogous argument, uniqueness of Whittaker models implies
for all characters $\chi_{\norm}$ of $T$ and additive characters $\psi\neq1$ of $S$ that 
\begin{equation*}
\dim (E_\psi)_{\chi_{\norm}} = \dim (E_{\chi_{\norm}})_\psi =  \#\{\text{generic constituents}\ \pi\ \text{of}\ E_{\chi_{\norm}} \} \ .
\end{equation*}

\begin{lem}\label{lem:gen_const_E_chi}
For generic $\Pi$ and every $T$-character $\chi$, there is at most one generic constituent $\pi$ in $E_{\chi_{\norm}}$.
\end{lem}
\begin{proof}
For generic $\Pi$ we have shown in lemma~\ref{lem:alpha_perfect} that $\alpha(\Pi)=j^!(\overline{\Pi}) \in {\cal C}$
is perfect.
We have $ \pi_0(\alpha(\Pi)) \cong J_P(\Pi)_\psi $ by lemma~\ref{Dual action}.
Now lemma~\ref{lem:equivalences2} implies that the number of generic constituents is
\begin{equation*}
\dim (E_{\chi_{\norm}})_\psi=\dim(J_P(\Pi)_\psi)_\chi = \dim(\pi_0(\alpha(\Pi)))_\chi \leq 1 \ .\qedhere
\end{equation*}
\end{proof}

\begin{lem}\label{lem:first_monodromy_VIa}
If $\Pi$ is of type \nosf{VIa}, normalized as in table \ref{tab:List}, and $\chi_{\norm} = \nu^{1/2}$, then
$E=E^{(\chi_{\norm})}$ has three irreducible $M^1$-constituents, where $\nu^{1/2}\circ \det$ occurs once and $\Sp(\nu^{1/2})$ occurs twice.
The coinvariant space $E_{\chi_{\norm}}$ is an indecomposable extension
$ 0 \to (\nu^{1/2} \!\circ\! \det) \to E_{\chi_{\norm}} \to \Sp(\nu^{1/2}) \to 0$,
so it is isomorphic to $1\times\nu$ as a $\Gl(2)$-module.
\end{lem}
\begin{proof}
The constituents are given by \cite[table~A.3]{Roberts-Schmidt}.
Since $E=E^{(\chi_{\norm})}$ has three
irreducible $M^1$-constituents and two of them are generic, lemma~\ref{lem:gen_const_E_chi}
implies $E_{\chi_{\norm}}\neq E$, i.e. the monodromy filtration is nontrivial 
in the sense that $\tau_{\chi_{\norm}}$ acts nontrivially, so $E \neq Gr_0(E)$.
Now, since $Gr_i(E_{\chi_{\norm}}) \cong P_{-i}(E)$ and every constituent of $P_{-i}(E)$
has multiplicity $\geq i+1$ in $E$,
the only possibility is $P_0(E)= \nu^{1/2}\circ\det$ and $P_{-1}(E)= \Sp(\nu^{1/2})$.
Since these primitive modules describe the induced increasing filtration on $E_{\chi_{\norm}}$
defined by $im(F_i(E)) \subseteq E_{\chi_{\norm}}$ with $F_{0}(E_{\chi_{\norm}}) \subseteq F_1(E_{\chi_{\norm}})$, we obtain 
$F_{0}(E_{\chi_{\norm}})=Gr_{0}(E) \cong \nu^{1/2}\circ \det$ and $F_1(E_{\chi_{\norm}})/F_0(E_{\chi_{\norm}})\cong P_1(E) \cong \Sp(\nu^{1/2})$.
Hence our claim follows that $E_{\chi_{\norm}}$ has two constituents
sitting in an exact sequence as above.
It remains to be shown that the sequence does not split.
Indeed, if this were not true, then
$$\Hom_{M}(E, (\nu^{\frac{1}{2}}\circ\det)\boxtimes \chi_{\Pi})\cong
\Hom_{M^1}(E_{\chi_{\norm}}, (\nu^{\frac{1}{2}}\circ\det)) \ $$  
would not vanish for $\chi_{\Pi}=\chi_{\norm}^\divideontimes=\nu^{-1/2}$.
By dual Frobenius reciprocity
$\Hom_G(\Pi, \Ind_P^G( (\nu^{\frac{1}{2}}\circ\det)\boxtimes \nu^{-1/2})$ would be non-trivial.
The irreducible constituents of
$\Ind_P^G( (\nu^{\frac{1}{2}}\circ\det)\boxtimes \nu^{-1/2})$ are of type \nosf{VId} and \nosf{VIb}
and thus not isomorphic to $\Pi$.
This is a contradiction, so $E_{\chi_{\norm}}$ is indecomposable.
The unique indecomposable extension is the $\Gl(2)$-module
\begin{equation*}
E_{\chi_{\norm}} \cong \nu^{1/2}\otimes  (\nu^{-1/2} \times \nu^{1/2})= 1\times \nu\ .\qedhere
\end{equation*}
\end{proof}

\begin{lem}\label{lem:filtration_VId}
For irreducible $\Pi\in\CCC_G(\omega)$ of type \nosf{VId} as in table \ref{tab:List},
we have $E=E^{(\chi_{\norm})}$
for $E=\delta_P^{-1/2}\otimes J_P(\Pi)$ and the $T$-character $\chi_{\norm}=\nu^{-1/2}$.
The graded components\footnote{Compare \cite[table A.5]{Roberts-Schmidt}.} of the monodromy filtration
are the $M^1\cong\Gl(2)$-modules $$Gr_1(E)\cong (\nu^{-1/2}\circ\det),\ Gr_0(E)\cong \Sp(\nu^{-1/2}),\ Gr_{-1}(E)\cong (\nu^{-1/2}\circ\det)\ ,$$
so that $E_{\chi_{\norm}}$ is a non-split extension of $M^1$-modules
$$ 0 \to \Sp(\nu^{-1/2}) \to E_{\chi_{\norm}} \to \nu^{-1/2}\circ\det \to 0 \ $$
and is thus isomorphic to the $\Gl(2)$-module $E_{\chi_{\norm}}\cong 1\times \nu^{-1}$.
\end{lem}

\begin{proof}%
It is clear that $E=E^{(\chi_{\norm})}$ by table A.3 of \cite{Roberts-Schmidt}.
We claim that the monodromy operator $\tau_\chi$ acts nontrivially on $J_P(\Pi)$.
This amounts to show that $T$ does not act by a central character on $J_P(\Pi)$,
although its acts on each of its three constituents by the same character.
Indeed, $J_P(\Pi)\in {\cal C}_{M}^{\fin}$ admits the unnormalized Jacquet quotient $J_P(\Pi)_U$ with respect to the subgroup
$U=\{ m_A\vert A=(\begin{smallmatrix} 1 & * \cr 0 & 1\end{smallmatrix})
\}$, which is isomorphic to the Borel Jacquet module $J_{B_G}(\Pi)$.
However, $J_{B_G}(\Pi)$ is also isomorphic to the coinvariant space $F_{U_Q}=i^*i^*(\overline{\Pi})$ of the Klingen-Jacquet module $F=J_Q(\Pi)\cong i^*(\overline\Pi)\in {\cal C}_{\Gl(2)}$
with respect to the maximal unipotent subgroup $U_Q\subseteq L_Q\cap B_G$ for the standard Klingen Levi subgroup $L_Q$ of $Q$. For details see section~\ref{functor eta}.
For $\Pi$ of type \nosf{VId}, $F$ admits two Jordan-H\"older factors,
see \cite[table A.4]{Roberts-Schmidt}.
Since their central characters as $\Gl(2)$-modules are different, $F$ splits as a direct sum of $\Gl(2)$-modules
$$ F \cong \bigl(\nu\circ\det\bigr)\ \oplus\ \bigl(\nu^{1/2} \otimes(1\times 1)\bigr) \ .$$
By remark~\ref{rmk:torus_action_on_Jacquet}, $T$ acts on $\nu^{1/2}(1\times 1)_{U_Q}$ via 
$x_\lambda\mapsto(\nu^{1/2} \times \nu^{1/2})_{U_Q}(\diag(\lambda,1))$.
Since $(1\times 1)_{U_Q}$ is not semisimple (\cite{Bump}, thm.\ 4.5.4), $T$ does not act semisimply on $J_{B_G}(\Pi)$.
On the other hand, $T$ is in the center of $M$ and acts by the same character $\chi$
on the three constituents
of $J_P(\Pi)$.
This implies that the monodromy operator $\tau_\chi$ acts nontrivially on $J_P(\Pi)$ which proves our claim.

That the filtration of $J_P(\Pi)_\chi$ and the graded components $Gr_i(J_P(\Pi))$ are of the required form 
is shown as in the case \nosf{VIa}.
In particular we obtain from this an exact sequence
\begin{equation*} 0 \to \Sp(\nu^{-1/2}) \to J_P(\Pi)_\chi \to (\nu^{-1/2}\circ\det) \to 0 \ .\end{equation*}
But if $J_P(\Pi)_\chi$ was semisimple, then
$$\Hom_{M}(\delta_P^{-1/2}\otimes J_P(\Pi), \Sp(\nu^{-1/2})\boxtimes \chi_{\Pi})\cong
\Hom_{M^1}(J_P(\Pi)_\chi, \Sp(\nu^{-\frac{1}{2}})) \neq 0\ $$
for $\chi_\Pi=\chi_{\norm}^\divideontimes=\nu^{1/2}$.
Hence there is a non-zero $G$-morphism from $\Pi$ to $\Ind_P^G(\Sp(\nu^{-1/2})\boxtimes \nu^{1/2})$ by dual Frobenius reciprocity.
However, the two irreducible constituents of $\Ind_P^G(\Sp(\nu^{-1/2})\boxtimes \nu^{1/2})$ are of type \nosf{VIa} and \nosf{VIc}, hence not isomorphic to \nosf{VId}.
This is a contradiction, so $J_P(\Pi)_\chi$ is indecomposable as an $M^1$-module and thus isomorphic to $(1\times \nu^{-1})$ as a $\Gl(2)$-module.
\end{proof}

\begin{thm} \label{Firstmonodromy}
For irreducible representations $\Pi\in {\cal C}_G(\omega)$ and
characters $\chi_{\norm}$ of $T$, the $M^1$-module $E_{\chi_{\norm}}\in {\cal C}_{M^1}^{\fin}$
is irreducible or zero except for the following two cases, normalized as in table \ref{tab:List}, and their twists:
\begin{enumerate}
\item[$\quad $ \nosf{ VIa}] \ \ where $J_P(\Pi)_\chi \cong\, 1\times \nu$ as an $M^1$-module for $\chi_{\norm}=\nu^{1/2}$,
\item[$\quad $ \nosf{ VId}] \ \ where $J_P(\Pi)_\chi \cong\, 1\times \nu^{-1}$  as an $M^1$-module for $\chi_{\norm}=\nu^{-1/2}$.
\end{enumerate}
\end{thm} 

\begin{proof}By a twist it is sufficient to assume that $\Pi$ is normalized as in table \ref{tab:List}.
For generic $\Pi$, the table of irreducible constituents \cite{Roberts-Schmidt}, table A.3,
implies that every constituent of the $M^1$-module $E_{\chi_{\norm}}$ is generic except for case \nosf{VIa} with $\chi_{\norm}=\nu^{1/2}$.
The statement is then implied by lemma~\ref{lem:gen_const_E_chi}. For the remaining case \nosf{VIa} the assertion is shown in lemma~\ref{lem:first_monodromy_VIa}.
For non-generic $\Pi$ the theorem follows
from an inspection of table A.3 of \cite{Roberts-Schmidt} and table~\ref{tab:Delta} except for 
case \nosf{IIb} with $\chi_1=1$ and for case \nosf{VId}.
For case \nosf{VId} with $\chi_{\norm}=\nu^{-1/2}$ the statement is shown in lemma~\ref{lem:filtration_VId}.
For the case \nosf{IIb}, $\Pi=\Ind_P^G((\chi_1\circ\det)\boxtimes\chi_1^{-1})$ and the constituents of
$E=\delta_P^{-1/2}\otimes J_P(\Pi)$ are
\begin{enumerate}
\item $(\chi_1\nu^{-1/2} \times \chi_1^{-1}\nu^{-1/2})\boxtimes \nu^{1/2}$,
\item $(\chi_1 \circ\det) \boxtimes \chi_1^{-1}$,
\item $(\chi_1^{-1} \circ\det) \boxtimes \chi_1$.
\end{enumerate} Now assume $E_{\chi_{\norm}}$ has length greater than one.
Since $\chi_1^2 \neq \nu^{\pm 1}$ by definition,
this implies $\chi_{\norm}=\chi_1=\chi_1^{-1}$.
Dual Frobenius reciprocity yields
\begin{gather*}  \Hom_G(\Pi, \Ind_P^G((\chi_1\!\circ\!\det)\boxtimes \chi_1^{-1})
 \cong \Hom_{M}(E, (\chi_1\!\circ\!\det)\boxtimes\chi_1^{-1}) \\
 \cong \Hom_{M^1}(E_{\chi_{\norm}}, (\chi_1\!\circ\!\det))
\end{gather*}
and this is one-dimensional, because $\Pi$ is irreducible.
Especially, $E_{\chi_{\norm}}$ is an indecomposable two-dimensional $M^1$-module.
Every finite-dimensional smooth module of $M^1\cong\Gl(2)$ factorizes over the determinant,
so we obtain an isomorphism $E_{\chi_{\norm}}\cong (\chi_1^{(2)}\circ \det)$ where $\chi_1^{(2)}$ is a Jordan-block of size two attached to $\chi_{1}$. Especially, the center $\{\diag(\lambda I_2,\lambda^{-1} I_2)\mid \lambda\in k^{\times}\}$ of $M^1$ acts by $\lambda\mapsto\chi^{(2)}(\lambda^2)$, thus not semisimply.
On the other hand, the center of $G$ acts semisimply by multiplication with $\omega$,
so the action of $T$ on $E_{\chi_{\norm}}$ cannot be semisimple. This provides a contradiction.
\end{proof}

\subsection{Monodromy theorem}

We now estimate the dimension of $\pi_0(\widetilde{\Pi})_\chi\cong J_P(\Pi)_{\chi,\rho}$ for characters $\chi$ of $T$ and Bessel characters $\rho$ that define a split Bessel model.
The following monodromy theorem is one of the key inputs for theorem~\ref{thm:sigma_tilde_Pi}.

\begin{lem}\label{lem:monodromy_VIa}
For irreducible $\Pi\in \CCC_G(\omega)$ of type \nosf{VIa} with normalization as in table \ref{tab:List},
all Bessel characters $\rho$ and characters $\chi$ of $T$, we have
\begin{equation*}
\dim(\pi_0(\widetilde \Pi)_\chi)  = 
\begin{cases}
0 &   \chi_{\norm} \!\neq\!  \nu^{1/2}\ ,\\
1 &   \chi_{\norm} \! =  \!  \nu^{1/2}\,,\ \rho\!\neq\!1\ ,\\
2 &   \chi_{\norm} \! =  \!  \nu^{1/2}\,,\ \rho\!=\!1\ . 
\end{cases} 
\end{equation*}
\end{lem}

\begin{proof}
For the normalized Siegel-Jacquet module $E=\delta_P^{-1/2}\otimes J_P(\Pi)$,
table~\ref{tab:Delta} implies that $E^{(\chi_{\norm})}=0$ for $\chi_{\norm}\neq\nu^{1/2}$.
For $\chi_{\norm}=\nu^{1/2}$ lemma~\ref{lem:first_monodromy_VIa}
gives a short exact sequence of $\Gl(2)\cong M^1$-modules
$$0 \to (\nu^{1/2}\!\circ\! \det) \to E_{\chi_\norm} \to \Sp(\nu^{1/2}) \to 0\ .$$
By remark~\ref{rmk:torus_action_on_Jacquet}, $\widetilde{T}$ acts on $E_{\chi_\norm}$ by $\widetilde{t}\mapsto\nu^{-1/2}(t_1t_2) E_{\chi_\norm}({\diag(t_1,t_2)})$.
Note that lemma~\ref{lem:proto_long_exact_sequence} gives a long exact sequence 
\begin{equation*}
\cdots\to \Sp^{\rho} \to (1\!\circ\! \det)_{\rho} \to (\nu^{-1/2}\otimes E_{\chi_\norm})_{M^1,\rho} \to \Sp_{\rho} \to 0 \ 
\end{equation*}
of $\rho$-(co)invariant spaces with respect to the action of $m_\diag(t_1,1)$.
Since $\Sp^{\,\rho}=0$ and $\dim( \Sp_\rho)=1$ by corollary~\ref{cor:Waldspurger-Tunnell},
we obtain 
\begin{equation*}
\dim((E_{\chi_\norm})_{\widetilde{T},\rho}) 
=  \dim(1\!\circ\!\det)_\rho+1\ .
\end{equation*}
Obviously, $(1\circ\det)_\rho$ is non-zero if and only if $\rho=1$.
The assertion follows from the isomorphism $\pi_0(\widetilde \Pi)_\chi\cong (E_{\chi_\norm})_{\widetilde{T},\rho}$ of $\C$-vector spaces.
\end{proof}

\begin{lem}\label{lem:monodromy_VId}
For irreducible $\Pi\in\CCC_G(\omega)$ of type \nosf{VId} with normalization as in table~\ref{tab:List}, all Bessel characters $\rho$ and characters $\chi$ of $T$, we have
\begin{equation*}
\dim(\pi_0(\widetilde \Pi)_\chi)  = 
\begin{cases}
0 &   \chi_{\norm} \!\neq\! \nu^{-1/2}\ ,\\
1 &   \chi_{\norm} \! =  \! \nu^{-1/2}\, . 
\end{cases} 
\end{equation*}
\end{lem}
\begin{proof}For $E= \delta_P^{-1/2}\otimes J_P(\Pi)$ and $\chi_{\norm}=\nu^{-1/2}$,
lemma~\ref{lem:filtration_VId} implies $E=E^{(\chi_\norm)}$. Further, $E_{\chi_\norm} \cong (1\times \nu^{-1})$ as $M^1\cong\Gl(2)$-module.
Now consider $E_{\chi_\norm}$ as a $\Gl_a(1)$-module via the embedding $\Gl_a(1)\hookrightarrow\Gl(2)\cong M^1$ as in example~\ref{example 4}.
Since $E_{\chi_{\norm}}$ has a unique Whittaker model, it has degree one as a $\Gl_a(1)$-module.
Lemma \ref{lem:torus_action_GL2} implies $(E_{\chi_{\norm}})^{\Gl(1),\rho}=0$ for every smooth character $\rho$ of $\Gl(1)\subseteq\Gl_a(1)$ since $1\times\nu^{-1}$ does not admit a one-dimensional $\Gl(2)$-subrepresentation.
Lemma~\ref{lem:RR} yields $\dim(\pi_0((\widetilde \Pi)_\chi)) 
= \dim(E_{\chi_{\norm}})_{\widetilde{T},\rho}
= \dim(E_{\chi_{\norm}})_{\Gl(1),\rho}
=1$ for every smooth $\rho$.
\end{proof}

\begin{thm}[Monodromy theorem] \label{WEIRD}\label{thm:monodromy}
For irreducible $\Pi\in {\cal C}_G(\omega)$, all smooth Bessel characters $\rho$
and all smooth characters $\chi$ of $T$ we have
$$  \dim(\pi_0(\widetilde \Pi)_{T,\chi})\leq 1 $$
except for twists of the following case:
If $\Pi$ is of type \nosf{VIa} with normalization as in table~\ref{tab:List} and $(\rho,\chi_{\norm})=(1,\nu^{1/2})$,
then $ \dim(\pi_0(\widetilde \Pi)_{T,\chi})=2$.
\end{thm}

\begin{proof}By a twist we can assume that $\Pi$ is normalized as in table \ref{tab:List}.
Recall that $\pi_0(\widetilde \Pi)_\chi \cong (E_{\chi_{\norm}})_{\widetilde{T},\rho}$ 
as before.
If the $M^1$-module $E_{\chi_{\norm}}$ is irreducible,
corollary~\ref{cor:Waldspurger-Tunnell} implies $\dim(E_{\chi_{\norm}})_{\widetilde{T},\rho}\leq 1$.
Indeed, $E_{\chi_{\norm}}$ is irreducible
by theorem \ref{Firstmonodromy}
except for the cases discussed in lemma~\ref{lem:monodromy_VIa} and lemma~\ref{lem:monodromy_VId}. %
\end{proof}

\begin{cor}\label{cor:T_coinvariants_Bessel}
For irreducible non-generic $\Pi\in\CCC_G(\omega)$ and all smooth characters $\rho$ and $\chi$ of $T$ it holds that $\dim(\widetilde{\Pi})_{T,\chi}\leq1$.
\end{cor}
\begin{proof}
The degree of $\widetilde{\Pi}$ is either one or zero by theorem~\ref{Mult1}.
If the degree is one, $\widetilde{\Pi}$ is perfect by corollary~\ref{cor:sigma_nongeneric}, so the assertion holds by lemma~\ref{PERF}.
If $\widetilde{\Pi}$ has degree zero, then $\widetilde{\Pi}\cong\pi_0(\widetilde{\Pi})$ as $T$-modules and the assertion follows from theorem~\ref{thm:monodromy}.
\end{proof}

\begin{cor}\label{cor:VIa_Bessel_not_perfect}
If $\Pi$ is of type \nosf{VIa}, normalized as in table \ref{tab:List}, and if $\rho=1$,
then $\widetilde\Pi$ is of degree one but not perfect.
\end{cor}
\begin{proof}
Corollary~\ref{cor:gen_Mult1} asserts $\deg(\widetilde\Pi)=1$.
If $\widetilde{\Pi}$ was perfect, lemma~\ref{lem:equivalences2} would imply $\dim\pi_0(\widetilde{\Pi})_\chi\leq1$,
 but this contradicts lemma~\ref{lem:monodromy_VIa}.
\end{proof}

\subsection{Constituents of $\pi_0(\delta_P^{-1/2}\otimes\widetilde{\Pi})$} %

For an irreducible $G$-module $\Pi$ and a Bessel character $\rho$ let $\widetilde{\Pi}$ be the
corresponding Bessel module. Let $E=\delta_P^{-1/2}\otimes J_P(\Pi)$ be the normalized Siegel-Jacquet module.
We determine the multiset $\widetilde\Delta(\Pi)$ of Jordan-H\"older constituents $\chi_{\norm}$ of the $T$-module $\delta_P^{-1/2}\otimes\pi_0(\widetilde \Pi)\cong  E_{\widetilde{T},\rho}$.

\begin{lem} \label{ugly} Suppose $\Pi\in {\cal C}_G({\omega})$ is irreducible, normalized as in table \ref{tab:List} and
$\rho$ is a Bessel character.
If $\rho$ defines a Bessel model for $\Pi$,
then
\begin{equation*}\widetilde{\Delta}(\Pi)=
\begin{cases}
 \{\nu^{1/2},\nu^{1/2}\}  & \text{for type \nosf{VIa}}\ ,\\
 \{\nu^{-1/2},\nu^{-1/2}\} & \text{for type \nosf{VId}}\ ,\\
 \Delta(\Pi) & \text{otherwise}\ .
 \end{cases}
\end{equation*}
If $\rho$ does not define a Bessel model for $\Pi$, then
\begin{equation*}
\widetilde{\Delta}(\Pi)=
\begin{cases}
 \{\nu^{-3/2}\}          & \text{for type \nosf{IVd}}\ ,\ \rho\!=\!1\ ,\\
 \{\chi_0\nu^{-1/2}\}    & \text{for type \nosf{Vd}} \ ,\ \rho\!=\!1\ ,\\
 \{\nu^{-1/2}\}          & \text{for type \nosf{Vd}} \ ,\ \rho\!=\!\chi_0\ ,\\
 \{\nu^{1/2}\}          & \text{for type \nosf{VIb}}\ ,\ \rho\!=\!1\ ,\\
 \Delta_0(\Pi)           & \text{otherwise}\ .
\end{cases}
\end{equation*}
\end{lem}

\begin{proof}

By proposition~\ref{lasst} we can assume $J_P(\Pi)\neq0$.

If $\Pi$ is generic, every $\rho$ provides a Bessel model by corollary~\ref{cor:gen_Mult1}.
Lemma~\ref{lem:beta^rho=0_generic} implies $\beta^\rho(\Pi)=0$, so proposition~\ref{BETA1} yields $[\widetilde{\Pi}]=[\alpha(\Pi)]$ in the Grothendieck group $K_0(\CCC)$ and this means $\widetilde{\Delta}(\Pi)=\Delta_0(\Pi)$ by exactness of $\pi_0$.
The multiset $\Delta_0(\Pi)$ of constituents of the $T$-module $E_\psi$ is given by the $T$-action on the generic constituents of $E$, see table A.3 \cite{Roberts-Schmidt}. Except for type \nosf{VIa}, every constituent of $E$ is generic.

Now suppose $\Pi$ is non-generic, but not of type \nosf{IIb} or type \nosf{VId}.
The action of $M$ on $E$ respects the decomposition into generalized $T$-eigenspaces
$$E = \bigoplus_\chi E^{(\chi_{\norm})} \ .$$
The generalized eigenspaces $E^{(\chi_{\norm})}$ are irreducible $M$-modules,
so $M$ acts semisimply on $E$.
We treat the generic constituents of $E$, indexed by $\Delta_0(\Pi)$,
separately from the one-dimensional constituents.
By corollary~\ref{cor:Waldspurger-Tunnell} the generic constituents of $E$
contribute a one-dimensional $\rho$-equivariant quotient for every $\rho$,
so $\Delta_0(\Pi)\subseteq \widetilde{\Delta}(\Pi)$.
For the one-dimensional constituents $\sigma=(\mu\circ\det)\boxtimes\chi_\Pi$ of $E$, lemma \ref{mysterious} asserts that if a split Bessel model exists for some Bessel character $\rho$, then the torus $\widetilde{T}$ acts on $\sigma$ by $\rho=\mu\chi_\Pi$.
In that case the $T$-action on $\sigma$ contributes the character $\chi_{\norm}=\chi_\Pi^\divideontimes=\mu^2\chi_\Pi$ to $\widetilde\Delta$ if and only if $\rho$ provides a Bessel model, see remark~\ref{rmk:torus_action_on_Jacquet}.
For types \nosf{IVd}, \nosf{Vd} and \nosf{VIb} there is no split Bessel model, but in these cases the Siegel-Jacquet module is finite-dimensional and $E_{\widetilde{T},\rho}$ is easy to determine.

In case \nosf{IIb}, $\Pi$ is the fully induced representation $I=\Ind_P^G((\chi_1\circ\det)\boxtimes\chi_1^{-1})$.
Lemma~\ref{lem:filtration_diagram_exact} yields an exact sequence of $TS$-modules
$$ \cdots\to \beta^\rho(I_{\leq1}) \stackrel{\delta}\to \widetilde{I_{\geq 2}} \to \widetilde I \to \widetilde{I_{\leq1}} \to 0$$
 and the corresponding vertical sequences.
The Bessel character $\rho=1$ gives a split Bessel model for $I$. This is an ordinary case, so $\widetilde{I}_{\leq1}\cong \widetilde{I_0}$ and $\beta^\rho(I_{\leq1})\cong \beta^\rho(I_0)$.
Lemma~\ref{tildeKoh} gives $\beta^\rho(I_0)\cong\nu^{3/2}\chi_1$,
but lemma~\ref{tildemod} implies that $\widetilde{I_3}$ is perfect and $\widetilde{I_2}\cong\nu\neq \nu^{3/2}\chi_1$,
so the $T$-equivariant coboundary map $\delta:\beta^\rho(I_{\leq1}) \to \widetilde{I_{\geq 2}}$ is zero.
If we apply the functor $\pi_0$, we obtain the $T$-characters $\pi_0(\widetilde{I_3}) \cong \nu^{3/2}\chi_1$ and $\pi_0(\widetilde{I_2})\cong\nu$
and $\pi_0({\widetilde I_0})\cong \nu^{3/2}\chi_1^{-1}$. By exactness of $\pi_0$ this implies $\widetilde{\Delta}(\Pi) = \{\chi_1,\chi_1^{-1}, \nu^{-1/2}\} = \Delta(\Pi)$.

Now consider case \nosf{IIb} for $\rho\neq1$, which does not give a split Bessel model. $\beta^\rho(I)$ is perfect by lemma~\ref{lem:beta^rho_perfect} and finite-dimensional by proposition~\ref{BETA1}, thus zero. Furthermore, $\beta^\rho(I_0)$, $\widetilde{I_0}$ and $\widetilde{I_3}$ are zero by lemma~\ref{tildemod} and lemma~\ref{tildeKoh}.
We obtain an exact sequence $0\to\beta^\rho(I_1)\to\widetilde{I_2}\to\widetilde{I}\to\widetilde{I_1}\to0$ of $TS$-modules
where $\beta^\rho(I_1)\cong\widetilde{I_1}$. Thus $\widetilde{I}$ has the same constituents as $\widetilde{I_2}\cong \nu$ and this implies $\widetilde\Delta(\Pi)=\{\nu^{-1/2}\}=\Delta_0(\Pi)$ for $\rho\neq1$.

For case \nosf{VId},  $E=E^{(\chi_{\norm})}$ is a generalized eigenspace for $\chi_{\norm}=\nu^{-1/2}$.
By lemma~\ref{lem:filtration_VId}, $E_{\chi_{\norm}} \cong (1\times\nu^{-1})$ is indecomposable
as a module of $M^1\cong\Gl(2)$.
From lemma~\ref{lem:torus_action_GL2} we get the assertion $(E_{\chi_{\norm}})^{\widetilde{T},\rho}=0$ and by lemma~\ref{lem:monodromy_VId}, $(E_{\chi_{\norm}})_{\widetilde{T},\rho}$ is one-dimensional.
Lemma~\ref{lem:proto_long_exact_sequence} implies an exact sequence of $T$-modules
\begin{equation*}
0\to ((\nu^{-1/2}\circ\det)\boxtimes\nu^{1/2})_{\widetilde{T},\rho}\to E_{\widetilde{T},\rho}\to (E_{\chi_{\norm}})_{\widetilde{T},\rho}\to0\ .
\end{equation*}

The only Bessel character that defines a split Bessel model for $\Pi$ is $\rho=1$ and
this gives $\dim(E_{\widetilde{T},\rho})=2$, thus $\widetilde\Delta(\Pi)=\{\nu^{-1/2},\nu^{-1/2}\}$.
For $\rho\neq1$ the term $((\nu^{-1/2}\circ\det)\boxtimes\nu^{1/2})_{\widetilde{T},\rho}=0$ vanishes, so $E_{\widetilde{T},\rho}$ is one-dimensional and $\widetilde\Delta(\Pi)=\{\nu^{-1/2}\}=\Delta_0(\Pi)$.
\end{proof}

\subsection{Remarks on the Klingen-Jacquet module}\label{negligiblemodules}

For $I\in {\cal C}^{\fin}_G(\omega)$ the unnormalized Jacquet module $J_Q(I)$
is an $L_Q$-module for the Levi component $L_Q$ of the Klingen parabolic $Q$.
By the embedding in section~\ref{functor eta}
$$\Gl(2)\hookrightarrow  L_Q\ , \qquad \begin{pmatrix}\alpha&\beta\\\gamma&\delta\end{pmatrix}\ \mapsto \ \begin{pmatrix}\alpha\delta-\beta\gamma&0&0&0\\0&\alpha&0&\beta\\0&0&1&0\\0&\gamma&0&\delta\end{pmatrix}\ ,$$
$L_Q$ is isomorphic to the direct product
$\{\diag(t,1,t^{-1},1)\mid t\in k^{\times}\} \times\Gl(2)$.
Thus the irreducible representations of $L_Q$ have the form
$$  (\chi' \boxtimes \tau)(t,m) = \chi'(t\det(m))\tau(m)\ $$ 
for
 $ m= [\begin{smallmatrix} \alpha & \beta \cr \gamma & \delta \end{smallmatrix}]$
 in $\Gl(2)$ and some
$\tau\in\CCC_{\Gl(2)}$ and $\chi'\in \CCC_{\Gl(1)}$.
The upper left matrix entry of an element in $L_Q$ acts with $\chi'$.
The maximal split torus acts by $\tilde t x_\lambda\mapsto   \omega(t_2)  
(\chi'\otimes\tau)([\begin{smallmatrix} \lambda & 0 \cr 0 & t_1/t_2    \end{smallmatrix}])$ where $\omega = \chi'\omega_\tau$ for the central character $\omega_\tau$ of $\tau$.
The square root of the modulus character is $\delta_Q^{1/2}(m \tilde t x_\lambda)= \vert \det(m)\lambda\cdot\tfrac{t_1}{t_2}\vert$.

The constituents $\delta_Q^{-1/2}\otimes \pi=\chi' \boxtimes \tau$ of the normalized Klingen-Jacquet module $\delta_Q^{-1/2}\otimes J_Q(I)$ for irreducible
$\Pi\in\CCC_G(\omega)$ can be found in \cite{Roberts-Schmidt}, table A.4.\footnote{In \cite{Roberts-Schmidt}, table A.4 the symbol $\otimes$ is used instead of $\boxtimes$.}
The restriction of $\pi$ to $\Gl(2)$ is $(\nu\chi'\circ\det)\otimes\tau \in \CCC_{\Gl(2)}$.
For each constituent $\pi$ of $J_Q(\Pi)$ we show that $k^\rho(i_*(\pi)\vert_{\Gl(2)})$ in ${\cal C}_1$
is either zero or irreducible of the form $i_*(\chi)\in {\cal C}$ for a character $\chi$ of $T$.

\begin{lem} \label{TABULA}  For irreducible $\Pi\in{\cal C}_G(\omega)$, every irreducible constituent 
of $k^\rho(i_*(J_Q(\Pi)))\cong i_*((J_Q(\Pi)_U)^\rho)$ in $\CCC$
is a character
$$ \chi=  {\nu^2\omega\chi'}{\rho^{-1}}  $$
with $\chi'$ in the multiset $\Delta_Q$ listed in table~\ref{tab:Delta}.
\end{lem}

\begin{proof} The isomorphism $k^\rho(i_*(J_Q(\Pi)))\cong i_*((J_Q(\Pi)_U)^\rho)$ in $\CCC$ is given by lemma~\ref{3.1}.
Let $B$ denote the Borel subgroup  of upper triangular matrices in $\Gl(2)$ and
$U$ its unipotent radical.
The functor $k^\rho\circ i_*$ is left exact and it is therefore sufficient to consider 
the irreducible constituents of $J_Q(\Pi)$ as above.
The pullback of such a constituent to a $\Gl(2)$-module is $\pi=(\nu\chi'\circ\det)\otimes\tau$.
Since $\tilde T$ acts by $\omega(t_2)\pi(\diag(1,\delta))$ for $\delta=t_1/t_2$, Frobenius reciprocity implies
 $\dim \Hom_{\tilde T}(\rho, \pi_U) \leq 1$.
The torus $T$ acts on $\Hom_{\tilde T}(\rho, \pi_U)\cong (\pi_U)^{\widetilde{T},\rho}$  by the character %
 $$ \chi(x_\lambda) = \frac{\omega_\pi}{\rho}(\lambda) \quad , \quad x_\lambda\in T \ .$$
Here $\omega_\pi=\nu^{2}(\chi')^2\omega_\tau$ is the central character of $\pi$ as a $\Gl(2)$-module.
By construction $\omega=\chi'\omega_{\tau}$ implies $\omega_\pi= \nu^{2}\omega\chi'$, hence
$\chi(x_\lambda)  = \nu^{2}\omega\chi'\rho^{-1}(\lambda)$.
\end{proof} 

Let $\rho$ be a fixed character of $\tilde T$.
Then for every $I\in {\cal C}^{\fin}_G(\omega)$ lemma~\ref{lem:long_exact_main} provides a long exact sequence %
$$ \dots\to k^\rho(i_*(J_Q(I))) \stackrel{\delta}\to M_\rho(\alpha(I)) \to \widetilde{I} \to k_\rho(i_*(J_Q(I))) \to 0 \ .$$

\begin{lem}\label{DELTA}
For irreducible $\Pi\in {\cal C}_G(\omega)$ the image of the boundary map
$$\delta: k^\rho(i_*(J_Q(\Pi))) \longrightarrow M_\rho(\alpha(\Pi))$$
is contained in the finite-dimensional subspace ${\Danismanfunctor}(M)$ of $M=M_\rho(\alpha(\Pi))$ in $\CCC$.
Its image in $\pi_0(M)$ is contained in the subspace of $\bigoplus_{\chi_{\norm}\in \Delta_0(\Pi)} X^{(\chi)}$ defined by the generalized eigenspaces of $X= i^*(M)$ that are attached to characters $\chi$ for which there  exists a character $\chi'$ in $\Delta_Q$ with the property $ \rho=\nu^2\chi'\omega\chi^{-1}$. 
\end{lem}
\begin{proof}
The first statement is clear because $k^\rho(i_*(J_Q(\Pi)))$ is finite-dimensional.
Corollary~\ref{cor:small_extension_lemma_preliminary} implies $i^*M_\rho(A)\cong i^*(A)$ for every $A\in \CCC_{1}^{\fin}$, 
so we have $X=i^*(\alpha(\Pi))$ and this is spanned by the characters $\chi_{\norm}\in \Delta_0(\Pi)$.
Since there is an embedding ${\Danismanfunctor}(M) \hookrightarrow i^*(M)$, the last result follows from lemma \ref{TABULA}. 
\end{proof}

\section{Main results}\label{c:Main}

In this section we determine the Bessel modules $\widetilde{\Pi}=\beta_\rho(\Pi)$ and $\beta^\rho(\Pi)$, attached to irreducible representations $\Pi\in{\cal C}_G(\omega)$.
This provides the final step in the proof of theorem~\ref{thm:reg_L_factor} and thus determines the regular part $L_{\mathrm{reg}}^{\mathrm{PS}}(s, \Pi,\mu, \Lambda)$ of the Piateskii-Shapiro $L$-functions.

It is not difficult to compute the irreducible constituents of $\beta_\rho(\Pi)$.
Indeed, the constituent $\mathbb{S}$ appears with multiplicity one if and only if $\rho$ defines a Bessel model for $\Pi$, see theorem~\ref{Mult1}.
The one-dimensional constituents have already been determined in lemma~\ref{ugly}.
However, the essential information for us is the structure of the $TS$-module $\widetilde{\Pi}$ itself.
By proposition~\ref{L-reg}, the regular part of the $L$-factor is divided by $L(s-3/2,\widetilde{\Pi}/\widetilde{\Pi}^S)$.
Recall that the $TS$-module $\widetilde\Pi$ is uniquely determined by its quotient $\pi_0(\widetilde{\Pi})$, its submodule $\kappa(\widetilde{\Pi})$
and by the projection $\pi_0(\widetilde{\Pi})\twoheadrightarrow \pi_0(\widetilde{\Pi}/\kappa(\widetilde{\Pi}))$, see lemma~\ref{lem:universal extensions classification}. It is thus essential to determine these $TS$-modules.

The $T$-module $\pi_0(\widetilde\Pi)$ is uniquely characterized by
its monodromy properties
and the list of its constituents, see theorem~\ref{thm:monodromy} and lemma~\ref{ugly}.

In most cases $\kappa(\widetilde{\Pi})=\widetilde\Pi^S$ vanishes. In order to show this we verify the conditions of the embedding criterion in lemma \ref{lem:equivalences2}.
This requires to embed nontrivial extensions $\mathbb{E}[\chi]$ into $\widetilde\Pi$.
To do this, we distinguish the cases where $\Pi$ is non-generic and where $\Pi$ is generic. Our proof is similar in both cases and heavily relies on 
properties of the $\Gl_a(2)$-modules $\eta(\Pi)=\overline \Pi$ attached to $\Pi$. In particular, the
submodule $A=j^!(\overline\Pi)\in {\cal C}$ plays a prominent role in the arguments and the property that
$A$ itself is perfect if $\Pi$ is generic. The main step in the computation of $\widetilde\Pi=\beta_\rho(\Pi)\in {\cal C}$ then centers around the relation between $A$ and $\widetilde\Pi$.
It turns out that the Mellin transform $M_\rho: {\cal C}\to {\cal C}$ applied to $A$
catches an essential part of $\widetilde\Pi$ through the existence of a natural morphism
$M_\rho(A) \to \widetilde\Pi$. Using this morphism we construct the required embeddings
$\mathbb E[\chi] \hookrightarrow \widetilde\Pi$ from embeddings of $\mathbb E[\chi]$ into
$M_\rho(A)$ since the kernel of the mapping $M_\rho(A) \to \widetilde\Pi$ does not
present obstructions for this (lemma \ref{lem:embed_app}).
So the main task for the proof is the construction of embeddings into $M_\rho(A)$.

The case where $\Pi$ is non-generic is much simpler.  The reason is that the Mellin transform
$M_\rho(A)$ is more complicated in the generic case because then the degree of $A$ is always one.
The key result used is the computation of the Mellin transform in theorem~\ref{thm:ME_split}.
Hence our proof starts with the discussion of the non-generic cases and we prove
that $\widetilde\Pi$ is always perfect in these cases.
For the generic cases however, we encounter certain difficulties where the argument cannot be carried over. 
Except for the case \nosf{IIIa}, discussed separately, the combinatorical
results of section~\ref{s:Combinat} single out two series of exceptional cases where the 
argument breaks down. It turns out that for these two series of exceptional cases, the \emph{extraordinary exceptional cases} and the \emph{fully induced non-ordinary exceptional
cases}, the Bessel modules $\widetilde\Pi$, attached to $\Lambda$ respectively $\rho$, are not perfect.

We complete our discussion by a detailed analysis of these exceptional cases.
These finer results determine $\kappa(\widetilde{\Pi})=\widetilde{\Pi}^S$ in all cases and thus yield the final step in the proof of 

\begin{thm}[{theorem~\ref{thm:reg_L_factor}}]\label{thm:final_result}
 For irreducible $\Pi\in\CCC_G(\omega)$ and every split Bessel model $(\Lambda,\psi)$,
 the regular part $L_{\mathrm{reg}}^{\mathrm{PS}}(s,\Pi,1,\Lambda)$ of the Piatetskii-Shapiro $L$-factor
 is given by table~\ref{tab:regular_poles}.
\end{thm}

\subsection{Split Bessel Models}\label{s:split_Bessel}

Bessel models for irreducible $\Pi\in\CCC_G(\omega)$ have been classified in \cite{Roberts-Schmidt_Bessel}.
For the convenience of the reader we give a proof in the split case.
Recall that $\Pi$ admits a split Bessel model if and only if the degree of $\widetilde{\Pi}=\beta_\rho(\Pi)$ is positive.
\begin{thm}\label{Mult1}
For an irreducible $\Pi\in \CCC_G(\omega)$ the degree of $\beta_\rho(\Pi)=\widetilde{\Pi}$ is 
\begin{equation*}
\deg( \widetilde{\Pi})=
 \begin{cases}
     1 & \text{if}\ \Pi\ \text{is generic or}\  \rho\!\in\!\Delta_\pluss(\Pi)\ ,\\
     0 & \text{otherwise}\ .
 \end{cases}
\end{equation*}

\end{thm}

\begin{proof} %
For generic $\Pi$, see corollary~\ref{cor:gen_Mult1}.
We consider the case where $\Pi$ is non-generic.
Then the inequality of proposition~\ref{lasst} becomes sharp, so $\deg(\widetilde{\Pi})=d(\Pi,\rho)$.
This is non-zero if and only if $\rho\in\Delta_\pluss(\Pi)$ because $J_P(\Pi)_\psi$ decomposes as a direct sum of generalized $T$-eigenspaces.
It only remains to be shown that the degree cannot exceed one.
Without loss of generality we can assume $J_P(\Pi)\neq0$.
By dual Frobenius reciprocity there is an irreducible $\sigma_\Pi\in\CCC_M(\omega)$ such that $\Pi$ is a quotient of $I=\Ind_P^G(\sigma_\Pi)$.
If $\sigma_\Pi$ is infinite-dimensional, we can assume that $\sigma_\Pi\not\cong\omega\otimes\sigma_\Pi^\vee$.
Indeed, this is only violated in the case where $\pi=\nu\times\nu^{-1}$ where $\Pi$ is of type \nosf{IIIb};
in this case replace $\chi_1$ by $\chi_1^{-1}$.
Now lemma~\ref{lem:combinatoric} ensures
$\rho\neq\rho_\pm(\sigma_\Pi)$.
Thus, regardless of the dimension of $\sigma_\Pi$, the datum $(\sigma_\Pi,\rho)$ is ordinary in the sense of section \ref{s:SiegelInd}.
Right exactness of $\beta_\rho$ and proposition~\ref{KEY} imply
 $\deg(\widetilde{\Pi})\leq \deg(\widetilde{I}) = 1$.
\end{proof}

\subsection{$\beta^\rho(\Pi)$ for irreducible $\Pi$}

\begin{prop}\label{prop:beta^rho}
For irreducible $\Pi\in \CCC^{\fin}_G(\omega)$ with $J_P(\Pi)\neq0$, normalized as in table \ref{tab:List}, and smooth Bessel characters $\rho$,
\begin{enumerate}
\item if $\Pi$ is generic, then $\beta^\rho(\Pi)=0$ vanishes for every $\rho$;
\item if $\Pi$ is non-generic and $\rho$ provides a split Bessel model for $\Pi$, then $\delta_P^{-1/2}\otimes\beta^\rho(\Pi)\cong\mathbb{E}[X]$ is perfect.
Here $X$ is the unique cyclic $T$-module with constituents in $\Delta_1(\Pi)$ given by table~\ref{tab:Delta};
\item if $\Pi$ is non-generic and $\rho$ does not provide a split Bessel model, then
\begin{equation*}
\delta_P^{-3/2}\otimes \beta^\rho(\Pi)=
\begin{cases}
     \nu^{-3/2} & \text{type \nosf{IVd}}\ ,\ \rho=1\ ,\\
     \chi_0\nu^{-1/2} & \text{type \nosf{Vd}}\ ,\ \rho=1\ ,\\
     \nu^{-1/2} & \text{type \nosf{Vd}}\ ,\rho=\chi_0\ ,\ \\
     \nu^{1/2} & \text{type \nosf{VIb}}\ ,\ \rho=1\ ,\\
     0& \text{otherwise}\ .
\end{cases}
\end{equation*}
\end{enumerate}

\end{prop}
\begin{proof}
The statement for generic $\Pi$ has been shown in lemma~\ref{lem:beta^rho=0_generic}.

If $\Pi$ is non-generic and $\rho$ provides a split Bessel model, then $\widetilde{\Pi}$ has degree one and $\beta^\rho(\Pi)$ also has degree one by proposition~\ref{BETA1}.
By lemma~\ref{lem:beta^rho_Jacquet}, the constituents $\chi$ of the $T$-module $X\cong\pi_0(\beta^\rho(\Pi))$ are those of $J_P(\Pi)^{\widetilde{T},\rho}$, i.e.~corresponding to $\chi_{\norm}$ in the multiset $\widetilde{\Delta}(\Pi)\setminus \Delta_0(\Pi)=\Delta_1(\Pi)$ given by lemma~\ref{ugly}.
We have to show that $\beta^\rho(\Pi)$ is perfect in the cases where $\Delta_1(\Pi)$ is not empty.
For types \nosf{IIb}, \nosf{IVb} and \nosf{Vbc} this follows from lemma~\ref{lem:beta^rho_perfect}.
It only remains to consider the case where $\Pi$ is of type \nosf{VId} and $\rho=1$.
Then $\Delta_1(\Pi)$ contains only one element $\chi_{\norm}=\nu^{-1/2}$.
By lemma~\ref{lem:equivalences2} it is sufficient to construct an embedding $\mathbb{E}[\nu]\hookrightarrow \beta^\rho(\Pi)$.
Indeed, $\Pi$ is a quotient of $I=\Ind_P^G(\sigma_\Pi)$ for the representation $\sigma_\Pi=(\nu^{1/2}\circ\det)\boxtimes\nu^{-1/2}$ of $M$.
Lemma~\ref{tildeKoh} for $\rho=1$ provides an embedding $\mathbb{E}[\nu]\cong \beta^{\rho}(I_3)\hookrightarrow \beta^{\rho}(I)$.
By degree reasons, this provides a non-trivial morphism $\mathbb{E}[\chi]\to \beta^\rho(\Pi)$, which is injective by lemma~\ref{screening}.

If $\Pi$ is non-generic and $\rho$ does not give a split Bessel model, then $\deg(\widetilde{\Pi})=0$. Proposition~\ref{BETA1} yields $\deg(\beta^\rho(\Pi))=0$, so $\beta^\rho(\Pi)$ is a finite-dimensional $TS$-module with constituents $\chi$ corresponding to $\chi_{\norm}$ in 
$\Delta_1(\Pi)=\widetilde{\Delta}(\Pi)\setminus\Delta_0(\Pi)$. The constituents of $\widetilde\Delta(\Pi)$ have been determined in lemma~\ref{ugly}.
\end{proof}

\begin{lem}\label{lem:not_Siegel_induced}
 For $\Pi\in\CCC_G^{\fin}(\omega)$ with $J_P(\Pi)=0$ and every $\rho$, we have $\beta^\rho(\Pi)=0$.
\end{lem}
\begin{proof}
Indeed, $\widetilde{\Pi}\cong \mathbb{S}^{m_\Pi}$, where $m_\Pi$ is the multiplicity of Whittaker models, is  shown in proposition~\ref{lasst}.
Lemma~\ref{lem:beta^rho_Jacquet}
implies $\pi_0(\beta^\rho(\Pi))=0$. %
The degree $\deg(\beta^\rho(\Pi))=\deg(\widetilde{\Pi})-m_\Pi=0$ vanishes by proposition~\ref{BETA1}.
\end{proof}

\subsection{$\beta_\rho(\Pi)$ for irreducible $\Pi$}

We want to determine $\widetilde{\Pi}=\beta_\rho({\Pi})$ for every $\rho$ that defines a split Bessel model for $\Pi$. By lemma~\ref{lem:universal extensions classification} it suffices to determine $\pi_0(\widetilde{\Pi})$, see lemma~\ref{ugly}, and $\Danismanfunctor(\widetilde{\Pi})$, see theorem~\ref{thm:sigma_tilde_Pi}.
The proof depends on the exact sequence of lemma~\ref{lem:long_exact_main}.
For this reason we write $\alpha=j^!\circ\eta $ and use the Mellin functors, i.e.\ the additive endofunctors 
$M_\rho=k_\rho\circ j_!$ and $M^\rho=k^\rho\circ j_!$ of $\CCC$
that were studied in section~\ref{s:Mellin}.

\begin{lem}\label{lem:embed_app}
 Suppose $\chi$ is a character of $\Gl(1)$ and $I\in \CCC_G^{\fin}({\omega})$.
 An embedding $\mathbb{E}[\chi]\hookrightarrow M_\rho(\alpha(I))$ gives rise to an embedding $\mathbb{E}[\chi]\hookrightarrow \beta_\rho(I)$.
\end{lem}
\begin{proof}
Lemma~\ref{lem:long_exact_main} gives a long exact sequence in $\CCC$
\begin{equation*}
\cdots\to k^\rho i_* (B)\stackrel{\delta}{\to} M_\rho(A) \to k_\rho(\overline{I})\to k_\rho i_* (B)\to 0\ .
\end{equation*} for $A=\alpha(I)$ and $B=J_Q(I)\in {\cal C}_{\Gl(2)}$.
The modules $k^\rho i_* (B)$ and $k_\rho i_* (B)$ are finite-dimensional.
Hence the embedding $\mathbb{E}[\chi]\to M_\rho\alpha(I)$ remains non-trivial if composed with the morphism
$f: M_\rho(\alpha(I))\to k_\rho(\overline{I})=\beta_\rho(I)$.
The composition is injective by lemma \ref{screening}.
\end{proof}

\begin{thm}\label{thm:sigma_tilde_Pi}
Suppose $\Pi\in\CCC_{G}^{\fin}(\omega)$ is an irreducible smooth representation of $G$ with $J_P(\Pi)\neq0$, a unique quotient of $\Ind_P^G(\sigma_\Pi)$ as in table \ref{tab:List}. Let $\rho$ be a Bessel character that provides a split Bessel model for $\Pi$.
Then $\widetilde{\Pi}$ is perfect except for (twists of) the following cases, where $\kappa(\widetilde{\Pi})=\widetilde{\Pi}^S$ is a $T$-character:
\begin{center}
\begin{tabular}{llll}
\toprule
 Type       & $\sigma_\Pi$     &  $\{\rho,\rho^\ast\}$     &  $\delta_P^{-1/2}\otimes\widetilde\Pi^S$                       \\
 \midrule
\nosf{I} &$(\chi_1 \times\chi_2)\boxtimes 1$& $\{\nu^{1/2},\nu^{-1/2}\chi_1\chi_2\}$ &  $\chi_1\chi_2$ \\
           &                       &  $\{\nu^{1/2}\chi_1,\nu^{-1/2}\chi_2\}$       &  $\chi_2$       \\
           &                       &  $\{\nu^{1/2}\chi_2,\nu^{-1/2}\chi_1\}$       &  $\chi_1$       \\
           &                       &  $\{\nu^{1/2}\chi_1\chi_2,\nu^{-1/2}\}$       &  $1$             \\
\nosf{IIa}&$\Sp(\chi_1)\boxtimes \chi_1^{-1}$ &  $\{\nu^{1/2}\chi_1,\nu^{-1/2}\chi_1^{-1}\}$    &  $\chi_1^{-1}$  \\
           &                       &  $\{\nu^{1/2}\chi_1^{-1},\nu^{-1/2}\chi_1\}$  &  $\chi_1$       \\
\nosf{X} &$\pi_{c}\boxtimes 1$  &  $\{\nu^{1/2},\nu^{-1/2}\omega_{\pi_{c}} \}$  & $\omega_{\pi_c}$\\
           &                       &  $\{\nu^{1/2}\omega_{\pi_{c}},\nu^{-1/2} \}$  &  $1$\\
\midrule
\nosf{IIa}& $\Sp(\chi_1)\boxtimes \chi_1^{-1}$           &  $\{1\}$           &  $\nu^{1/2}$         \\
\nosf{Va} & $\Sp(\nu^{-1/2}\chi_0)\boxtimes\nu^{1/2}$    &  $\{1\}$           &  $\nu^{1/2}$         \\
            &                                            &  $\{\chi_0\}$      &  $\nu^{1/2}\chi_0$   \\
\nosf{VIa}& $\Sp(\nu^{-1/2})\boxtimes\nu^{1/2}$          &  $\{1\}$           &  $\nu^{1/2}$         \\
\nosf{XIa}& $\nu^{-1/2}\pi_{c}\boxtimes\nu^{1/2}$        &  $\{1\}$           &  $\nu^{1/2}$         \\
\bottomrule
\end{tabular}
\end{center}
\end{thm}

\begin{proof} For non-generic $\Pi$, see corollary~\ref{cor:sigma_nongeneric}.
For generic $\Pi$, by lemma~\ref{duality} we can assume that $(\Pi,\rho)$ belongs to the four cases listed in corollary~\ref{cor:comb_cases2} with $\rho\notin\Delta_\pluss(\Pi)$.
If $\rho\notin \Delta_\pluss(\Pi)\cup\Delta_{\minuss}(\Pi)$,
the result follows from corollary~\ref{perfectness widetildePi}.
If $(\Pi,\rho)$ belongs to the fully induced non-ordinary exceptional cases, the result is shown in proposition~\ref{prop:fully_induced_non-ordinary}.
For the extraordinary exceptional cases, see proposition~\ref{prop:extra-ordinary_exceptional}.
For $\Pi$ of type \nosf{IIIa} with $\rho\in\Delta_{\minuss}(\Pi)\cap\Delta_{\minuss}^\divideontimes(\Pi)$,
see proposition~\ref{prop:special case IIIa}.
\end{proof}

\begin{cor}
For spherical unitary irreducible $\Pi\in\CCC_{G}(\omega)$ and unitary Bessel characters $\rho$ that provide a split Bessel model to $\Pi$, the module $\widetilde{\Pi}$ is perfect.
\end{cor}
\begin{proof}
If $\Pi$ is generic and spherical unitary, then it is fully Borel induced $\chi_1\times\chi_2\rtimes\chi_\Pi$ from unramified characters $\chi_1,\chi_2,\chi_\Pi$ with $|\chi_\Pi\chi_i|=\nu^\beta$ for a real $-\tfrac12<\beta<\tfrac12$, see \cite[tables~A.2, A.15]{Roberts-Schmidt}.
For this case and for every non-generic $\Pi$, theorem~\ref{thm:sigma_tilde_Pi} implies the statement.
\end{proof}

\begin{prop}
For an irreducible $\Pi\in\CCC_{G}^{\fin}(\omega)$ assume that a smooth Bessel character $\rho$ does not provide a split Bessel model for $\Pi$.
Then $\Pi$ is non-generic and the Bessel module $\beta_\rho(\Pi) = \widetilde{\Pi}$ is finite-dimensional and given in table~\ref{tab:Bessel_module_no_split_Bessel_model}.

\end{prop}
\begin{proof}
$\Pi$ must be non-generic by the classification of split Bessel models, see theorem~\ref{Mult1}.
The action of $S$ on $\widetilde{\Pi}$ is trivial by lemma~\ref{FINITE},
so $\widetilde{\Pi}$ is isomorphic to $\pi_0(\widetilde{\Pi})$.
The constituents of $\pi_0(\widetilde{\Pi})$ are given by lemma~\ref{ugly}.
There are no non-trivial extensions between these constituents,
so $\widetilde{\Pi}$ is semisimple.
\end{proof}

\subsection{Non-generic representations}\label{s:beta_rho_non-generic}

\begin{prop}\label{prop:embedding_nongeneric}
For non-generic irreducible $\Pi\in\CCC_G(\omega)$ with Bessel character $\rho$ that defines a split Bessel model,
and every $\chi_{\norm}\in \widetilde{\Delta}(\Pi)$ there is an embedding
\begin{equation*}
\mathbb{E}[\chi]\hookrightarrow \widetilde{\Pi}\  \quad \text{in}\ \ \CCC\ .
\end{equation*}
\end{prop}
\begin{proof}
By a twist we can assume that $\Pi$ is normalized as in table \ref{tab:List}.
The space of Whittaker functionals is isomorphic to $j^!\alpha(\Pi)=0$, so lemma~\ref{lem:Gelfand-Kazhdan} implies $$\alpha(\Pi)=i_* i^* (\alpha(\Pi)) \ .$$ We dinstinguish the cases $\chi_\norm\in\Delta_0(\Pi)$ and $\chi_\norm\in \Delta_1(\Pi)$.

First we consider the case where $\chi_{\norm}\in \Delta_0(\Pi)$, or in other words where $\nu^{1/2}\chi_{\norm} \in \Delta_{\pluss}(\Pi)$.
Since $\Pi$ is non-generic, by theorem~\ref{Mult1} only the finitely many characters $\rho\in\Delta_{\pluss}(\Pi)$ define split Bessel models.
By lemma~\ref{lem:intersection--*}, the involution $\rho \mapsto \rho^\divideontimes={\omega}\rho^{-1}$ acts transitively on $\Delta_{\pluss}(\Pi)=\Delta^\divideontimes_{\pluss}(\Pi)$, so either there is a unique Bessel character $\rho=\rho^\divideontimes$ or there are two characters $\rho$, $\rho^\divideontimes$.
By lemma~\ref{duality}, it is sufficient to fix the Bessel character $\rho = \nu^{1/2}\chi_{\norm}$.
The corresponding unnormalized $T$-character is
$ \chi=\nu\rho \in \delta_P^{1/2}\Delta_0(\Pi)$.
By definition of $\Delta_0(\Pi)$, there is an embedding of $T$-modules into the unnormalized Jacquet module
$$ \chi\ \hookrightarrow \ i^*\alpha(\Pi)= J_P(\Pi)_\psi \ .$$ with finite-dimensional cokernel $E$.
Since every character $\chi_{\norm}$ occurs in $\Delta_0(\Pi)$ at most once (table~\ref{tab:Delta}),
$E$ does not contain $\chi$ as a constituent.
For the exact sequence $0\to\chi\to i^*\alpha(\Pi)\to E\to0$ in $\CCC_T^\fin$,
lemma~\ref{lem:long exact sequence_affine} yields a long exact sequence of $\Gl_a(1)$-modules
$$\cdots\to M^\rho(i_* (E))\to M_\rho(i_*(\chi))\to M_\rho(i_*i^*(\alpha(\Pi)))\to M_\rho(i_* (E))\to0\ .$$
In lemma~\ref{lem:small extension} we have shown that $M^\rho(i_* (E))=0$ and $M_\rho(i_*(\chi))\cong\mathbb{E}[\chi]$ for our fixed choice of $\rho=\nu^{-1}\chi$.
Since $\alpha(\Pi) = i_*i^*\alpha(\Pi)$ in $\CCC_1$, as explained above,
we obtain an embedding $$\mathbb{E}[\chi]\hookrightarrow M_\rho(\alpha(\Pi))\ .$$
By lemma~\ref{lem:embed_app} this yields an embedding $\mathbb{E}[\chi]\hookrightarrow k_\rho(\overline{\Pi}) = \widetilde{\Pi}$.

Now we consider the second case where $\chi_{\norm}\in\Delta_1(\Pi)$. This case occurs if and only if $\Pi$ is type \nosf{IIb, IVb, Vb, Vc, VId}, see lemma~\ref{lem:Delta_1}.
By corollary~\ref{cor:sigma_Pi_existence} there is an irreducible one-dimensional
$M$-module $\sigma_\Pi=\pi\boxtimes\chi_\Pi$
such that $\Pi$ is a quotient of $I=\Ind_P^G(\sigma_\Pi)$ and $\chi_\Pi=\chi_{\norm}$ as $\Gl(1)$-characters.
Then $(\sigma_\Pi,\rho)$ is an ordinary pair, so $\widetilde{I}$ has degree one by lemma~\ref{KEY} and there is
an embedding $\mathbb{E}[\delta_P^{1/2}\chi_\Pi]=\mathbb{E}[\chi]\hookrightarrow\widetilde{I}$.
By uniqueness of Bessel models (theorem~\ref{Mult1}), $\widetilde{\Pi}$ has degree one,
so the kernel of the projection $\widetilde{I}\to\widetilde{\Pi}$ has degree zero.
The composition $\mathbb{E}[\chi]\to\widetilde{\Pi}$ is non-zero, so by lemma~\ref{screening} it is an embedding.
\end{proof}

\begin{cor}\label{cor:sigma_nongeneric}
For non-generic irreducible $\Pi\in\CCC_G^{\fin}(\omega)$ and every smooth character $\rho$ that
provides a split Bessel model, the Bessel module $\widetilde\Pi$ is perfect.
\end{cor}

\begin{proof}
The degree of $\widetilde\Pi$ is one by theorem~\ref{Mult1}. By lemma~\ref{lem:perfect_model}, perfectness of $\widetilde\Pi$ is equivalent to the monodromy estimate $\dim\pi_0(\widetilde{\Pi})_\chi\leq1$ and the existence of an embedding $\mathbb{E}[\chi]\hookrightarrow\widetilde\Pi$ for every $\chi_{\norm}\in\widetilde\Delta(\Pi)$.
The monodromy estimate is theorem~\ref{WEIRD} and the embedding is given by proposition~\ref{prop:embedding_nongeneric}.
\end{proof}

\subsection{Generic representations}\label{s:discussion}

Fix a generic $\Pi\in {\cal C}_G(\omega)$, normalized as in table \ref{tab:List}, and a Bessel character
$\rho$ as in corollary~\ref{cor:comb_cases2}.
We determine ${\Danismanfunctor}(\widetilde\Pi)$ and show it is non-zero in certain cases.

\textit{Critical characters}.
We say  $\chi_{\norm}\in\Delta(\Pi)$ is \emph{critical} for $\rho$ and $\Pi$ if it equals
$$\chi_{\crit} = \nu^{1/2}\cdot \rho \ .$$
Similarly we say the associated character $\chi=\nu^{3/2}\chi_{\crit}$ is critical for $\rho$.
The critical character is clearly uniquely determined by $\rho$ and $\Pi$.
A critical character $\chi_{\crit}$ exists if and only if $\rho\in\Delta_{\minuss}(\Pi)$,
so lemma \ref{lem:exceptional++*} implies

\begin{lem}\label{exch}
A character $\chi_{\norm}$ is critical for both $(\Pi,\rho)$ and $(\Pi,\rho^\divideontimes)$ if and only if, up to a twist,
$(\Pi,\rho)$ belongs to the extraordinary exceptional cases in corollary~\ref{cor:comb_cases2} .
\end{lem}

\textit{Perfectness}. To show perfectness for $\widetilde \Pi$ we can apply the characterization of perfect $TS$-modules
given in lemma \ref{lem:equivalences2} since we know $\deg(\widetilde \Pi)=1$ from theorem~\ref{cor:gen_Mult1}.
One of the two conditions for perfectness is true quite generally by theorem~\ref{WEIRD},
it holds except for the twists of the case \nosf{VIa},
$\rho=1$ which appears in the list of extraordinary exceptional cases.
Excluding this case, it suffices to construct for every character $\chi_{\norm}\in \Delta_0(\Pi)$ an embedding
\begin{equation*}
\mathbb E[\chi] \hookrightarrow k_\rho(\overline{\Pi})\cong\widetilde{\Pi}\ .
\end{equation*}

\textit{The module $\alpha(\Pi)$}. For irreducible generic $\Pi$ recall from lemma \ref{lem:TildeDelta}  $$\pi_0(\widetilde\Pi)^{ss} = \bigoplus_{\chi_{\norm}\in \Delta_0(\Pi)} \chi$$ (up to semisimplification).
Notice, for $\chi_{\norm}\in \Delta_0$ the character $\chi=\nu^{3/2}\chi_{\norm}$ is a constituent of 
$Y=\pi_0(\widetilde \Pi)$ so 
we have an exact sequence $$ 0\to\mathbb{S}\to\alpha(\Pi)\to i_*(Y) \to 0\ . $$ 
By lemma~\ref{lem:alpha_perfect} $\alpha(\Pi)$ is perfect,
so for all $\chi_{\norm}\in \Delta_0$ lemma~\ref{lem:equivalences2} and lemma~\ref{Dual action}
provide an embedding $\mathbb{E}[\chi]\hookrightarrow \alpha(\Pi)$.
More generally, for $0\neq X \subseteq Y^{(\chi)} \subseteq Y$ and for the $\Gl(1)$-module
$F=Y/X$ we obtain an exact sequence
\begin{equation*}
0 \to  \mathbb{E}[X] \to \alpha(\Pi)  \to i_*(F) \to  0 \ \qquad \text{in}\ {\cal C}\ .
\end{equation*}

Since $\rho\notin\Delta_{\pluss}(\Pi)$ by assumption, lemma \ref{lem:small extension} yields $M^\rho(i_*F)=0$. Now lemma~\ref{lem:long exact sequence_affine} provides 
\begin{lem}\label{lem:first_delta}
For $\rho\notin\Delta_{\pluss}(\Pi)$ there is an exact sequence
\begin{equation*}
\xymatrix{ 0\ar[r]^-{\delta} & M_\rho(\mathbb{E}[X])\ar[r] & M_\rho(\alpha(\Pi))\ar[r] & M_\rho(i_* F)\ar[r] & 0}\ .
\end{equation*}
\end{lem}

\begin{prop} \label{preliminary}
For generic irreducible $\Pi\in\CCC_G(\omega)$, a Bessel character $\rho\notin \Delta_{\pluss}(\Pi)$ and every non-critical character $\chi_{\norm} \in \Delta_0(\Pi)$,
there exists an embedding of $\mathbb{E}[\chi]$ into $\widetilde\Pi$ as $TS$-modules.
\end{prop}

\begin{proof}By a twist we can assume that $\Pi$ is normalized as in table~\ref{tab:List}.
For the construction of embeddings, recall from theorem~\ref{thm:big_extension}
$$M_\rho(\mathbb{E}[\chi]) \cong \mathbb{E}[\chi]$$
if $\chi$ is not critical for $\rho$,
whereas $M_\rho(\mathbb{E}[\chi]) \cong \mathbb S \oplus \chi$
splits for criticial $\chi$.
We assumed that $\chi$ is not critical, so lemma~\ref{lem:first_delta} for $X=\chi$
gives an embedding
$$ \mathbb{E}[\chi] \hookrightarrow M_\rho(\alpha(\Pi))\ .$$
Lemma \ref{lem:embed_app} yields the required embedding
$\mathbb{E}[\chi] \hookrightarrow \widetilde{\Pi}$.
\end{proof}

\begin{cor}\label{perfectness widetildePi}
For generic irreducible $\Pi\in\CCC_G(\omega)$ and a Bessel character $\rho\notin \Delta_{\minuss}(\Pi)\cup \Delta_{\pluss}(\Pi)$, the Bessel module $\widetilde\Pi$ is perfect.
\end{cor}

\begin{proof} Critical characters only exist for $\rho\in\Delta_{\minuss}(\Pi)$, which is excluded.
For non-critical characters $\chi_\norm\in\Delta_0(\Pi)$ proposition \ref{preliminary} yields embeddings $\mathbb E[\chi] \hookrightarrow \widetilde \Pi$.
The monodromy estimate $\dim(\pi_0(\widetilde \Pi)_\chi) \leq 1$ holds by theorem~\ref{WEIRD}.
Lemma~\ref{lem:equivalences2} implies the statement.
\end{proof}

\begin{cor}\label{cor:sigma_is_critical_Jordanblock}
Fix an irreducible generic representation $\Pi\in\CCC_G(\omega)$ that is not a twist of type \nosf{VIa},
and a Bessel character $\rho\notin \Delta_{\pluss}(\Pi)$. %
Then $\delta_P^{-1/2}\otimes{\Danismanfunctor}(\widetilde\Pi)$ is either zero or a single Jordan block attached to the critical character $\chi_{\crit}$.
\end{cor}
\begin{proof}Without loss of generality assume that $J_P(\Pi)\neq0$ and $\Pi$ is normalized as in table \ref{tab:List}. 
Theorem~\ref{WEIRD} asserts that every $T$-character appears in $\pi_0(\widetilde{\Pi})$ and thus in ${\Danismanfunctor}(\widetilde{\Pi})$ with at most a single Jordan block.
Fix a constituent $\chi$ of ${\Danismanfunctor}(\widetilde\Pi)$, then
$\chi_{\norm}=\nu^{-3/2}\chi\in \widetilde\Delta(\Pi)={\Delta}_0(\Pi)$ by lemma~\ref{lem:kappa_pi_0_injective} and lemma~\ref{ugly}.
If $\chi_{\norm}\neq\chi_{\crit}$ is not critical,
then proposition~\ref{preliminary} provides an embedding
$\mathbb{E}[\chi]\hookrightarrow\widetilde{\Pi}$, but this contradicts lemma~\ref{lem:embedding_criterion}.
\end{proof}

\subsection{Exceptional cases}

We now determine ${\Danismanfunctor}(\delta_P^{-1/2}\!\otimes\!\widetilde{\Pi})$ for the exceptional cases of corollary~\ref{cor:comb_cases2}.
Recall that the critical character is
$\chi_{\crit}=\nu^{1/2}\rho\in\Delta_0(\Pi)$.

\begin{prop}[Extraordinary exceptional case]\label{prop:extra-ordinary_exceptional}
For generic irreducible $\Pi\in\CCC_G(\omega)$ and a Bessel character $\rho$ with $\rho=\rho^\divideontimes\in\Delta_{\minuss}(\Pi)$,
$${\Danismanfunctor}(\delta_P^{-1/2}\!\otimes\!\widetilde \Pi) \cong \chi_{\crit}\ .$$
\end{prop}
\begin{proof}
By a twist we can assume that $J_P(\Pi)\neq0$ and $\Pi$ is normalized as in table~\ref{tab:List}.
Corollary~\ref{cor:comb_cases2} asserts that $\Pi$ is of type \nosf{IIa}, \nosf{Va}, \nosf{VIa} or \nosf{XIa} and  $\rho\notin\Delta_{\pluss}(\Pi)$. We first consider the case where $\Pi$ is not of type \nosf{VIa}.

By corollary~\ref{cor:sigma_is_critical_Jordanblock}, ${\Danismanfunctor}(\delta_P^{-1/2}\otimes\widetilde\Pi)$ is a single Jordan block attached to the critical character $\chi_{\crit}$.
The generalized eigenspace $X=(\pi_0(\widetilde{\Pi}))^{(\chi)}$ for $\chi=\nu^{3/2}\chi_{\crit}$
is one-dimensional by lemma~\ref{ugly}, i.e. $X\cong\chi$.
Theorem \ref{thm:ME_split} yields an injection 
$\chi \hookrightarrow M_\rho(\mathbb{E}[X])$.
By lemma~\ref{lem:first_delta} this induces an injection 
$$   i_*(\chi) \hookrightarrow M_\rho(\alpha(\Pi))\ .$$
Now we need to study the boundary map $\delta$ of lemma~\ref{lem:long_exact_main}.
The analysis of this secondary boundary map is done in lemma \ref{DELTA}, where it is shown that the 
image of this secondary boundary map can have an influence only if
$$ \nu^{1/2}\omega\chi^{-1}_{\crit} \cdot \chi' = \rho $$
holds for some $\chi'\in \Delta_Q$.
The condition $\rho=\rho^\divideontimes$ is equivalent to $\chi'=1$.
Since $\Pi$ is not of type \nosf{VIa},
the character $\chi'=1$ is not contained in $\Delta_Q(\Pi)$.
Lemma~\ref{DELTA} and lemma~\ref{lem:long_exact_main} yield an injection $\chi\hookrightarrow k_\rho(\overline\Pi)$. 
Since $\dim\pi_0(\widetilde{\Pi})^{(\chi)}\leq1$ by lemma \ref{ugly} and since ${\Danismanfunctor}(\delta_P^{-1/2}\otimes\widetilde{\Pi})$ is a single Jordan block attached to $\chi_{\crit}$, we obtain an isomorphism ${\Danismanfunctor}(\delta_P^{-1/2}\otimes\widetilde \Pi) \cong \chi_{\crit}$.

Consider the case where $\Pi$ is of type \nosf{VIa} and $\rho=1$.
Lemma~\ref{lem:monodromy_VIa} and lemma~\ref{ugly} imply
$\pi_0(\delta_P^{-1/2}\otimes\widetilde{\Pi})=\chi_{\crit}\oplus \chi_{\crit}$ as a $T$-module.
Since $\widetilde{\Pi}$ has degree $\deg(\widetilde{\Pi})=1$ by theorem~\ref{Mult1},
$\tilde \Pi$ is not perfect by lemma~\ref{lem:equivalences2}
and this means ${\Danismanfunctor}(\delta_P^{-1/2}\otimes\widetilde{\Pi})$ is either $\chi_{\crit}$ or $\chi_{\crit}\oplus\chi_{\crit}$.
In the latter case $\delta_P^{-1/2}\otimes\widetilde{\Pi}$ would be semisimple,
but this contradicts the fact that there is an embedding $\mathbb{E}[\chi_{\crit}]\hookrightarrow \delta_P^{-1/2}\otimes(\widetilde{\Pi}/{\Danismanfunctor}(\widetilde{\Pi}))$ from lemma~\ref{lem:key2}. 

This proves ${\Danismanfunctor}(\delta_P^{-1/2}\otimes\widetilde \Pi) \cong \chi_{\crit}$
for the extraordinary exceptional cases.
\end{proof}

\begin{lem}\label{lem:Mellin_inflation}
Fix an irreducible generic $\Pi\in\CCC_G(\omega)$ and a Bessel character $\rho\in \Delta_{\pluss}^\divideontimes(\Pi)$.
Then $\chi_{\crit}=\nu^{1/2}\rho$ is in $\Delta_0^\divideontimes(\Pi)$ and there is an embedding
$$\mathbb{E}[\chi_{\crit}^\divideontimes]\hookrightarrow \delta_P^{-1/2}\otimes\widetilde{\Pi}/{\Danismanfunctor}(\widetilde{\Pi})\ .$$
\end{lem}
\begin{proof}
By lemma~\ref{lem:alpha_perfect}, $\alpha(\Pi)\in\CCC$ is perfect of degree one and $\pi_0(\alpha) \cong J_P(\Pi)_\psi$ as a $T$-module for the unnormalized Siegel-Jacquet module $J_P(\Pi)$.
Thus for every $\chi_{\norm}\in\Delta_0(\Pi)$ there is a projection $\alpha(\Pi)\twoheadrightarrow i_*\chi$.
The right-exact Mellin functor $M_{\rho^\divideontimes}$ then provides a projection
$$M_{\rho^\divideontimes}(\alpha(\Pi))\twoheadrightarrow M_{\rho^\divideontimes}(i_*\chi)\ .$$
Especially, this holds for $\chi=\delta_{P}^{1/2}\chi_{\crit}^\divideontimes=\nu\rho^\divideontimes$, then
$M_{\rho^\divideontimes}(i_*\chi)=\mathbb{E}[\chi]\in\CCC$ is perfect of degree one by lemma~\ref{lem:small extension}.
Since $M_{\rho^\divideontimes}(\alpha(\Pi))$ also has degree one by theorem~\ref{thm:big_extension}, the quotient of $M_{\rho^\divideontimes}(\alpha(\Pi))$ by its maximal finite-dimensional submodule is isomorphic to $\mathbb{E}[\chi]$.
Lemma~\ref{lem:long_exact_main} provides a non-trivial morphism 
$$M_{\rho^\divideontimes}(\alpha(\Pi))\to\widetilde{\Pi}$$ 
with finite-dimensional kernel. Dividing out the maximal finite-dimensional submodules we obtain a well-defined non-trivial morphism
\begin{equation*}\mathbb{E}[\delta_P^{1/2}\chi_{\crit}^\divideontimes] \cong \mathbb{E}[\chi]\to \widetilde\Pi/{\Danismanfunctor}(\widetilde\Pi)\ \quad\text{ in } \CCC\ .\end{equation*}
It is injective by lemma \ref{SCREEN}.
Normalization with $\delta_P^{-1/2}$ gives the result.
\end{proof}

\begin{prop}[Fully induced non-ordinary case]\label{prop:fully_induced_non-ordinary}
For irreducible generic $\Pi\in\CCC_G(\omega)$ and a Bessel character $\rho$ assume that
$\rho\in \Delta_{\pluss}^\divideontimes(\Pi)\cap\Delta_{\minuss}(\Pi)$ and $\rho\notin\Delta_{\pluss}(\Pi)\cup\Delta_{\minuss}^\divideontimes(\Pi)$.
Then
\begin{equation*}
{\Danismanfunctor}(\delta_P^{-1/2}\!\otimes\!\widetilde\Pi)\cong \chi_{\crit}\ .
\end{equation*}
\end{prop}
\begin{proof}By a twist we can assume that $\Pi$ is normalized as in table~\ref{tab:List}.
$\Pi$ belongs to the generic, fully Siegel induced cases \nosf{I}, \nosf{IIa}, \nosf{X} by lemma~\ref{lem:exceptional+-*}.
By corollary~\ref{cor:sigma_Pi_existence},
there exists an irreducible representation $\sigma_\Pi=\pi \boxtimes\chi_\Pi$ of $M=M^1\times \{t_\lambda\mid \lambda \in k^{\times}\}$ such that $I=\Ind_P^G(\sigma_\Pi)$ admits $\Pi$ as a quotient and $\chi_\Pi(t_\lambda)=\chi_{\crit}(x_\lambda)$.
Table~\ref{tab:list_sigma} shows that $I\cong\Pi$ is irreducible.
Especially, we have $\rho=\rho_+(\sigma_\Pi)\neq \rho_-(\sigma_\Pi)$, so
the pair $(\sigma_\Pi,\rho)$ is non-ordinary but not extraordinary.
By theorem~\ref{Mult1}, $\widetilde{\Pi}$ has degree one, so
lemma~\ref{KEY} provides an embedding $i_*(\chi)\hookrightarrow \widetilde{\Pi}$ of $TS$-modules for $\chi=\nu^{3/2}\chi_{\crit}$.

It remains to be shown that $\dim{\Danismanfunctor}(\widetilde{\Pi})$ is not larger than one.
By corollary~\ref{cor:sigma_is_critical_Jordanblock}, ${\Danismanfunctor}(\widetilde{\Pi})$ is a single Jordan block attached to $\chi=\nu^{3/2}\chi_{\crit}$.
If $\dim\pi_0(\widetilde{\Pi})^{(\chi)}=1$, we are done. If $\dim\pi_0(\widetilde{\Pi})^{(\chi)}=2$ and  $\chi_{\crit}=\chi_{\crit}^\divideontimes$,
then lemma~\ref{lem:Mellin_inflation} provides an embedding $\mathbb{E}[\chi]\hookrightarrow \widetilde{\Pi}/{\Danismanfunctor}(\widetilde{\Pi})$, so $\dim{\Danismanfunctor}(\widetilde{\Pi})^{(\chi)}\leq1$ and we are done.
Finally, assume either $\dim\pi_0(\widetilde{\Pi})^{(\chi)}=2$ with $\chi_{\crit}\neq\chi_{\crit}^\divideontimes$ or $\dim\pi_0(\widetilde{\Pi})^{(\chi)}>2$. Table~\ref{tab:Delta} shows that this is only possible for $\Pi$ of case \nosf{I} with $\chi_1=1$ or $\chi_2=1$.
Fix the Bessel filtration $(I_\nu)_\nu$ for $\Pi=\Ind_P^G(\sigma_\Pi)$ and $\sigma_\Pi=(\chi_1\times\chi_2)\boxtimes1$.
It is sufficient to show that $${\Danismanfunctor}(\widetilde{I_{\leq1}})=0\ ,$$
then counting constituents implies the statement.
Indeed, $\widetilde{I_1}$ is perfect by lemma \ref{tildemod} and Kirillov theory.
If ${\Danismanfunctor}(\widetilde{I_{\leq1}})\neq0$, then $\widetilde{I_0}\cong {\Danismanfunctor}(\widetilde{I_{\leq1}})$ is a submodule, so the sequence $0\to \widetilde{I_1}\to \widetilde{I_{\leq1}}\to \widetilde{I_0}\to0$ splits.
By our assumptions, one of the constituents $\nu^{3/2}\chi_i$ in $\pi_0(\widetilde{I_1})$ is isomorphic to $\widetilde{I_0}\cong \nu^{3/2}\chi_1\chi_2$, so $\dim\pi_0(\widetilde{I_{\leq1}})_{\nu^{3/2}\chi_i}\geq2$.
But theorem~\ref{WEIRD} implies the monodromy bound $\dim\pi_0(\widetilde{I_{\leq1}})_{\nu^{3/2}\chi_i}\leq \dim\pi_0(\widetilde{I})_{\nu^{3/2}\chi_i} \leq1$, so this provides a contradiction.
\end{proof}

\begin{prop}\label{prop:special case IIIa} If $(\Pi,\rho)$ belongs to the exceptional case \nosf{IIIa} with $\Pi$ as in table~\ref{tab:List} and $\rho\in \{1,\chi_1\}$,
 then the Bessel module $\widetilde \Pi$ is perfect.
\end{prop}

\begin{proof}
Perfectness of $\widetilde\Pi$ will follow from lemma~\ref{lem:equivalences2}, where theorem~\ref{WEIRD} provides the monodromy estimate.
It only remains to construct embeddings $$\mathbb{E}[\chi] \hookrightarrow \delta_P^{-\frac{1}{2}}\widetilde\Pi\ \quad,\qquad\chi\in\widetilde{\Delta}(\Pi)=\{\nu^{\frac{1}{2}},\nu^{\frac{1}{2}}\chi_1\}\ .$$
For the irreducible representation $\sigma_\Pi=(\chi_1^{-1}\times\nu^{-1})\boxtimes\nu^{\frac{1}{2}}\chi_1\in\CCC_M(\omega)$ of the Levi subgroup $M$ of $P$,
the pair $(\sigma_\Pi,\rho)$ is nonordinary
and $\Pi$ is a quotient of $I\!=\! \Ind_P^G(\sigma_\Pi)$.
The non-generic irreducible representation $\Xi=\ker(I\twoheadrightarrow\Pi)$ of type \nosf{IIIb} admits a split Bessel model for $\rho$.
Furthermore, the sequence
$$ \xymatrix{  0 \ar[r] & \widetilde \Xi \ar[r]^j & \widetilde I \ar[r] & \widetilde \Pi \ar[r] & 0} \ $$
is exact because $\beta^\rho(\Pi)=0$ vanishes by lemma~\ref{lem:beta^rho=0_generic}.
Hence $\deg(\widetilde \Xi) = \deg(\widetilde \Pi)=1$ and therefore
$ \deg(\widetilde I)= 2 $ and $ \deg(\widetilde I_{\leq1})\geq 1 $.
By degree reasons the image of the coboundary map $\delta:\beta^\rho(I_{\leq 1})\to \widetilde I_{\geq 2}$ in the Bessel filtration
must be finite-dimensional, thus has zero intersection with the perfect submodule
$\widetilde{I_3}\subseteq \widetilde{I_{\geq2}}$. By lemma~\ref{ugly}, 
$$\pi_0(\delta_P^{-\frac{1}{2}}\otimes\widetilde \Pi )\!=\! \nu^{\frac{1}{2}}\!\oplus\! \nu^{\frac{1}{2}}\chi_1 \quad  \mbox{and}\quad \pi_0(\delta_P^{-\frac{1}{2}}\otimes\widetilde \Xi)\!=\! \nu^{-\frac{1}{2}}\!\oplus\! \nu^{-\frac{1}{2}}\chi_1\ =:X$$ are decomposable because $\chi_1\neq1$.
Prop.~\ref{prop:beta^rho} and left-exactness of $\beta^\rho$ imply $\beta^\rho(I)\cong\beta^\rho(\Xi)=\mathbb{E}[0]=\mathbb{S}$,
so the kernel of $\delta$ is $\mathbb{S}\subseteq\beta^\rho(I_1)$. Thus the image of $\delta$ is isomorphic to $\widetilde{I_2}$, and so the cokernel of $\delta$ is isomorphic to $\widetilde{I_3}$.
We obtain a commutative diagram with exact sequence
\begin{equation*}
\xymatrix@R-3mm{
         &                       & \delta_P^{-\frac{1}{2}}\otimes\widetilde{\Xi}\ar@{^{(}->}[d]_j\ar[dr]^\varphi & \\
0 \ar[r] &  {\underbrace{\delta_P^{-\frac{1}{2}}\otimes\widetilde{I_{3}}}_P} \ar[r]^-{i} & {\underbrace{\delta_P^{-\frac{1}{2}}\otimes \widetilde I}_M} \ar[r]^-{p} &   {\underbrace{\delta_P^{-\frac{1}{2}}\otimes\widetilde{I_{\leq 1}}}_Q} \ar[r] &  0 {\ .}
}
\end{equation*} 
We want apply to lemma \ref{Ext-Lemma} to this diagram and claim that this provides the required embedding.
Note that $P\cong\mathbb E[\chi_1{\nu^{\frac{1}{2}}}]$ and $\delta_P^{-\frac{1}{2}}\otimes\widetilde \Xi \cong \mathbb E[X]$ because $\widetilde{\Xi}$ is perfect by prop.~\ref{prop:embedding_nongeneric} and thus uniquely determined by $X$.
By lemma~\ref{tildeKoh} and Kirillov theory, $\delta_P^{-\frac{1}{2}}\widetilde{I_1}$ is perfect and embeds into $Q$ by degree reasons and lemma~\ref{screening}. Thus
${{\Danismanfunctor}}(Q)$ is either zero or isomorphic to $\delta_P^{-\frac{1}{2}} \widetilde I_0 = \nu^{-\frac{1}{2}}$.

To prove our claim, by lemma \ref{Ext-Lemma} it suffices to show there is no embedding $\mathbb{E}[Y]\hookrightarrow P=\mathbb E[{\nu^{\frac{1}{2}}\chi_1}]$ where $Y= \ker(\psi)$
is defined by $\psi: X \to {\Danismanfunctor}(Q)\subseteq \nu^{-\frac{1}{2}}$.
Indeed, we have $\nu^{-\frac{1}{2}}\chi_1 \subseteq Y$,  hence  $\Hom(\mathbb{E}[Y], \mathbb E[{\nu^{\frac{1}{2}}\chi_1}])= 0$ by lemma~\ref{SCREEN}.
Lemma~\ref{Ext-Lemma} provides the required embedding
\begin{equation*}
P=\mathbb{E}_1 [\nu^{\frac{1}{2}}\chi_1] \hookrightarrow \delta_P^{-\frac{1}{2}}\widetilde\Pi \ .
\end{equation*}
The analogous argument with $\sigma_\Pi=(\chi_1\times\nu^{-1})\boxtimes\nu^{\frac{1}{2}}$ provides an embedding $\mathbb{E}[\nu^{\frac{1}{2}}]\hookrightarrow\delta_P^{-\frac{1}{2}}\widetilde{\Pi}$.
Hence $\widetilde\Pi$ is perfect by lemma~\ref{lem:equivalences2}.
\end{proof}

\newpage

\appendix
\section{Appendix}

\subsection{Functors}

\begin{lem}\label{Casselman}\label{lemma 3}
For a totally disconnected locally compact group $\Gamma$ with a compactly generated\footnote{A topological group is \emph{compactly generated} if it is the union of its compact subgroups.} normal subgroup $\widetilde{\Gamma}$,
the functor ${\CCC}_\Gamma\to\CCC_{\Gamma/\widetilde{\Gamma}}$, sending a $\Gamma$-module $A$ to the coinvariant space $A_{\widetilde\Gamma}$ is exact.
\end{lem}

\begin{proof}
 This is well-known, compare \cite[2.35]{Bernstein-Zelevinsky76}.
\end{proof}

\begin{lem} \label{lemma 2}\label{lem:proto_long_exact_sequence}
For a totally disconnected locally compact group ${\Gamma}$ with a normal subgroup $\widetilde{\Gamma}$ isomorphic to $k^\times$,
every exact sequence $0\to E\to F\to G \to0$ in ${\CCC}_\Gamma$
gives rise to an exact sequence in $\CCC_{\Gamma}$
$$ 0 \to  E^{\widetilde{\Gamma}}\to F^{\widetilde{\Gamma}} \to G^{\widetilde{\Gamma}}
\to E _{\widetilde{\Gamma}} \to F_{\widetilde{\Gamma}} \to G_{\widetilde{\Gamma}}  \to 0 \ .$$
\end{lem}

\begin{proof} By $k^{\times} \cong \pi^\Z \times \mathfrak{o}^{\times}$ and the compactness of $\mathfrak{o}^{\times}$ this is immediately
reduced to the corresponding well-known statement for $\Gamma = \Z$, i.e.\ the snake lemma for the monodromy operator $\tau=\pi-id$.
\end{proof}

\subsection{Filtrations}\label{s_app:filtration}

For an irreducible representation $(\pi,V)$ of $M_1\cong\Gl(2)$ with central character $\omega_\pi$ and a character $\chi_\Pi$ of the torus $\{t_\lambda\mid\lambda\in k^\times\}$ the irreducible representation $\sigma_\Pi=\pi\boxtimes\chi_\Pi$ of $M$ determines the (normalized) Siegel induced representation $I=\Ind_P^G(\sigma_\Pi)$. Recall that $I$ is isomorphic to the right regular action of $G$ on the space of smooth functions $f:G\to V$ such that
$$f(mng)=\sigma_\Pi(m)\delta_P^{1/2}(m)f(g)$$ for $m\in M$, $n\in N$ and $g\in G$
with the modulus character $\delta_P(m_A t_\lambda)=\vert\frac{\det A}{\lambda}\vert^3$.
By the Bruhat decomposition, the double coset space $P\!\setminus\! G\!/\!P$ is represented by the relative Weyl group $W_P=\{id,\mathbf{s}_2,\mathbf{w}\}$ with the long root $\mathbf{w}=\mathbf{s}_2\mathbf{s}_1\mathbf{s}_2$.
This gives rise to a filtration $I_{\cell}^{\bullet}$ of $I$ by $P$-modules
$$0\subseteq I_{\cell}^{2}\subseteq I_{\cell}^{1}\subseteq I_{\cell}^{0}=I$$
with quotients $I_{\cell}^{i}/I_{\cell}^{i+1}$
isomorphic to certain induced representations
given explicitly in lemma 5.2.1 of \cite{Roberts-Schmidt_Bessel} where $I_{\cell}^{\bullet}$ is denoted $I^\bullet$.
The restriction to $T\widetilde{T}N$ gives rise to a filtration
$$0\subseteq I_{\geq3}\subseteq I_{\geq2} \subseteq I_{\geq1}\subseteq I_{\geq0}=I\ $$
with quotients $I_i=I_{\geq i}/I_{\geq i+1} \in \CCC_{T\widetilde{T}N}$ as follows:
\begin{enumerate}
\item $I_3=I_{\geq3}$ is the restriction of the $P$-module $I_{\cell}^2$ to $T\widetilde{T}N$.
It is isomorphic to the complex vector space $C_c^\infty(k^3,V)$
with the action of $T\widetilde{T}N$ on $f\in C_c(k^3,V)$
\begin{align*}
 &(s_{a,b,c}f)(x,y,z)    = f(x+a,y+b,z+c)\ ,\\
 &(\widetilde{t}f)(x,y,z)= \pi(\diag(t_2,t_1))\chi_\Pi(t_1t_2)f(xt_2/t_1,y,zt_1/t_2)\ ,\\
 &(x_\lambda f)(x,y,z)   =|\lambda|^{-3/2}\chi_\Pi(\lambda)f(\lambda^{-1}x,\lambda^{-1}y,\lambda^{-1}z)\ .
\end{align*} 
 
\item The $P$-module $I_{\cell}^1/I_{\cell}^2$ admits a model as certain subspace $W$ of $C^\infty(k^2,V)$, see \cite[pp.517ff]{Roberts-Schmidt_Bessel}. The action of $T\widetilde{T}N$ on $f\in W$ is
\begin{align*}
 &(s_{a,b,c}f)(u,w)    = \pi\left(\begin{smallmatrix}1&b+u a\\0&1\end{smallmatrix}\right)f(u,w+au^2+2bu+c)\ ,\\
 &(\widetilde{t}f)(u,w)= \chi_\Pi(t_1t_2)\omega_\pi(t_1)\left|t_1/t_2\right|^{3/2} f(\tfrac{t_1}{t_2}u,\tfrac{t_1}{t_2}w) \ ,\\
 &(x_\lambda f)(u,w)   = \chi_\Pi(\lambda)\pi(\lambda,1) f(u,\lambda^{-1}w)   \ .
\end{align*}
This action preserves the subspace $I_2:=C_c(k^\times\times k,V)\subseteq W$.

\item $I_1=W/I_2$ is isomorphic to the direct sum $ C_c^\infty(k,V)\oplus C_c^\infty(k,V)$. 
Indeed, an isomorphism is given by the pair of projections $p=(p_1,p_2)$ where $p_1(f)(w)=f(0,w)$ is evaluation at $u=0$ and $p_2$ is the projection defined in \cite[(102)]{Roberts-Schmidt_Bessel}. Since $p_2(f)$ only depends on $f(u,w)$ for large values of $u$, we obtain an isomorphism.
The action of $T\widetilde{T}N$ on $f\in C^\infty_c(k,V)\cong p_1(W/I_2)$ is given by the action on $W$ for $u=0$:
\begin{align*}
 &(s_{a,b,c}f)(w)    = \pi\left(\begin{smallmatrix}1&b\\0&1\end{smallmatrix}\right)f(w+c)\ ,\\
 &(\widetilde{t}f)(w)= \chi_\Pi(t_1t_2)\omega_\pi(t_1)|t_1/t_2|^{3/2} f(\tfrac{t_1}{t_2}w) \ ,\\
 &(x_\lambda f)(w)   = \chi_\Pi(\lambda)\pi(\diag(\lambda,1)) f(\lambda^{-1}w)                \ .
\end{align*}
The second factor $p_2(W/I_2)$ is isomorphic to $p_1(W/I_2)$ as a $TS$-module via conjugation in $I$ with the Weyl element ${\mathbf{s}_1}$. The action of $\widetilde{T}$ and $\widetilde{N}$ is the same up to interchanging $t_1$, $t_2$, and $a$, $c$, respectively.

\item $I_0$ is the restriction of the $P$-module $I_{\cell}^0/I_{\cell}^1$ to $T\widetilde{T}N$.
It is isomorphic to $V$ with the trivial action of $N$ and $T\widetilde{T}$ acting on $v\in V$ by
\begin{align*}
 &\widetilde{t}v     = \chi_\Pi(t_1t_2)\pi(\diag(t_1,t_2))v \ ,\\
 &x_\lambda v        = |\lambda|^{3/2}\chi_\Pi(\lambda) \omega_\pi(\lambda)v  \ .
\end{align*}
\end{enumerate}

\textit{$\widetilde{N}$-coinvariants.}
The functor of $\widetilde{N}$-coinvariants is exact, so the $T\widetilde{T}S$-module $I_{\widetilde{N}}$ admits a filtration with quotients isomorphic to $(I_i)_{\widetilde{N}}$. It is straightforward to show that these are explicitly given by

\begin{enumerate}
\item $(I_3)_{\widetilde{N}}\cong C_c^\infty(k,V)$ with action on $f\in C_c^\infty(k,V)$ by
\begin{align*}
 &(s_{0,b,0}f)(y)    = f(y+b)\ , \\
 &(\widetilde{t}f)(y)= \pi(\diag(t_2,t_1))\chi_\Pi(t_1t_2)f(y)\ ,\\
 &(x_\lambda f)(y)   =|\lambda|^{1/2}\chi_\Pi(\lambda)f(\lambda^{-1}y)\ .
\end{align*}
\item $(I_2)_{\widetilde{N}}\cong C_c^\infty(k^\times ,V_{\left(\begin{smallmatrix}1&\ast\\0&1\end{smallmatrix}\right)})$ with action on $f\in C_c^\infty(k^\times ,V_{\left(\begin{smallmatrix}1&\ast\\0&1\end{smallmatrix}\right)})$ by
\begin{align*}
 &(s_{0,b,0}f)(u)    =  f(u)\ ,\\
 &(\widetilde{t}f)(u)= \chi_\Pi(t_1t_2)\omega_\pi(t_1)\left\vert t_1/t_2\right\vert^{1/2} f(\tfrac{t_1}{t_2}u) \ ,\\
 &(x_\lambda f)(u)      = \vert\lambda\vert\chi_\Pi(\lambda)\pi_{(\begin{smallmatrix}1&\ast\\0&1\end{smallmatrix})}(\lambda,1) f(u)   \ .
\end{align*}
\item $(I_1)_{\widetilde{N}}\cong V\oplus V$ with action on $v\in V$ in the first factor by
\begin{align*}
 &(s_{0,b,0}v)       = \pi\left(\begin{smallmatrix}1&b\\0&1\end{smallmatrix}\right)v\ ,\\
 &(\widetilde{t}v)   = |t_1/t_2|^{1/2}\chi_\Pi(t_1t_2)\omega_\pi(t_1) v                 \ ,\\
 &(x_\lambda f)(w)   = \vert\lambda\vert \chi_\Pi(\lambda)\pi(\diag(\lambda,1)) v               \ ,
\end{align*}
and the same action on the second factor up to interchanging $t_1$ and $t_2$.
\item $(I_0)_{\widetilde{N}}\cong I_0$ with the same action of $T\widetilde{T}S$ as on $I_0$.
 \end{enumerate}

\subsection{Combinatorics}\label{s:Combinat}

In this section we study $\Gl(1)$-characters that occur as $T$-characters in certain quotients of the normalized Siegel-Jacquet module $\delta_P^{1/2}\otimes J_P(\Pi)\neq0$ of an irreducible $\Pi\in \CCC_G(\omega)$ with central character $\omega$.
Let $$\Delta_0(\Pi)\subseteq\widetilde{\Delta}(\Pi)\subseteq{\Delta}(\Pi)$$ denote
the finite multisets of $\Gl(1)$-characters by which $x_\lambda$ acts on the irreducible constituents of
$\delta_P^{-1/2}\otimes J_P(\Pi)_\psi$, $\delta_P^{-1/2}\otimes \pi_0(\widetilde\Pi)$ and $\delta_{P}^{-1/2}\otimes J_P(\Pi)$.\footnote{Recall that $\pi_0(\widetilde\Pi) \cong J_P(\Pi)_{\widetilde{T},\rho}$, see lemma~\ref{lem:Bessel_Jacquet}.}
This, of course, may depend on the choice of the underlying Bessel character that defines $\widetilde\Pi$.
We see in lemma~\ref{lem:TildeDelta} that it does not depend on this choice as long as $\widetilde\Pi$ provides a split Bessel model.
Let $\Delta_1(\Pi)$ denote the multiset of characters by which $x_\lambda$ acts on the constituents of $\beta^{\rho}(\Pi)$ for the same Bessel character. Prop.~\ref{BETA1} asserts a union of multisets
$$\Delta_0(\Pi)\sqcup \Delta_1(\Pi)=\widetilde{\Delta}(\Pi)\ .$$
Define $$\Delta_\pm(\Pi)=\nu^{\pm1/2}\Delta_0(\Pi)\ .$$
These multisets provide interesting information about the structure of
$\Pi$.
We have shown in lemma \ref{ugly}:
\begin{lem}\label{lem:TildeDelta}
For irreducible $\Pi\in\mathcal{C}_G(\omega)$, normalized as in table~\ref{tab:List}, and $\widetilde{\Delta}(\Pi)$ defined with respect to a Bessel character that provides a split Bessel model,
$$\widetilde{\Delta}(\Pi)=
\begin{cases}
\{\nu^{ 1/2},\nu^{ 1/2}\} & \text{case \nosf{VIa}\,,}\\
\{\nu^{-1/2},\nu^{-1/2}\} & \text{case \nosf{VId}\,,}\\
\Delta(\Pi) & \text{otherwise.}
\end{cases}$$
For generic $\Pi$ we have $\widetilde{\Delta}(\Pi)=\Delta_0(\Pi)$.
\end{lem}

\begin{lem}
For irreducible representations $\Pi\in\mathcal{C}_G(\omega)$ with non-zero Siegel-Jacquet module the corresponding multisets are listed in table~\ref{tab:Delta}.\footnote{In order to make our results explicit, we always assume that $\Pi$ is normalized as in table~\ref{tab:List}. One easily derives the general case by a suitable twist.}
\end{lem}
\begin{proof}
The constituents of $\Delta(\Pi)$ are given in table~A.3 of \cite{Roberts-Schmidt} and $\widetilde{\Delta}(\Pi)$ is given by lemma~\ref{lem:TildeDelta}. The rest is straightforward.
\end{proof}

For the involution $\rho\mapsto\rho^\divideontimes=\omega\rho^{-1}$ and multisets $M$ as above define $$M^\divideontimes=\{\rho^\divideontimes\,|\,\rho\in M\}\ .$$
Attached to every irreducible $\Pi$ is the group $A(\Pi)=\{\chi\,|\,\chi\otimes\Pi\cong\Pi\}$ of $T$-characters that preserve $\Pi$ under twisting.
The involution $\rho\mapsto\rho^\divideontimes$ generates orbits $\{\rho,\rho^\divideontimes\}$ of cardinality one or two.

\begin{lem}For irreducible $\Pi$ with $J_P(\Pi)\neq0$ the group
$A(\Pi)$ is trivial except for twists of the cases \nosf{Va} and \nosf{Vd} where it is $A(\Pi)=\{1,\chi_0\}$.
\end{lem}
\begin{proof}
For every $\chi\in A(\Pi)$ we have $\chi\Delta(\Pi)=\Delta(\Pi)$. This already implies $\chi=1$ except for the cases \nosf{Va}, \nosf{Vd} where it implies $\chi\in\{1,\chi_0\}$.
Up to semisimplification, the Borel induced representation $\chi_0\times\chi_0\rtimes\chi_\Pi$ is invariant under twists with the quadratic character $\chi_0$. Its unique essentially square-integrable constituent of type \nosf{Va} must be preserved under this twist.
The constituents of type \nosf{Vb} and \nosf{Vc} are $\chi_0$-twists of each other, so \nosf{Vd} is also preserved.
\end{proof}

For irreducible representations $\Pi\in\CCC_G(\omega)$ tables \ref{tab:List} and \ref{tab:Delta} imply
\begin{lem} \label{Besselexist}
$\Pi$ has a split Bessel model if and only if $\Delta_0(\Pi)$ is non-empty.
\end{lem}

\begin{lem}\label{lem:intersection--*}
For generic irreducible $\Pi$ %
the intersection $\Delta_{\pluss}(\Pi)\cap\Delta_{\pluss}^\divideontimes(\Pi)$ is empty.
For non-generic irreducible $\Pi$ we have $\Delta_{\pluss}(\Pi)=\Delta_{\pluss}^\divideontimes(\Pi)$
and the involution $\rho\mapsto\rho^\divideontimes$ acts transitively on $\Delta_{\pluss}(\Pi)$.
\end{lem}

\begin{lem}\label{lem:intersection_+-}
The intersection $\Delta_{\minuss}(\Pi)\cap\Delta_{\pluss}(\Pi)$ is empty except for twists of
\begin{center}
\begin{tabular}{ccccc}
\toprule
Type      & \nosf{IIIa}& \nosf{IIIa} & \nosf{IIIb} & \nosf{IIIb}  \\
with       & $\chi_1=\nu$ & $\chi_1=\nu^{-1}$ & $\chi_1=\nu$ & $\chi_1=\nu^{-1}$\\
$\Delta_{\minuss}(\Pi)\cap \Delta_{\pluss}(\Pi)$ & $\{\nu\}$      & $\{1\}$     & $\{1\}$ & $\{\nu^{-1}\}$\ .\\
\bottomrule
\end{tabular}
\end{center}
\end{lem}

\begin{lem}\label{lem:exceptional++*}
The intersection $\Delta_{\minuss}(\Pi)\cap\Delta_{\minuss}^\divideontimes(\Pi)$ is empty except for twists of
\begin{center}
\begin{tabular}{ccccccc}
\toprule
Type      & \nosf{IIa} & \nosf{IIIa} & \nosf{IVc} & \nosf{Va} & \nosf{VIa} & \nosf{XIa} \\
$\Delta_{\minuss}(\Pi)\cap\Delta_{\minuss}^\divideontimes(\Pi)$ & $\{1\}$ & $\{1,\chi_1\}$ & $\{1\}$ & $\{1,\chi_0\}$ & $\{1,1\}$& $\{1\}$\ .\\
\bottomrule
\end{tabular}
\end{center}

In these cases, the intersection is an orbit with basepoint $\rho=1$ under the joint action of $A(\Pi)$ and the involution $\rho\mapsto\rho^\divideontimes$.
\end{lem}

\begin{lem}\label{lem:exceptional+-*}
If there is $\rho\in\Delta_{\minuss}(\Pi)\cap\Delta_{\pluss}^\divideontimes(\Pi)$
with $\rho\notin\Delta_{\minuss}^\divideontimes(\Pi)\cup\Delta_{\pluss}(\Pi)$, then $(\Pi,\rho)$ belongs to a twist of one of the following cases
\begin{center}
\begin{tabular}{cccc}
\toprule
Type   & \nosf{I} & \quad\nosf{IIa} & \quad\nosf{X}\\
 $\nu^{\frac12}\rho\in$ 
 & $\{1,\chi_1,\chi_2,\chi_1\chi_2\}$ 
 & $\{\chi_1,\chi_1^{-1}\}$ 
 & $\{1,\omega_{\pi_{c}}\}$\ .\\
 \bottomrule
\end{tabular}
\end{center}
If there exists
$\rho\in\Delta_{\minuss}^\divideontimes(\Pi)\cap \Delta_{\pluss}(\Pi)\cap \Delta_{\minuss}(\Pi)$, then
$(\Pi,\rho)$ is a twist of case \nosf{IIIa} with $(\chi_1,\rho)=(\nu,\nu)$ or $(\chi_1,\rho)=(\nu^{-1},1)$.
\end{lem}

\begin{prop}\label{prop:comb_cases}
Fix a generic irreducible $\Pi$. In each orbit $\{\rho,\rho^\divideontimes\}$ we can choose $\rho$ such that exactly one of the following assertions holds:
\begin{enumerate}
\item $\rho$ is not in $\Delta_{\pluss}(\Pi)\cup \Delta_{\minuss}(\Pi)$,
\item $\rho$ is in $\Delta_{\pluss}^\divideontimes(\Pi)\cap\Delta_{\minuss}(\Pi)$ and not in $\Delta_{\minuss}^{\divideontimes}(\Pi)\cup\Delta_{\pluss}(\Pi)$,
\item $\rho$ is in $\Delta_{\minuss}(\Pi)\cap\Delta_{\minuss}^\divideontimes(\Pi)$ and not in $\Delta_{\pluss}(\Pi)$.
\end{enumerate}
Especially, $\rho\notin\Delta_{\pluss}(\Pi)$ holds in every case.
\end{prop}
\begin{proof}
Assume the first assertion cannot be satisfied, so $\rho$ is both in  $\Delta_{\pluss}(\Pi)\cup \Delta_{\minuss}(\Pi)$ and in $\Delta_{\pluss}^\divideontimes(\Pi)\cup \Delta_{\minuss}^\divideontimes(\Pi)$.
If $\rho$ is not in $\Delta_{\minuss}(\Pi)\cap\Delta_{\minuss}^\divideontimes(\Pi)$, then $\rho$ is in $\Delta_{\pluss}(\Pi)$ or in $\Delta_{\pluss}^\divideontimes(\Pi)$.
By choice of $\rho$ we can assume that $\rho\in \Delta_{\pluss}^\divideontimes(\Pi)$.
Lemma \ref{lem:intersection--*} assures $\rho\notin \Delta_{\pluss}(\Pi)$,
hence $\rho\in\Delta_{\minuss}(\Pi)$.
Especially, $\rho\notin\Delta_{\minuss}^\divideontimes(\Pi)\cup \Delta_{\pluss}(\Pi)$, so the second assertion holds.
On the other hand, if $\rho$ is in $\Delta_{\minuss}(\Pi)\cap\Delta_{\minuss}^\divideontimes(\Pi)$, then the same holds for $\rho^\divideontimes$. By possibly replacing $\rho$ by $\rho^\divideontimes$ we can assume $\rho\notin\Delta_{\pluss}(\Pi)$ by lemma~\ref{lem:intersection--*}, so the third assertion holds.
\end{proof}

We distinguish three series of exceptional cases for the pairs $(\Pi,\rho)$.
\begin{cor}\label{cor:comb_cases2}
Suppose  $\Pi\in\CCC_G(\omega)$ is generic and irreducible, $J_P(\Pi)\neq0$ and $\Pi$ is normalized as in table~\ref{tab:List}.
In each orbit $\{\rho,\rho^\divideontimes\}$ one can choose $\rho$ such that exactly one of the following assertions holds:
 \begin{enumerate}
\item $(\Pi,\rho)$ belongs to the \textbf{non-exceptional cases} where $\rho\notin\Delta_{\pluss}(\Pi)\cup\Delta_{\minuss}(\Pi)$.
\item $(\Pi,\rho)$ belongs to the \textbf{fully induced non-ordinary exceptional cases}
\begin{enumerate}
\item[\quad\nosf{I}]  and $\rho$ is in $\{\nu^{-\frac12},\nu^{-\frac12}\chi_1,\nu^{-\frac12}\chi_2,\nu^{-\frac12}\chi_1\chi_2\}$,
\item[\quad\nosf{IIa}]and $\rho$ is in $\{\nu^{-\frac12}\chi_1,\nu^{-\frac12}\chi_1^{-1}\}$,
\item[\quad\nosf{X}]  and $\rho$ is in $\{\nu^{-\frac12},\nu^{-\frac12}\omega_{\pi_{c}}\}$.
\end{enumerate}
\item $(\Pi,\rho)$ belongs to the \textbf{extraordinary exceptional cases} \nosf{IIa}, \nosf{Va}, \nosf{VIa}, \nosf{XIa} with $\rho=\rho^\divideontimes\in\Delta_{\minuss}(\Pi)$.
\item $(\Pi,\rho)$ belongs to the \textbf{exceptional cases \nosf{IIIa}} where $\Pi$ is of type \nosf{IIIa} and $\rho\in \Delta_{\minuss}(\Pi)\cap\Delta_{\minuss}^\divideontimes(\Pi)=\{1,\chi_1\}$,
\end{enumerate}
and in each case $\rho\notin\Delta_{\pluss}(\Pi)$.
\end{cor}
\begin{proof} Use proposition~\ref{prop:comb_cases}. The second assertion is implied by lemma~\ref{lem:exceptional+-*}.
The third and fourth assertion correspond to the third assertion of proposition~\ref{prop:comb_cases} by lemma~\ref{lem:exceptional++*}.
\end{proof}

Recall that for an irreducible $M$-module $\sigma_\Pi=\pi\boxtimes\chi_\Pi$ as in section \ref{s:SiegelInd} we have defined $\rho_+(\sigma_\Pi)=\nu^{-1/2}\chi_\Pi$ and $\rho_-(\sigma_\Pi)=\rho_+^\divideontimes(\sigma_\Pi)$.

\begin{lem}\label{lem:combinatoric}
Fix an infinite-dimensional irreducible $\sigma_\Pi\in\CCC_G(\omega)$ such that $\sigma_\Pi\not\cong\omega\otimes\sigma_\Pi^\vee$
and such that there is an exact sequence
\begin{equation*}0\to\Xi\to \Ind_P^G(\sigma_\Pi)\to\Pi\to0
\end{equation*}
with generic irreducible quotient $\Pi$ and irreducible submodule $\Xi$. Then
$\Delta_{\pluss}(\Xi)=\{\rho_+(\sigma_\Pi), \rho_-(\sigma_\Pi)\}$ as sets, i.e.\ without counting multiplicity, and this is disjoint to $\{\rho_+(\omega\otimes\sigma^\vee_\Pi),\rho_-(\omega\otimes\sigma^\vee_\Pi)\}$.
\end{lem}
\begin{proof}
See tables~\ref{tab:List}, \ref{tab:list_sigma}, \ref{tab:Delta}.
\end{proof}%

For irreducible $\Pi\in\CCC_G(\omega)$ and Bessel characters that provide a split Bessel model for $\Pi$,
the multisets $\widetilde{\Delta}(\Pi)$ and  $\Delta_1(\Pi)$ are independent of the specific choice of the Bessel character by lemma~\ref{lem:TildeDelta}.

\begin{lem}\label{lem:Delta_1}
For irreducible $\Pi\in\mathcal{C}_G^{\fin}(\omega)$ and Bessel characters providing a split Bessel model for $\Pi$,
the multiset $\Delta_1(\Pi)$ is empty except for twists of

\begin{center}
\begin{tabular}{cccccc}
\toprule
 Type           & \nosf{IIb} & \nosf{IVb} & \nosf{Vb} & \nosf{Vc} & \nosf{VId} \\
 $\Delta_1(\Pi)$& $\{\chi_1,\chi_1^{-1}\}$ & $\{\nu^{3/2}\}$ & $\{\chi_0\nu^{1/2}\}$ & $\{\nu^{1/2}\}$ & $\{\nu^{-1/2}\}$\ .\\
 \bottomrule
\end{tabular}
\end{center}
\end{lem}
\begin{proof}See lemma~\ref{lem:TildeDelta} and table~\ref{tab:Delta}.
\end{proof}

\begin{lem}\label{lem:fatcell_disjointness}
If $\sigma_\Pi\in\CCC_M(\omega)$ is one-dimensional with an exact sequence 
\begin{equation*}
 0\to\Xi\to \Ind_P^G(\sigma_\Pi)\to\Pi\to0\ ,
\end{equation*}
where the irreducible quotient $\Pi$ has a split Bessel model, then $\Delta_1(\Pi)\cap\Delta(\Xi)$ is empty.
\end{lem}
\begin{proof} By a twist we can assume that $\sigma_\Pi$ is normalized as in table~\ref{tab:List}.
The constituents $\Pi$ and $\Xi$ are then given by theorem~\ref{thm:sigma_Pi_existence} and table~\ref{tab:list_sigma}.
The statement follows from table~\ref{tab:Delta}.
\end{proof}

\begin{table}
\caption{Irreducible representations $\Pi$ of $\GSp(4,k)$ with $J_P(\Pi)\neq0$.\label{tab:List}}
\begin{center}
\begin{tabular}{llllllll}
\toprule
Type & $\sigma_{\Pi}$            & ${\omega}$ & Conditions       & $\rho$ \\
\midrule
\nosf{I}
& $(\chi_1\times\chi_2)\boxtimes1, $ & $\chi_1\chi_2$ & $\chi_1,\chi_2,\chi_1\chi_2,\chi_1\chi_2^{-1}\neq\nu^{\pm1}$  & all\\
\nosf{IIa}
& $\Sp(\chi_1)\boxtimes \chi_1^{-1}$               & $1$ & $\chi_1^2\neq \nu^{\pm1}$ and $\chi_1\neq\nu^{\pm3/2}$  & all \\
\nosf{IIb}
& $(\chi_1\circ \det) \boxtimes \chi_1^{-1}$      & $1$ & $\chi_1^2\neq \nu^{\pm1}$ and $\chi_1\neq\nu^{\pm3/2}$  & $1$ \\

\nosf{IIIa}
& $(\chi_1^{-1}\times\nu^{-1})\boxtimes\nu^{1/2}\chi_1$& $\chi_1$&$\chi_1\neq1,\nu^{\pm2},\nu^{-1}$  & all \\
\nosf{IIIb}
& $(\chi_1 \times \nu) \boxtimes \nu^{-1/2}$      & $\chi_1$ &  $\chi_1\neq 1,\nu^{\pm2},\nu^{-1}$ & $\chi_1$, $1$ \\
\nosf{IVa}
& $\Sp(\nu^{-3/2})\boxtimes\nu^{3/2}$              & $1$ &                           & all \\
\nosf{IVb}
& $(\nu^{-3/2}\circ\det)\boxtimes\nu^{3/2}$       & $1$ &                       & $1$ \\
\nosf{IVc}
& $\Sp(\nu^{3/2})\boxtimes\nu^{-3/2}$              & $1$ &                   & $\nu^{\pm1}$ \\
\nosf{IVd}
& $(\nu^{3/2}\circ\det)\boxtimes\nu^{-3/2}$       & $1$ &                    & none \\
\nosf{Va}
&$\Sp(\nu^{-1/2}\chi_0)\boxtimes\nu^{1/2}$         & $1$ & $\chi_0^2=1\neq\chi_0$      & all \\
\nosf{Vb}
& $\Sp(\nu^{1/2}\chi_0) \boxtimes \nu^{-1/2}$      & $1$ & $ \chi^2_0 =1\neq\chi_0$    & $1$ \\
\nosf{Vc}
& $\Sp(\nu^{1/2}\chi_0) \boxtimes \chi_0\nu^{-1/2}$& $1$ & $ \chi^2_0 =1\neq\chi_0$    & $\chi_0$ \\
\nosf{Vd}
& $(\nu^{1/2}\chi_0\circ\det)\boxtimes\nu^{-1/2}$ & $1$ & $\chi_0^2=1\neq\chi_0$      & none \\
\nosf{VIa}
& $\Sp(\nu^{-1/2})\boxtimes\nu^{1/2}$              & $1$ &                             & all \\
\nosf{VIb}
& $(\nu^{-1/2}\circ\det)\boxtimes\nu^{1/2}$       & $1$ &                             & none \\
\nosf{VIc}
& $\Sp( \nu^{1/2}) \boxtimes \nu^{-1/2}$           & $1$ &                             & $1$ \\
\nosf{VId}
& $(\nu^{1/2}\circ \det) \boxtimes \nu^{-1/2}$    & $1$ &                             & $1$ \\
\nosf{X}
& $\pi_{c}\boxtimes1$           & $\omega_{\pi_{c}}$  & $\omega_{\pi_{c}}\neq\nu^{\pm1}$  &all \\
\nosf{XIa}
& $\nu^{-1/2}\pi_{c}\boxtimes\nu^{1/2}$        & $1$ &  $\omega_{\pi_{c}}=1$    & all \\
\nosf{XIb}
& $\nu^{1/2}\pi_{c} \boxtimes \nu^{-1/2}$      & $1$ &  $\omega_{\pi_{c}} = 1$  & $1$ \\
\bottomrule
\end{tabular}
\end{center}

Up to a possible twist, these are the irreducible representations $\Pi\in\CCC_G$ with central character $\omega$ and non-zero Siegel-Jacquet module.
Each $\Pi$ is defined as the unique quotient of $I\!=\!\Ind_P^G(\sigma_\Pi)$ with an irreducible $\sigma_\Pi\in\CCC_M(\omega)$
subject to the formulated conditions.
 $\pi_c$ denotes a cuspidal irreducible smooth $\Gl(2)$-module.
The special representation $\Sp(\chi)$ for a $\Gl(1)$-character $\chi$ is the unique quotient of $\nu^{-1/2}\chi\times\nu^{1/2}\chi\in\CCC_{\Gl(2)}$.
For types \nosf{I}, \nosf{IIa}, \nosf{IIb} and \nosf{X}, $\Pi\!=\! I$ is irreducible.
The last column lists the smooth $\Gl(1)$-characters $\rho$ such that $\Pi$ admits a Bessel model attached to the character $\Lambda\!=\!\rho\boxtimes\rho^\divideontimes$ of the split torus $\widetilde{T}\cong k^\times\times k^\times$ as in section~\ref{s:preliminaries} where $\rho^\divideontimes=\omega\rho^{-1}$, see theorem~\ref{Mult1}.
\end{table}
 
\begin{table}
\caption{Siegel induced representations \label{tab:list_sigma}}

\begin{small}
\begin{center}
\begin{tabular}{lllll}
\toprule
$\Pi$ & $\sigma_\Pi$                  & $\Xi^{ss}$ & $\rho_+(\sigma_\Pi)$ & $\rho_-(\sigma_\Pi)$ \\
\midrule
\nosf{I} 
& $(\chi_1\times\chi_2)\boxtimes1, $                    & $0$  & $\nu^{-1/2}$ & $\nu^{1/2}\chi_1\chi_2$\\
& $(\chi_1^{-1}\times\chi_2)\boxtimes\chi_1, $          & $0$  & $\nu^{-1/2}\chi_1$ & $\nu^{1/2}\chi_2$\\
& $(\chi_1\times\chi_2^{-1})\boxtimes\chi_2, $          & $0$  & $\nu^{-1/2}\chi_2$ & $\nu^{1/2}\chi_1$\\
& $(\chi_1^{-1}\times\chi_2^{-1})\boxtimes\chi_1\chi_2$ & $0$  & $\nu^{-1/2}\chi_1\chi_2$ & $\nu^{1/2}$\\
\nosf{IIa} 
& $\Sp(\chi_1)\boxtimes \chi_1^{-1}$                     & $0$  & $\nu^{-1/2}\chi_1^{-1}$ & $\nu^{1/2}\chi_1$\\
& $\Sp(\chi_1^{-1})\boxtimes \chi_1$                     & $0$  & $\nu^{-1/2}\chi_1$ & $\nu^{1/2}\chi_1^{-1}$\\
& $\nu^{-1/2}(\chi_1\times\chi_1^{-1})\boxtimes\nu^{1/2}$ & \nosf{IIb}& $1$ & $1$\\
\nosf{IIb} 
& $(\chi_1\circ\det)\boxtimes \chi_1^{-1}$              & $0$  & $\nu^{-1/2}\chi_1^{-1}$ & $\nu^{1/2}\chi_1$\\
& $(\chi_1^{-1}\circ\det)\boxtimes \chi_1$              & $0$  & $\nu^{-1/2}\chi_1$ & $\nu^{1/2}\chi_1^{-1}$\\
& $\nu^{1/2}(\chi_1\times\chi_1^{-1})\boxtimes\nu^{-1/2}$ & \nosf{IIa}& $\nu^{-1}$ & $\nu$\\
\nosf{IIIa} 
& $(\chi_1^{-1}\times\nu^{-1})\boxtimes \nu^{1/2}\chi_1$& \nosf{IIIb} & $\chi_1$    & $1$\\
& $(\chi_1\times\nu^{-1})\boxtimes \nu^{1/2}$           & \nosf{IIIb} & $1$         & $\chi_1$\\
\nosf{IIIb} 
& $(\chi_1^{-1}\times\nu)\boxtimes \nu^{-1/2}\chi_1$    & \nosf{IIIa} & $\nu^{-1}\chi_1$ & $\nu$\\
& $(\chi_1\times\nu)\boxtimes \nu^{-1/2}$               & \nosf{IIIa} & $\nu^{-1}$  & $\nu\chi_1$\\
\nosf{IVa} 
&  $\Sp(\nu^{-3/2})\boxtimes\nu^{3/2}$                  & \nosf{IVc}  & $\nu$       & $\nu^{-1}$\\
\nosf{IVb}
& $(\nu^{-3/2}\circ\det)\boxtimes \nu^{3/2}$            & \nosf{IVd}  & $\nu$       & $\nu^{-1}$\\
& $(\nu^2\times\nu^{-1})\boxtimes\nu^{-1/2}$            & \nosf{IVa}, \nosf{IVc}, \nosf{IVd} & $\nu^{-1}$  & $\nu$\\
\nosf{IVc}
& $\Sp(\nu^{3/2})\boxtimes \nu^{-3/2}$                  & \nosf{IVa}  & $\nu^{-2}$  & $\nu^{2}$\\
& $(\nu^{-2}\times\nu)\boxtimes\nu^{1/2}$               & \nosf{IVa}, \nosf{IVb}, \nosf{IVd}  & $1$  & $1$\\
\nosf{IVd} 
& $(\nu^{3/2}\circ\det)\boxtimes \nu^{-3/2}$            & \nosf{IVb}   & $\nu^{-2}$  & $\nu^2$\\
\nosf{Va}
& $\Sp(\nu^{-1/2}\chi_0)\boxtimes\nu^{1/2}$             & \nosf{Vb}    & $1$         & $1$\\
& $\Sp(\nu^{-1/2}\chi_0)\boxtimes\chi_0\nu^{1/2}$       & \nosf{Vc}    & $\chi_0$    & $\chi_0$\\
\nosf{Vb}
& $\Sp(\nu^{1/2}\chi_0)\boxtimes\nu^{-1/2}$             & \nosf{Va}    & $\nu^{-1}$  & $\nu$\\
& $(\nu^{-1/2}\chi_0\circ\det)\boxtimes\nu^{1/2}\chi_0$ & \nosf{Vd}    & $\chi_0$    & $\chi_0$\\
\nosf{Vc}
& $\Sp(\nu^{1/2}\chi_0)\boxtimes\nu^{-1/2}\chi_0$       & \nosf{Va}    & $\nu^{-1}\chi_0$ & $\nu\chi_0$\\
& $(\nu^{-1/2}\chi_0\circ\det)\boxtimes\nu^{1/2}$       & \nosf{Vd}    & $1$         & $1$\\
\nosf{Vd}
& $(\nu^{1/2}\chi_0\circ\det)\boxtimes\nu^{-1/2}\chi_0$ & \nosf{Vb}    & $\nu^{-1}\chi_0$ & $\nu\chi_0$\\
& $(\nu^{1/2}\chi_0\circ\det)\boxtimes\nu^{-1/2}$       & \nosf{Vc}    & $\nu^{-1}$  & $\nu$\\
\nosf{VIa}
& $\Sp(\nu^{-1/2})\boxtimes\nu^{1/2}$                   & \nosf{VIc}   & $1$         & $1$\\
\nosf{VIb}
& $(\nu^{-1/2}\circ\det)\boxtimes\nu^{1/2}$             & \nosf{VId}   & $1$         & $1$\\
\nosf{VIc}
& $\Sp(\nu^{1/2})\boxtimes\nu^{-1/2}$                   & \nosf{VIa}   & $\nu^{-1}$  & $\nu$\\
\nosf{VId}
& $(\nu^{1/2}\circ\det)\boxtimes\nu^{-1/2}$             & \nosf{VIb}   & $\nu^{-1}$  & $\nu$\\
\nosf{X} 
& $\pi_{c}\boxtimes1$                                   & $0$  & $\nu^{-1/2}$& $\nu^{1/2}\omega_{\pi_{c}}$\\
& $\pi_{c}^\vee\boxtimes\omega_{\pi_{c}}$               & $0$  & $\nu^{-1/2}\omega_{\pi_{c}}$ & $\nu^{1/2}$\\
\nosf{XIa}
& $\nu^{-1/2}\pi_{c}\boxtimes\nu^{1/2}$                 & \nosf{XIb}   & $1$         & $1$\\
\nosf{XIb}
& $\nu^{1/2}\pi_{c}\boxtimes\nu^{-1/2}$                 & \nosf{XIa}   & $\nu^{-1}$  & $\nu$\\
\bottomrule
\end{tabular}
\end{center}
\end{small}
For each irreducible $\Pi\in\CCC_G$ with $J_P(\Pi)\neq0$, normalized as in table~\ref{tab:List}, these are the irreducible $\sigma_\Pi=\pi\boxtimes\chi_\Pi\in \CCC_M(\omega)$ such that $\Pi$ is a quotient of the Siegel induced representation $I=\Ind_P^G(\sigma_\Pi)$. The third column lists the constituents of the kernel $\Xi=\ker(I\to \Pi)$.
The last columns list the characters $\rho_+\!=\!\nu^{-1/2}\chi_\Pi$ and $\rho_-\!=\!\nu^{1/2}{\omega}\chi_\Pi^{-1}$.
\end{table}

\begin{landscape}
\begin{table}
\caption{Multisets\label{tab:Delta}}

\begin{center}
\begin{tabular}{llllllll}
\toprule
$\Pi$ & $\Delta(\Pi)$  & $\widetilde{\Delta}(\Pi)$    &$\Delta_0(\Pi)$   & $\Delta_1(\Pi)$  & $\Delta_{\pluss}(\Pi)$ & $\Delta_Q(\Pi)$ & $\omega$\\
      &                & if $\deg\widetilde{\Pi}=1$   &                   & if $\deg\widetilde{\Pi}=1$                  \\
\midrule
\nosf{I}   & $\{1,\chi_1,\chi_2,\chi_1\chi_2\}$ & $\{1,\chi_1,\chi_2,\chi_1\chi_2\}$ & $\{1,\chi_1,\chi_2,\chi_1\chi_2\}$ & $\emptyset$ & $\nu^{1/2}\Delta_0(\Pi)$&$\{\chi_1,\chi_1^{-1},\chi_2,\chi_2^{-1}\}$& $\chi_1\chi_2$   \\
\nosf{IIa} & $\{\chi_1,\nu^{1/2},\chi_1^{-1}\}$ & $\{\chi_1,\nu^{1/2},\chi_1^{-1}\}$ & $\{\chi_1,\nu^{1/2},\chi_1^{-1}\}$ & $\emptyset$ & $\nu^{1/2}\Delta_0(\Pi)$
&$\{\chi_1\nu^{1/2},\chi_1^{-1}\nu^{1/2}\}$&$1$\\
\nosf{IIb} & $\{\chi_1,\nu^{-1/2},\chi_1^{-1}\}$ & $\{\chi_1,\nu^{-1/2},\chi_1^{-1}\}$&$\{\nu^{-1/2}\}$          & $\{\chi_1,\chi_1^{-1}\}$&$\{1\}$&$\{\chi_1\nu^{-1/2},\chi_1^{-1}\nu^{-1/2}\}$&$1$\\
\nosf{IIIa}& $\{\nu^{1/2}, \chi_1\nu^{1/2}\}$ & $\{\nu^{1/2},\chi_1\nu^{1/2}\}$&$\{\nu^{1/2},\chi_1\nu^{1/2}\}$&$\emptyset$&$\{\nu,\chi_1\nu\}$&$\{\chi_1,\chi_1^{-1},\nu\}$&$\chi_1$\\
\nosf{IIIb}& $\{\nu^{-1/2},\chi_1\nu^{-1/2}\}$    & $\{\nu^{-1/2},\chi_1\nu^{-1/2}\}$  &$\{\nu^{-1/2},\chi_1\nu^{-1/2}\}$  &$\emptyset$&$\{1,\chi_1\}$ &$\{\chi_1,\chi_1^{-1},\nu^{-1}\}$&$\chi_1$\\
\nosf{IVa} & $\{\nu^{3/2}\}$                     & $\{\nu^{3/2}\}$                 &$\{\nu^{3/2}\}$            &$\emptyset$     &$\{\nu^{2}\}$ &$\{\nu^2\}$       &$1$\\
\nosf{IVb} & $\{\nu^{3/2},\nu^{-1/2}\}$          & $\{\nu^{3/2},\nu^{-1/2}\}$      &$\{\nu^{-1/2}\}$           &$\{\nu^{3/2}\}$ &$\{1\}$       &$\{\nu^{-2},\nu\}$&$1$\\
\nosf{IVc} & $\{\nu^{-3/2},\nu^{1/2}\}$          & $\{\nu^{-3/2},\nu^{1/2}\}$      &$\{\nu^{-3/2},\nu^{1/2}\}$ &$\emptyset$   &$\{\nu,\nu^{-1}\}$&$\{\nu^2,\nu^{-1}\}$&$1$\\
\nosf{IVd} & $\{\nu^{-3/2}\}$                    &                                 &$\emptyset$                &                &$\emptyset$   &$\{\nu^{-2}\}$    &$1$\\
\nosf{Va}  & $\{\nu^{1/2},\chi_0\nu^{1/2}\}$  & $\{\nu^{1/2},\chi_0\nu^{1/2}\}$&$\{\nu^{1/2},\chi_0\nu^{1/2}\}$&$\emptyset$&$\{\nu,\chi_0\nu\}$&$\{\nu\chi_0\}$   &$1$\\
\nosf{Vb}  & $\{\nu^{-1/2},\chi_0\nu^{1/2}\}$   & $\{\nu^{-1/2},\chi_0\nu^{1/2}\}$ &$\{\nu^{-1/2}\}$           &$\{\chi_0\nu^{1/2}\}$ &$\{1\}$  &$\{\chi_0\}$    &$1$\\
\nosf{Vd}  & $\{\nu^{-1/2},\chi_0\nu^{-1/2}\}$   &                                 &$\emptyset$                &                 &$\emptyset$    &$\{\nu^{-1/2}\chi_0\}$&$1$\\
\nosf{VIa} & $\{\nu^{1/2},\nu^{1/2},\nu^{1/2}\}$ & $\{\nu^{1/2},\nu^{1/2}\}$       &$\{\nu^{1/2},\nu^{1/2}\}$  &$\emptyset$ &$\{\nu,\nu\}$&$\{\nu,1\}$&$1$\\
\nosf{VIb} & $\{\nu^{1/2}\}$        &                  &$\emptyset$           &           &$\emptyset$    &$\{1\}$&$1$\\
\nosf{VIc} & $\{\nu^{-1/2}\}$                          & $\{\nu^{-1/2}\}$                  &$\{\nu^{-1/2}\}$            &$\emptyset$                 & $\{1\}$          & $\{1\}$&$1$\\
\nosf{VId} & $\{\nu^{-1/2},\nu^{-1/2},\nu^{-1/2}\}$& $\{\nu^{-1/2},\nu^{-1/2}\}$        &$\{\nu^{-1/2}\}$            &$\{\nu^{-1/2}\}$  &$\{1\}$  & $\{1,\nu^{-1}\}$ & $1$\\
\nosf{X}   & $\{1,\omega_{\pi_c}\}$             & $\{1,\omega_{\pi_c}\}$             & $\{1,\omega_{\pi_c}\}$             &$\emptyset$  & $\{\nu^{1/2},\nu^{1/2}\omega_{\pi_c}\} $&$\emptyset$&$\omega_{\pi_{c}}$\\
\nosf{XIa} & $\{\nu^{1/2}\}$        & $\{\nu^{1/2}\}$  & $\{\nu^{1/2}\}$      &$\emptyset$      &$\{\nu\}$ &$\emptyset$&$1$\\
\nosf{XIb} & $\{\nu^{-1/2}\}$       & $\{\nu^{-1/2}\}$ & $\{\nu^{-1/2}\}$     &$\emptyset$      &$\{1\}$        &$\emptyset$&$1$\\
\bottomrule
\end{tabular}
\end{center}

For irreducible $\Pi\in\CCC_G(\omega)$ as in table~\ref{tab:List}, this is the list of multisets defined in section~\ref{s:Combinat}.
The multisets $\widetilde{\Delta}(\Pi)$ and $\Delta_1(\Pi)$ depend on the choice of a Bessel character.
The entries in the corresponding columns refer only to the case where the Bessel character defines a split Bessel model for $\Pi$.
In the cases left blank there is no split Bessel model.
\end{table}
\end{landscape}

\begin{table}
\caption{$T$-modules in the Bessel filtration}
\begin{center}
\begin{tabular}{llllll}
\toprule
Type& $\delta_P^{-1/2}\widetilde{I_3}$ & $\delta_P^{-1/2}\widetilde{I_2}^{ss}$, $\delta_P^{-1/2}\widetilde{I_{1}}$ &  $\delta_P^{-1/2}\widetilde{I_0}$& $\ \delta_P^{-1/2}\otimes\pi_0(\widetilde{\Pi})$ &  ${\omega}$ \\
\midrule
\nosf{I}   & $\boxed{1}$& $\boxed{\chi_1}\oplus\boxed{\chi_2}$ & $\boxed{\chi_1\chi_2}$  & $1,\chi_2,\chi_1,\chi_1\chi_2$ & $\chi_1\chi_2$\\
\nosf{IIa} & $\boxed{\chi_1^{-1}}$      & $\boxed{\nu^{1/2}}$   & $\boxed{\chi_1}$ & $\chi_1^{-1},\nu^{1/2},\chi_1$ & $1$ \\
\nosf{IIb} & $\boxed{\chi_1^{-1}}$      & $\boxed{\nu^{-1/2}}$    & $\boxed{\chi_1}$& $\chi_1^{-1},\nu^{-1/2}, \chi_1$ & $1$\\
\nosf{IIIa}& $\boxed{\nu^{1/2}\chi_1}$  & $\boxed{\nu^{1/2}}\oplus\chi_1\nu^{-1/2}$   & $\nu^{-1/2}$  & $\nu^{1/2}\chi_1$, $\nu^{1/2}$ & $\chi_1$\\
\nosf{IIIb}& $\boxed{\nu^{-1/2}}$ & $\boxed{\nu^{-1/2}\chi_1} \oplus \nu^{1/2}$ &  $\nu^{1/2}\chi_1$ & $\nu^{-1/2},\nu^{-1/2}\chi_1$& $\chi_1$\\
\nosf{IVa} & $\boxed{\nu^{3/2}}$  & $\nu^{1/2}$                & $\nu^{-3/2}$  &$\nu^{3/2}$                     &$1$\\
\nosf{IVb} & $\boxed{\nu^{3/2}}$  & $\boxed{\nu^{-1/2}}$         & $\nu^{-3/2}$  &$\nu^{3/2}$, $\nu^{-1/2}$       &$1$\\
\nosf{IVc} & $\boxed{\nu^{-3/2}}$ & $\boxed{\nu^{1/2}}$        & $\nu^{3/2}$ &$\nu^{-3/2},\nu^{1/2}$          &$1$\\
\nosf{IVd} \\ %
\nosf{Va}  & $\boxed{\nu^{1/2}}$  & $\boxed{\chi_0\nu^{1/2}}$  & $\nu^{-1/2}$  &$\nu^{1/2}$, $\chi_0\nu^{1/2}$  &$1$\\
\nosf{Vb}  & $\boxed{\nu^{-1/2}}$ & $\boxed{\chi_0\nu^{1/2}}$  & $\nu^{1/2}$ &$\nu^{-1/2}$, $\chi_0\nu^{1/2}$ &$1$\\
\nosf{Vc}  & $\boxed{\nu^{-1/2}\chi_0}$ & $\boxed{\nu^{1/2}}$    & $\chi_0\nu^{1/2}$ &$\chi_0\nu^{-1/2}$, $\nu^{1/2}$ &$1$\\
\nosf{Vd}  \\ %
\nosf{VIa} & $\boxed{\nu^{1/2}}$  & $\boxed{\nu^{1/2}}$          & $\nu^{-1/2}$  &$\nu^{1/2}$, $\nu^{1/2}$        &$1$\\
\nosf{VIb} \\ %
\nosf{VIc} & $\boxed{\nu^{-1/2}}$ & $\nu^{1/2}$                  & $ \nu^{1/2}$  &$\nu^{-1/2}$                    &$1$\\
\nosf{VId} & $\boxed{\nu^{-1/2}}$ & $ \boxed{\nu^{-1/2}}$        & $ \nu^{1/2}$  &$\nu^{-1/2},\nu^{-1/2}$         &$1$\\
\nosf{X}   & $\boxed{1}$          & $0$            & $\boxed{\omega_{\pi_{c}}}$  &$1$, $\omega_{\pi_{c}}$ &$\omega_{\pi_{c}}$\\
\nosf{XIa} & $\boxed{\nu^{1/2}}$  & $0$                          & $\nu^{-1/2}$  &$\nu^{1/2}$                     &$1$\\
\nosf{XIb} & $\boxed{\nu^{-1/2}}$ & $0$                          & $\nu^{1/2}$   &$\nu^{-1/2}$                    &$1$\\
\bottomrule
\end{tabular}
\end{center}

For $\Pi$ and $I$ as in table~\ref{tab:List} and Bessel characters $\rho$ that provide a split Bessel model to $\Pi$,
these are the normalized $T$-characters $\chi_{\norm}$
that occur as constituents in $\delta_P^{-1/2}\otimes \widetilde{I_i}$
and $\delta_P^{-1/2}\otimes \pi_0(\widetilde{\Pi})$.
The boxed entries contribute to $\delta_P^{-1/2}\otimes \pi_0(\widetilde{\Pi})$. 
In the cases left blank there is no split Bessel model.
\end{table}

\begin{table}\caption{Spinor $L$-factors (regular part) \label{tab:regular_poles}}
\begin{footnotesize}
\begin{center}
 \begin{tabular}{llll}
\toprule
Type & $\Pi\in\CCC_G(\omega)$                      &  $\rho$  & $L^{\mathrm{PS}}_\mathrm{reg}(s,\Pi,\Lambda,1)\ ,\ \  \Lambda=\rho\boxtimes\rho^\divideontimes$  split\\
\midrule
\nosf{I}    & $\chi_1\times\chi_2\rtimes\sigma$  &  all & $L(s,\sigma)L(s,\chi_1\sigma)L(s,\chi_2\sigma)L(s,\chi_1\chi_2\sigma)$\\
\nosf{IIa}  & $\chi \St\rtimes\sigma$            &  all & $L(s,\sigma)L(s,\chi^2\sigma)L(s,\nu^{1/2}\chi\sigma)$\\
\nosf{IIb}  &$\chi\mathbf{1}\rtimes \sigma$ & $\chi\sigma$    &$L(s,\sigma)L(s,\chi^2\sigma)L(s,\nu^{-1/2}\chi\sigma)$\\
\nosf{IIIa} &$\chi\rtimes \sigma \St$    & all                & $L(s,\nu^{1/2}\chi\sigma)L(s,\nu^{1/2}\sigma)$\\
\nosf{IIIb} &$\chi\rtimes \sigma\mathbf{1}$ & $\sigma,\chi\sigma$ & $L(s,\nu^{-1/2}\chi\sigma)L(s,\nu^{-1/2}\sigma)L(s,\nu^{1/2}\chi\sigma)L(s,\nu^{1/2}\sigma)$\\     
\nosf{IVa}  &$\sigma \St_{G}$            & all                & $L(s,\nu^{3/2}\sigma)$\\
\nosf{IVb}  &$L(\nu^2,\nu^{-1}\sigma \St)$& $\sigma$          & $L(s,\nu^{3/2}\sigma)L(s,\nu^{-1/2}\sigma)$\\
\nosf{IVc}  &$L(\nu^{3/2} \St,\nu^{-3/2}\sigma)$& $\nu^{\pm1}\sigma$ & $L(s,\nu^{1/2}\sigma)L(s,\nu^{-3/2}\sigma)L(s,\nu^{3/2}\sigma)$\\
\nosf{IVd}  &$\sigma\mathbf{1}_G$        & none               & ---\\
\nosf{Va}   &$\delta([\xi,\nu\xi],\nu^{-1/2}\sigma)$& all     & $L(s,\nu^{1/2}\sigma)L(s,\nu^{1/2}\xi\sigma)$\\
\nosf{Vb}   &$L(\nu^{1/2}\xi \St, \nu^{-1/2}\sigma)$ &$\sigma$& $L(s,\nu^{-1/2}\sigma)L(s,\nu^{1/2}\xi\sigma)$\\ 
\nosf{Vc}   &$L(\nu^{1/2}\xi \St, \nu^{-1/2}\xi\sigma)$ & $\xi\sigma$ & $L(s,\nu^{1/2}\sigma)L(s,\nu^{-1/2}\xi\sigma)$\\
\nosf{Vd}   &$L(\nu\xi,\xi\rtimes\nu^{-1/2}\sigma)$ & none    & ---\\
\nosf{VIa}  &$\tau(S,\nu^{-1/2}\sigma)$  & all                & $L(s,\nu^{1/2}\sigma)^2$\\
\nosf{VIb}  &$\tau(T,\nu^{-1/2}\sigma)$  & none               & ---\\
\nosf{VIc}  &$L(\nu^{1/2} \St,\nu^{-1/2}\sigma)$ & $\sigma$   & $L(s,\nu^{-1/2}\sigma)$\\
\nosf{VId}  &$L(\nu,1\rtimes\nu^{-1/2}\sigma)$ & $\sigma$     & $L(s,\nu^{-1/2}\sigma)^2$\\
\nosf{VII}  &$\chi\rtimes\pi$            & all                & $1$\\
\nosf{VIIIa}&$\tau(S,\pi)$               & all                & $1$\\
\nosf{VIIIb}&$\tau(T,\pi)$               & none               & ---\\
\nosf{IXa}  &$\delta(\nu\xi,\nu^{-1/2}\pi)$& all              & $1$\\
\nosf{IXb}  & $L(\nu\xi,\nu^{-1/2}\pi)$  & none               & ---\\
\nosf{X}    &$\pi\rtimes\sigma$          & all                & $L(s,\sigma)L(s,\omega_{\pi}\sigma)$\\
\nosf{XIa}  &$\delta(\nu^{1/2}\pi,\nu^{-1/2}\sigma)$&all      & $L(s,\nu^{1/2}\sigma)$\\
\nosf{XIb}  &$L(\nu^{1/2}\pi, \nu^{-1/2}\sigma)$ & $\sigma$   & $L(s,\nu^{-1/2}\sigma)$\\
            & cuspidal generic           & all                & $1$\\
            & cuspidal non-generic       & none               & ---\\
\bottomrule
\end{tabular}
\end{center}
\end{footnotesize}
For every irreducible smooth representation $\Pi$ of $G=\GSp(4,k)$, 
the third column lists the smooth characters $\rho$ of $k^\times$ such that the character $\Lambda=\rho\boxtimes\omega\rho^{-1}$ of $\widetilde{T}\cong k^\times\times k^\times$ yields a split Bessel model for $\Pi$.
The last column gives the regular part of Piatetskii-Shapiro's spinor $L$-factor attached to this split Bessel model.
The notation follows \cite{Sally_Tadic} and \cite{Roberts-Schmidt}.
For typographical reasons we set $\mu=1$;
the general case follows from the identity $L^{\mathrm{PS}}_{\mathrm{reg}}(s,\Pi,\mu,\Lambda)= L^{\mathrm{PS}}_{\mathrm{reg}}(s,\mu\otimes\Pi,1,\Lambda)$.

\end{table}

\begin{table}
\begin{small}
\begin{center}
\caption{Normalized Bessel modules of degree one.\label{tab:Bessel_module_split_Bessel_model}}
\begin{tabular}{llllc}
\toprule
Type             & $\Pi\in\CCC_G(\omega)$                  & $\rho$ & $\delta_P^{-1/2}\otimes\widetilde{\Pi}\in\CCC_{TS}$ & perfect\\
\midrule
\nosf{I}         & $\chi_1\times\chi_2\rtimes\sigma$       & $\rho\in\Delta_-(\Pi)$               & $\mathbb{E}[X\to X/ \nu^{1/2}\rho]$&\\
                 &                                         & $\rho\in\Delta^\divideontimes_-(\Pi)$& $\mathbb{E}[X\to X/ \nu^{1/2}\rho^\divideontimes]$&\\
                 &                                         & every other $\rho$ & $\mathbb{E}[X]$ & $\bullet$\\
\nosf{IIa}       & $\chi\St_{\Gl(2)}\rtimes\sigma$                  & $\rho\in\Delta_-(\Pi)$               & $\mathbb{E}[X\to X/ \nu^{1/2}\rho]$&\\
                 &                                         & $\rho\in\Delta^\divideontimes_-(\Pi)$& $\mathbb{E}[X\to X/ \nu^{1/2}\rho^\divideontimes]$&\\
                 &                                         & every other $\rho$             & $\mathbb{E}[X]$ & $\bullet$\\
\nosf{IIb}       & $\chi\mathbf{1}_{\Gl(2)}\rtimes\sigma$  & $\chi\sigma$                       & $\mathbb{E}[X]$ & $\bullet$ \\
\nosf{IIIa}      & $\chi\rtimes\sigma\St_{\GSp(2)}$        & all                            & $\mathbb{E}[\nu^{1/2}\sigma\oplus\chi_1\nu^{1/2}\sigma]$ & $\bullet$ \\
\nosf{IIIb}      & $\chi\rtimes\sigma\mathbf{1}_{\GSp(2)}$ & $\sigma$                       & $\mathbb{E}[\nu^{-1/2}\sigma\oplus\chi_1\nu^{-1/2}\sigma]$ & $\bullet$ \\
\nosf{IVa}       & $\sigma\St_G$                           & all $\rho$                     & $\mathbb{E}[\nu^{3/2}\sigma]$ & $\bullet$ \\
\nosf{IVb}       & $L(\nu^2,\nu^{-1} \St)$                 & $\rho=\sigma$                  & $\mathbb{E}[\nu^{3/2}\sigma\oplus\nu^{-1/2}\sigma]$ & $\bullet$ \\
\nosf{IVc}       & $L(\nu^{3/2} \St,\nu^{-3/2}\sigma)$     & $\rho=\nu^{\pm1}\sigma$        & $\mathbb{E}[\nu^{-3/2}\sigma\oplus\nu^{1/2}\sigma]$ & $\bullet$ \\
\nosf{Va}        & $\delta([\xi,\nu\xi],\nu^{-1/2})$       & $\sigma$             & $\xi\nu^{1/2}\sigma\oplus\mathbb{E}[\nu^{1/2}\sigma]$     & \\
                 &                                         & $\xi\sigma$& $\xi\nu^{1/2}\sigma\oplus\mathbb{E}[ \nu^{1/2}\sigma]$&\\
                 &                                         & every other $\rho$             & $\mathbb{E}[\nu^{1/2}\sigma\oplus\xi\nu^{1/2}\sigma]$   &$\bullet$ \\
\nosf{Vb}        & $L(\nu^{1/2}\xi \St, \nu^{-1/2}\sigma)$ &    $\sigma$                    & $\mathbb{E}[\nu^{-1/2}\sigma\oplus\nu^{1/2}\xi\sigma]$ & $\bullet$ \\
\nosf{Vc}        & $L(\nu^{1/2}\xi \St, \nu^{-1/2}\xi\sigma)$ & $\xi\sigma$                 & $\mathbb{E}[\nu^{1/2}\sigma\oplus\nu^{-1/2}\xi\sigma]$ & $\bullet$ \\
\nosf{VIa}       & $\tau(S,\nu^{-1/2}\sigma)$              & $\sigma$                       & $\mathbb{E}[\nu^{1/2}\sigma]\oplus\nu^{1/2}\sigma$   & \\
                 &                                         & every other $\rho$             & $\mathbb{E}[(\nu^{1/2}\sigma)^{(2)}]$ & $\bullet$ \\
\nosf{VIc}       & $L(\nu^{1/2} \St,\nu^{-1/2}\sigma)$     & $\sigma$                       & $\mathbb{E}[\nu^{-1/2}\sigma]$ & $\bullet$ \\
\nosf{VId}       & $L(\nu,1\rtimes\nu^{-1/2}\sigma)$       & $\sigma$                       & $\mathbb{E}[(\nu^{-1/2}\sigma)^{(2)}]$ & $\bullet$ \\
\nosf{VII}       & $\chi\rtimes\pi$                        & all                            & $\mathbb{S}$  & $\bullet$\\
\nosf{VIIIa}     & $\tau(S,\pi)$                           & all                            & $\mathbb{S}$  & $\bullet$\\
\nosf{IXa}       & $\delta(\nu\xi,\nu^{-1/2}\pi)$          & all                            & $\mathbb{S}$  & $\bullet$\\
\nosf{X}         & $\pi\rtimes\sigma $                     & $\rho\in\Delta_-(\Pi)$               & $\mathbb{E}[X\to X/ \nu^{1/2}\rho]$&\\
                 &                                         & $\rho\in\Delta^\divideontimes_-(\Pi)$& $\mathbb{E}[X\to X/ \nu^{1/2}\rho^\divideontimes]$&\\
                 &                                         & every other $\rho$ & $\mathbb{E}[X]$ & $\bullet$\\
\nosf{XIa}       & $\delta(\nu^{1/2}\pi,\nu^{-1/2}\sigma)$ & $\sigma$                        & $\mathbb{S}\oplus\nu^{1/2}\sigma$ \\
                 &                                         & every other $\rho$                   & $\mathbb{E}[\nu^{1/2}\sigma]$  & $\bullet$\\
\nosf{XIb}       & $L(\nu^{1/2}\pi, \nu^{-1/2}\sigma)$     & $\sigma$                        & $\mathbb{E}[\nu^{-1/2}]$ & $\bullet$\\
                 & cuspidal generic                        & all                             & $\mathbb{S}$    & $\bullet$\\
\bottomrule
\end{tabular}
\end{center}
\end{small}
These are the normalized Bessel modules $\delta_P^{-1/2}\otimes\widetilde\Pi=\delta_P^{-1/2}\otimes\beta_\rho(\Pi)$ of degree one.
For typographical reasons, $X$ denotes $\delta_P^{-1/2}\otimes\pi_0(\widetilde{\Pi})\cong\delta_P^{-1/2}\otimes(\Pi_P)_{\widetilde{T},\Lambda}$. Except for the case where $\Pi$ is of type \nosf{VIa} and $\rho=\sigma$, this $X$ is the unique cyclic $\Gl(1)$-module with constituents in $\widetilde{\Delta}(\Pi)$ by theorem~\ref{thm:monodromy}.
The last column indicates by a bullet if $\widetilde{\Pi}$ is perfect.
For every pair $(\Pi,\rho)$ not in this list, the Bessel module has degree zero and occurs in table~\ref{tab:Bessel_module_no_split_Bessel_model}.
The notation follows \cite{Sally_Tadic} and \cite{Roberts-Schmidt}.
\end{table}

\begin{table}
\begin{small}
\begin{center}
\caption{Normalized Bessel modules of degree zero. \label{tab:Bessel_module_no_split_Bessel_model}}
\begin{tabular}{llll}
\toprule
Type        & $\Pi\in\CCC_G(\omega)$                  & $\rho$    &  $\delta_P^{-1/2}\otimes\widetilde{\Pi}\in\CCC_{TS}$\\
\midrule
\nosf{IIb}  & $(\chi_1\circ\det)\rtimes\sigma$        & all $\neq\chi\sigma$     &  $\nu^{-1/2}\sigma\chi_1$ \\
\nosf{IIIb} & $\chi\rtimes\sigma\mathbf{1}_{\GSp(2)}$ & all $\neq\sigma,\chi\sigma$ & $\nu^{-1/2}\sigma\oplus\chi_1\nu^{-1/2}\sigma$ \\
\nosf{IVb}  & $L(\nu^2,\nu^{-1} \St)$                 & all $\neq\sigma$            & $\nu^{-1/2}\sigma$ \\
\nosf{IVc}  & $L(\nu^{3/2} \St,\nu^{-3/2}\sigma)$     & all $\neq\nu^{\pm1}\sigma$  & $\nu^{-3/2}\sigma\oplus\nu^{1/2}\sigma$\\
\nosf{IVd}  & $\sigma\mathbf{1}_G$                    & $\sigma$                    & $\nu^{-3/2}\sigma$ \\
            &                                         & all $\neq\sigma$            & $0$\\
\nosf{Vb}   & $L(\nu^{1/2}\xi \St, \nu^{-1/2}\sigma)$ & all $\neq\sigma$            & $\nu^{-1/2}\sigma$ \\
\nosf{Vc}   & $L(\nu^{1/2}\xi \St, \nu^{-1/2}\xi\sigma)$ & all $\neq\xi\sigma$      & $\nu^{-1/2}\xi\sigma$ \\
\nosf{Vd}   & $L(\nu\xi,\xi\rtimes\nu^{-1/2}\sigma)$  & $\sigma$                    & $\nu^{-1/2}\xi\sigma$ \\
            &                                         & $\xi\sigma$                 & $\nu^{-1/2}\sigma$ \\
            &                                         & all $\neq\sigma,\xi\sigma$  & $0$\\
\nosf{VIb}  & $\tau(T,\nu^{-1/2}\sigma)$              & $\sigma$                    & $\nu^{1/2}\sigma$  \\
            &                                         & all $\neq\sigma$            & $0$                \\
\nosf{VIc}  & $L(\nu^{1/2} \St,\nu^{-1/2}\sigma)$     & all $\neq\sigma$            & $\nu^{-1/2}\sigma$ \\
\nosf{VId}  & $L(\nu,1\rtimes\nu^{-1/2}\sigma)$       & all $\sigma$                & $\nu^{-1/2}\sigma$ \\
\nosf{VIIIb}& $\tau(T,\pi)$                           & all                         & $0$                \\ 
\nosf{IXb}  & $L(\nu\xi,\nu^{-1/2}\pi)$               & all                         & $0$                \\
\nosf{XIb}  & $L(\nu^{1/2}\pi, \nu^{-1/2}\sigma)$     & all $\neq\sigma$            & $\nu^{-1/2}\sigma$ \\
            & cuspidal non-generic                    & all                         & $0$                \\

\bottomrule
\end{tabular}
\end{center}
\end{small}
These are the normalized Bessel modules $\delta_P^{-1/2}\otimes\widetilde\Pi=\delta_P^{-1/2}\otimes\beta_\rho(\Pi)$ of degree zero. For every pair $(\Pi,\rho)$ not in this table, the Bessel module has degree one and occurs in table~\ref{tab:Bessel_module_split_Bessel_model}.
The notation follows \cite{Sally_Tadic} and \cite{Roberts-Schmidt}.
\end{table}

\FloatBarrier

\newpage

\begin{small}
\bibliographystyle{amsalpha}

\begin{thebibliography}{ccccccc}\advance\itemsep by -.5em

\bibitem[BZ76]{Bernstein-Zelevinsky76}
I.N.~Bernshtein and A.V.~Zelevinsky,
\newblock{\textit{Representations of the group $GL(n,F)$ where $F$ is a non-archimedean local field}.}
\newblock{{Russ. Math. Surveys} 31:3:1-68 (1976).}

\bibitem[BZ77]{Bernstein-Zelevinsky77}
I.N.~Bernshtein and A.V.~Zelevinsky,
\newblock{\textit{Induced representations of reductive $\mathfrak{p}$-adic groups. I}.}
\newblock{{Annales scientifiques de l'\'{E}.N.S.}, 10(4):441-472, (1977).}

\bibitem[BW]{Borel_Wallach}
A.~Borel and N.~Wallach,
\newblock{ \textit{Continuous cohomology, discrete subgroups, and representations of reductive groups.}}
\newblock{ {{AMS}, volume 67 of {Mathematical surveys and monographs}, (2000).}}

\bibitem[B98]{Bump}
D.~Bump,
\newblock{\textit{Automorphic forms and representations},}
\newblock{{Cambridge Studies}, vol.\,55, (1998).}

\bibitem[D14]{Danisman}
Y.~Danisman,
\newblock{\textit{Regular poles for the p-adic group $\mathrm{GSp}_4$.}}
\newblock{{Turk. J. Math.}, 38:587--613, (2014).}

\bibitem[D15a]{Danisman_Annals}
Y.~Danisman,
\newblock{\textit{Local factors of non-generic supercuspidal representations of $GSp_4$}.}
\newblock{{Math.\ Ann.}, 361:1073--1121, (2015).}

\bibitem[D15b]{Danisman2}
Y.~Danisman,
\newblock \textit{Regular poles for the p-adic group $\mathrm{GSp}_4$ II.}
\newblock{{Turk. J. Math.}, 39:369--394, (2015).}

\bibitem[D17]{Danisman3}
Y.~Danisman,
\newblock {\textit{$L$-factor of irreducible $\chi_1\times\chi_2\rtimes\sigma$.}}
\newblock{{Chin.\ Ann.\ Math.} 38B(4):1019--1036, (2017).}

\bibitem[LSZ17]{LSZ_EulerGSp4}
D.~Loeffler, C.~Skinner and S.L.~Zerbes,
\newblock{\textit{Euler systems for $\mathrm{GSp}(4)$}}
\newblock{ArXiv:1706.00201. Preprint (2017).}

\bibitem[KW01]{Kiehl-Weissauer}
R.~Kiehl and R.~Weissauer,
\newblock{\textit{Weil conjectures, perverse sheaves and l'adic Fourier transform}.}
\newblock{Volume 42 of {Ergebnisse der Mathematik und ihrer Grenzgebiete}, Springer-Verlag (2001).}

\bibitem[N79]{Novod_L_factors}
M.E.~Novodvorsky,
\newblock{\textit{Automorphic L-functions for symplectic group $\mathrm{GSp}(4)$.}}
\newblock{In: Automorphic Forms, Representations, and L-functions, p.87--95, Proc. Symp. Pure Math., 33(2). AMS, (1979).}

\bibitem[PS97]{PS-L-Factor_GSp4}
I.~Piatetskii-Shapiro,
\newblock \textit{$L$-functions for $\mathrm{GSp}_4$}.
\newblock {{Pac.\,J.}, 181(3):259--275, (1997).}

\bibitem[PSS81]{Soudry_Piatetski_L_Factors}
I.~Piatetskii-Shapiro and D.~Soudry,
\newblock{\textit{{The {$L$} and {$\epsilon$} factors for {$\mathrm{GSp}(4)$}}},}
\newblock{J. Fac. Sci. Univ. Tokyo, 28:505--530, (1981).}

\bibitem[RS07]{Roberts-Schmidt}
B.~Roberts and R.~Schmidt,
\newblock {\textit{Local newforms for {$\mathrm{GSp}(4)$}}, 
\newblock{volume 1918 of {Lecture Notes in Mathematics}}. Springer, (2007).}

\bibitem[RS16]{Roberts-Schmidt_Bessel}
B.~Roberts and R.~Schmidt,
\newblock {\textit{Some results on Bessel functionals for $\mathrm{GSp}(4)$}.}
\newblock{Documenta Math. 21:467-553, (2016).}

\bibitem[RW18a]{Anisotropic_Exceptional}
M.~R\"osner and R. Weissauer,
\newblock{\textit{Exceptional poles of local spinor $L$-functions of $\mathrm{GSp}(4)$ with anisotropic Bessel models.}}
\newblock{Preprint, (2018). ArXiV 1803.07539}

\bibitem[RW18b]{Subregular}
M.~R\"osner and R. Weissauer,
\newblock{\textit{Regular poles of local $L$-functions for $\GSp(4)$ with respect to split Bessel models (the subregular cases).}}
\newblock{Preprint, (2018).}

\bibitem[ST94]{Sally_Tadic}
P.~Sally and M.~Tadi{\'c},
\newblock {\textit{Induced representations and classifications for {$\mathrm{GSp}(2,F)$} and  {$\mathrm{Sp}(2,F)$}}}.
\newblock {M{\'e}moires de la Soci{\'e}t{\'e} Math{\'e}matique de France},  52:75--133, (1994).

\bibitem[T94]{Tadic}
M.~Tadi\'{c},
\newblock{\textit{Representations of $p$-adic symplectic groups}.}
\newblock{Comp.~Math.~90:123--181, Kluwer (1994).}

\bibitem[T00]{Takloo-Bighash}
R.~Takloo-Bighash,
\newblock{\textit{$L$-functions for the $p$-adic group $\mathrm{GSp}(4)$}.}
\newblock{Amer. J. Math., 122:1085--1120, (2000).}

\bibitem[W17]{W_Excep}
R.~Weissauer,
\newblock {\textit{Exceptional poles of local $L$-functions for $\GSp(4)$ with respect to split Bessel models}.}
\newblock{ArXiv:1712.06370v1 [math.RT] (2017).}

\end{thebibliography}

\end{small}

\vskip 15 pt
\begin{footnotesize}
\centering{Mirko R\"osner\\ Mathematisches Institut, Universit\"at Heidelberg\\ Im Neuenheimer Feld 205, 69120 Heidelberg\\ email: mroesner@mathi.uni-heidelberg.de}

\vskip 10 pt
\centering{Rainer Weissauer\\ Mathematisches Institut, Universit\"at Heidelberg\\ Im Neuenheimer Feld 205, 69120 Heidelberg\\ email: weissauer@mathi.uni-heidelberg.de}

\end{footnotesize}
\end{document}